\documentclass[11pt]{amsart}
\usepackage{amsmath,amssymb,mathrsfs, xcolor, amsthm}
\usepackage{graphicx} % Required for inserting images
\usepackage{float} % in the preamble
\usepackage[title]{appendix}
\usepackage{geometry}
\usepackage{comment}
\usepackage{hyperref}
\usepackage{enumitem}
\usepackage{mathtools}
\usepackage{tikz}
\usetikzlibrary{arrows.meta}
\usetikzlibrary{calc}
\newcommand{\ChuProp}{\cite[Proposition~10.4]{chu2024existence}}
\newcommand{\bB}{\mathbf{B}}
\newcommand{\bC}{\mathbf{C}}
\newcommand{\bF}{\mathbf{F}}
\newcommand{\bL}{\mathbf{L}}
\newcommand{\bM}{\mathbf{M}}
\newcommand{\fb}{\mathfrak{b}}
\newcommand{\fg}{\mathfrak{g}}
\DeclareMathOperator{\arccot}{arccot} 
\newcommand{\RN}[1]{%
  \textup{\uppercase\expandafter{\romannumeral#1}}%
}
 
\newtheorem{thm}{Theorem}[section]
\newtheorem*{thm*}{Theorem}

\newtheorem{lem}[thm]{Lemma}
\newtheorem{prop}[thm]{Proposition}

\newtheorem{clm}[thm]{Claim}
\newtheorem{conj}[thm]{Conjecture}

\theoremstyle{definition}
\newtheorem{defn}[thm]{Definition}

\newtheorem{qe}[thm]{Question}

\theoremstyle{remark}
\newtheorem{rem}[thm]{Remark}
\newtheorem{notation}[thm]{Notation}

\title{Free boundary minimal annuli in convex balls}
\author{Dongyeong Ko and Guanhua Shao}
\address{Department of Mathematics, Massachusetts Institute of Technology, Cambridge, MA 02139}
\email{dyko@mit.edu}
\address{Department of Mathematics, Rutgers University -- New Brunswick, Piscataway, NJ 08854}
\email{gs977@math.rutgers.edu}
\date{July 2026}
\begin{document}

\maketitle
\begin{abstract}
We construct a $4$-parameter family of properly embedded genus-zero surfaces with at most two boundary components in a Riemannian $3$-ball. The construction is a free boundary analog of the canonical $5$-parameter family of surfaces on the $3$-sphere by Marques-Neves in the proof of  Willmore conjecture. In the Euclidean unit ball, this family realizes the critical catenoid as a Simon-Smith min-max limit.  As applications, we give an upper bound of Almgren-Pitts $4$-width $\omega_4(\mathbb{B}^3)$ by the area of the critical catenoid. Moreover, we prove the existence of at least three free boundary minimal annuli in any compact $3$-manifold with nonnegative Ricci curvature and strictly convex boundary. 
\end{abstract}

\section{Introduction}
\begin{comment}
Free boundary minimal submanifolds on manifolds with boundary have been extensively studied. They are free boundary analogs of closed minimal submanifolds in closed manifolds.

The round $3$-sphere is one of the simplest and most intriguing closed manifolds. Almgren \cite{almgren1966some} proved that any immersed minimal $2$-sphere in $\mathbb{S}^{3}$ is totally geodesic, hence it is congruent to an equatorial sphere. For genus one surfaces, Brendle \cite{Brendle2012} showed that any embedded minimal torus in $\mathbb{S}^{3}$ is congruent to the Clifford torus $1/ \sqrt{2} \,\mathbb{S}^{1} \times 1/ \sqrt{2} \,\mathbb{S}^{1}$, which was conjectured by Lawson \cite{lawson1970unknottedness}. Urbano \cite{urbano1990minimal} showed that Clifford torus is the unique minimal surface of Morse index $5$ in the round $3$-sphere up to congruence. In their  proof of Willmore conjecture \cite{willmore1965note}, Marques-Neves \cite{marques2014} showed that Clifford torus is the least-area closed embedded minimal surface in $\mathbb{S}^3$ other than an equator sphere.
    
\end{comment}
Free boundary minimal surfaces in the unit ball $\mathbb{B}^{3}$ are natural analogs of minimal surfaces on the round $3$-sphere $\mathbb{S}^{3}$. Nitsche \cite{nitsche1985stationary} proved that any immersed free boundary minimal disk in $\mathbb{B}^{3}$ is an equatorial disk. However, the characterization of free boundary minimal annuli in the unit ball still remains open. The critical catenoid is conjectured to be the unique properly embedded free boundary minimal annulus in $\mathbb{B}^3$ up to congruence \cite[Conjecture 1.1]{FraserLi2014}.

Although the Morse index of the critical catenoid was shown to be $4$ in \cite{tran2016index,devyver2019index, smith2019morse, Medvedev2023}, whether the critical catenoid is the only free boundary minimal surface of index $4$ is still unknown \cite[Open question 6]{LiMartin2020}. As a natural analog of the Willmore conjecture \cite{marques2014}, whether the critical catenoid is the free boundary minimal surface of smallest area other than the equator disks is also an open question \cite[Open question 10]{LiMartin2020}. 

One of the key steps in the proof of Willmore conjecture by Marques-Neves \cite{marques2014} is that to each closed surface in $\mathbb{S}^3$, they construct a $5$-parameter canonical family of surfaces. Then they apply the Almgren-Pitts min-max theory to this family and show that the min-max width is achieved by the Clifford torus if the associated surface has positive genus. Nevertheless, to the authors' knowledge, no natural $4$-parameter family of surfaces whose min-max width is realized by the critical catenoid has been known. 

In this paper, we construct a $4$-parameter family of properly embedded surfaces in $\mathbb{B}^3$ with a topological control, which is a free boundary analog of the canonical $5$-parameter family of Marques-Neves \cite{marques2014}. \begin{comment}
More specifically, we construct a  $4$-parameter free boundary Simon-Smith family (See Definition \ref{def:fbssfamily}) of properly embedded annuli (with its degenerations) in the $3$-ball, which can be parameterized by $\mathbb{RP}^4$.    
\end{comment} 
\begin{thm}\label{thm:introgenuszerotwoboundary4parameter}
        There exists a free boundary Simon-Smith $4$-parameter family
        \[
             \Phi_4:\mathbb{RP}^{4} \rightarrow \mathcal{S}^*(\mathbb{B}^{3}),
        \]
         where $\Phi_4(x)$ satisfies the topological bound $(0,1)$ for $x \in \mathbb{RP}^{4}$.
    \end{thm}
In Theorem \ref{thm:introgenuszerotwoboundary4parameter}, we define the genus and boundary complexity of a surface to be the sum, over all of its connected components, of the genus and the number of boundary components minus one, respectively, as in Franz-Schulz \cite[Definition 1.6]{franz2023topological}. Then we say that a surface $\Sigma$ satisfies the topological bound $(\fg_{0},\fb_{0})$ if the genus and the sum of the genus and the boundary complexity are less than or equal to $\fg_{0}$ and $\fg_{0}+\fb_{0}$, respectively (see \S \ref{subsec: Simon-Smith min-max theory}). We denote the set of punctate surfaces by $\mathcal{S}^{*}(M)$ which is a generalized class of the set of properly embedded surfaces in $M$ (see Definition \ref{def: Punctate surfaces}). 
    
    Theorem \ref{thm:introgenuszerotwoboundary4parameter} may be viewed as a first step toward constructing higher parameter families of surfaces with prescribed genus and boundary complexity which are nontrivial (See Ketover \cite[\S 1.2]{ketover2022flipping} for the remark on the closed setting). More broadly, the $4$-parameter family gives a framework for a program toward constructing free boundary minimal surfaces with prescribed topology using a variational method. 
    
    After running Simon-Smith min-max construction to the $4$-parameter family in Theorem \ref{thm:introgenuszerotwoboundary4parameter}, we prove that this family achieves the critical catenoid $\mathbb{K}$ as a Simon-Smith min-max limit:

 \begin{thm} \label{intro thm: width is the critical catenoid}
    The Simon-Smith min-max width within the homotopy class $\Lambda(\Psi)$ is
    \[
    \bL _{SS}(\Lambda(\Phi_4)) = \text{Area}(\mathbb{K}).
    \]
    In particular, the min-max limit surface is obtained by a multiplicity $1$ critical catenoid.
    \end{thm}

The $p$-width $\{ \omega_{p} \}$ of a compact manifold $(M,g)$ is a nonlinear analog of the spectrum of its Laplace–Beltrami operator, which were introduced by Gromov \cite{gromov2002isoperimetry, gromov2006dimension, gromov2007singularities} and further developed by Guth \cite{guth2009minimax}. The $p$-widths have played a crucial role through the min-max theory to enhance our understanding of minimal hypersurfaces in ambient manifolds. For instance, the existence of infinitely many closed and embedded minimal hypersurfaces on every smooth manifold $(M^{n+1},g)$ ($3 \le n+1 \le 7$), which is one of Yau's conjectures in \cite{yau1982problem}, was proven in a series of papers by Marques-Neves \cite{marques2017existence}, Song \cite{song2023existence}, and Zhou \cite{zhou2020multiplicity} (See also Chodosh-Mantoulidis \cite{chodosh2020minimal} and Chu-Stern \cite{chu2025minimals}). In the free boundary setting, Li \cite{li2015general}, Li-Zhou \cite{li2021min}, and Wang \cite{wang2024existence} developed the Almgren-Pitts min-max theory for free boundary minimal hypersurfaces. The explicit values of the $p$-widths have been known for several manifolds, like spheres and balls; see \cite{chodosh2023p,chu2023free, chu2023strong, donato2020first, donato2024first, marx2026p, nurser2016low}. 

In particular, on the round $3$-sphere, it is now known that $\omega_{5}(\mathbb{S}^{3}) = 2 \pi^{2}$ by the resolution of the Willmore conjecture by Marques-Neves \cite{marques2014} and Nurser's refinement \cite{nurser2016low}, and we have more information on the next widths. However, to the authors' knowledge, the only known result on the widths of the unit ball after disks is that $\omega_{6}(\mathbb{B}^{3})< 2 \pi$ by Chu \cite{chu2023free}, and it is conjectured that the 4-width is detected by the critical catenoid in the same paper (See \S 1.1 in \cite{chu2023free}).  

Since our Simon-Smith family and its deformation by isotopies within a homotopy class is an Almgren-Pitts $4$-sweepout (Proposition \ref{prop:sslargerap}), we have that the $4$-width of the unit ball in Almgren-Pitts min-max is smaller than or equal to the Simon-Smith width of $\Lambda(\Phi_4)$. Combining this with Proposition \ref{prop:lowerfourthbound} and Theorem \ref{intro thm: width is the critical catenoid}, we obtain the following bounds on the $4$-width of the unit ball.
\begin{thm} \label{intro thm:upperbound4width} The fourth width in Almgren-Pitts volume spectrum of the unit ball satisfies the following
\[
\pi<\omega_{4}(\mathbb{B}^{3}) \le \text{Area}(\mathbb{K}).
\]
\end{thm}
\begin{comment}
We are not able to rely on the Willmore energy to bound area within our $4$-parameter family as in the sphere case. The areas of the $5$-parameter canonical family of Marques-Neves are bounded above by its Willmore energy (See Ros \cite{ros1999willmore}). However, the area functional and Willmore energy are not compatible for general properly embedded surfaces in the unit ball (see \S \ref{intro: Willmore energy}). Moreover, Marques-Neves \cite{marques2014} apply Urbano's classification of Morse index $5$ closed and embedded minimal surfaces on the round $3$-sphere \cite{urbano1990minimal}. However, as mentioned above, the classification of index 4 free boundary minimal surfaces remains open.  
\end{comment}

As another application of Theorem \ref{thm:introgenuszerotwoboundary4parameter}, we show the existence of three free boundary minimal annuli in general strictly convex balls. 
  
  The simplest form of a minimal submanifold with controlled topology is an embedded geodesic. It is proven that a closed Riemannian $2$-sphere $(S^{2},g)$ admits at least three simple closed geodesics by Lyusternik-Schnirelmann \cite{lyusternik1947topological} and Grayson \cite{grayson1989shortening} which is conjectured by Poincaré \cite{poincare1905lignes}, based on Birkhoff's min-max theory \cite{birkhoff1917dynamical}. 

In one dimension higher, Yau conjectured the existence of at least four embedded minimal spheres in every Riemannian $3$-sphere in \cite{yau1982problem}. Simon-Smith \cite{smith1983existence} proved that there always exists at least one minimal sphere in any Riemannian $3$-sphere by developing the min-max theory with topological control (See also the survey of Colding-De Lellis \cite{colding2003min}). Wang and Zhou recently confirmed the conjecture in spheres with Ricci positive metrics or bumpy metrics \cite{wang2023existence} (See also Haslhofer-Ketover \cite{Haslhofer2019} and Sarnataro-Stryker \cite{sarnataro2023optimal}). More recently, Wang and Zhou proved the existence of two minimal spheres in $3$-spheres with arbitrary metric in \cite{wang2026existence}. 

For minimal tori on Riemannian $3$-spheres, White \cite{white1989every, white1991space} conjectured that every Riemannian $3$-sphere contains at least $5$ embedded minimal tori. He proved that Riemannian $3$-spheres with positive Ricci curvature contain at least one minimal torus. Chu-Li \cite{chu2024existence} recently confirmed this conjecture for Riemannian $3$-spheres with positive Ricci curvature by developing the repetitive min-max theory. Based on this work, there have been a series of papers on embedded minimal surfaces with prescribed genus. See \cite{chu2025minimal, chu2025min, chu2026enumerative}.

As a free boundary analog of Yau's conjecture on minimal spheres, it has been conjectured that there exist at least three embedded free boundary minimal disks in any strictly convex $3$-ball with nonnegative Ricci curvature. Grüter-Jost \cite{gruter1986embedded} proved the existence of an embedded free boundary disk (See also Struwe \cite{struwe1984free}, Jost \cite{jost1986existence} and Li \cite{li2015general}). Haslhofer-Ketover \cite{haslhofer2025free} recently proved the existence of two free boundary minimal disks in compact three-manifolds with strictly convex boundary and nonnegative Ricci curvature, and confirmed the conjecture in generic metric. More recently, Sarnataro-Stryker-Wang-Zhou \cite{sarnataro2026existence} proved the existence of three minimal disks in general metrics.

The problem of the existence of free boundary minimal surfaces of prescribed topology remains open in general. Since surfaces with boundary which have the next first Betti number to disks are annuli, the existence and enumerative problem of free boundary minimal annuli is a natural free boundary counterpart of White's conjecture on minimal tori in $3$-spheres. Maximo-Nunes-Smith \cite{maximo2017free} proved the existence of one free boundary minimal annulus in compact three-manifolds with strictly convex boundary and nonnegative Ricci curvature by developing degree theory techniques (See also Aché-Maximo-Wu \cite{ache2016metrics}). For minimal annuli, in the case of Euclidean ellipsoids, \cite{schulz2024equivariant} constructed three distinct $D_2$-equivariant free boundary minimal annuli. lFor more works on the construction of minimal submanifolds of controlled topological types, see also \cite{ko2024morse,ko2025regularity,ko2026existence,li2025minimal,li2024existence,wang2026embedded,zhou2016free}.

We now introduce our result on the min-max construction of three free boundary minimal annuli in a compact $3$-manifold with strictly convex boundary and nonnegative Ricci curvature. By changing the rotational symmetry axis of our initial annulus, we obtain the $6$-parameter version of Theorem \ref{thm:introgenuszerotwoboundary4parameter} (Theorem \ref{thm:genuszerotwoboundary}).
\begin{thm}\label{thm:introgenuszerotwoboundary6parameter}
        There exists a $6$-parameter family   
        \[
            \Phi_6:Y \rightarrow \mathcal{S}^*(\mathbb{B}^{3})
        \]
         that is a free boundary Simon-Smith family parametrized by $Y:=\mathbb{RP}^{4} \times \mathbb{RP}^{2}$. Moreover, $\Phi_{6}$ satisfies the topological bound $(0,1)$.
    \end{thm}
Then by developing the free boundary version of repetitive min-max construction by Chu-Li \cite{chu2024existence}, we obtain the second main result. 
\begin{thm}\label{thm:threeminimalannuli} Every compact $3$-manifold with nonnegative Ricci curvature and strictly convex boundary contains at least 3 embedded free boundary minimal annuli.
\end{thm}
Theorem \ref{thm:threeminimalannuli} is the first construction of multiple free boundary minimal surfaces with prescribed topology in arbitrary compact three manifolds with nonnegative Ricci curvature and strictly convex boundary, after minimal disks by Haslhofer-Ketover \cite{haslhofer2025free} and Sarnataro-Stryker-Wang-Zhou \cite{sarnataro2026existence}. Since an annulus has genus zero and two boundary components, this gives an idea of the contribution of the number of boundary components to the topological property of moduli space.

Let us give the heuristic behind the number three in Theorem \ref{thm:threeminimalannuli}. Denote by $\mathcal A$ the space of (unknotted) properly embedded annuli in the standard three-ball. Hatcher's theorem of Smale's conjecture \cite{hatcher1983proof} leads one to expect that $\mathcal A\simeq \mathbb{RP}^{2}$. Thus, Lusternik–Schnirelmann theory predicts at least $\operatorname{cat}(\mathbb{RP}^{2})=3$ critical points of the area functional, corresponding to three free boundary minimal annuli. The main difficulty in transforming this heuristic into the existence result is that the space of annuli is not compact, for instance, annuli may degenerate into disks, and different min--max classes may detect the same geometric annulus, possibly with different multiplicities.

%\subsection{Part \uppercase\expandafter{\romannumeral 1}}
\subsection{Main ideas} We outline the proof of our main results.
\subsubsection{Canonical family}
Inspired by Marques-Neves \cite{marques2014}, to each properly embedded surface $\Sigma$ that intersects $\partial \mathbb{B}^3$ orthogonally, we associate a $4$-parameter canonical family of surfaces $\Phi_4^{\Sigma}$.

For every $v \in \mathbb{B}^{3}$, we consider the conformal map $F_v : \mathbb{B}^3\to \mathbb{B}^3$ by
 \begin{equation}
   F_v(x) = \frac{(1 - |v|^2)x - (1 - 2 \langle v, x \rangle + |x|^2) v}{1 - 2 \langle v, x \rangle + |x|^2 |v|^2},\quad v \in\mathbb{B}^3.
 \end{equation}
$F_v$ arises as the hyperbolic translation in hyperbolic geometry and every conformal map in $\mathbb{B}^3$ differs from some $F_v$ by an orthogonal linear map $O \in O(3)$ (see \cite[\S 4.5]{John2019}). Using the inverse stereographic projection $\pi: \mathbb{R}^3 \cup \{ \infty \} \to \mathbb{S}^3$ (see \S \ref{subsec: Canonical family}), we define the canonical $4$-dimensional family of surfaces:
\[
\Sigma_{(v, t)} = \partial \left(\pi^{-1}(\{ x \in \mathbb{S}_+^3: d_v(x) < t \}) \right) \cap \mathbb{B}^3, \quad (v, t) \in \mathbb{B}^3 \times [-\pi, \pi].
\]
Here $d_v: \overline{\mathbb{S}_+^3} \to \mathbb{R}$ is the signed distance function to the oriented surface $\pi(\Sigma_v) = \pi(F_v(\Sigma)) \subset \overline{\mathbb{S}_+^3}$, which becomes well-defined after we choose a unit normal vector field $N$ on $\Sigma$. The distance is computed with respect to the standard metric of the upper hemisphere $\mathbb{S}_+^3 \subset \mathbb{S}^3$. 

\subsubsection{Boundary blow-up} \label{subsubsec: Boundary blow-up} The $4$-dimensional family $\{\Sigma_{(v, t)}\}_{(v, t) \in \mathbb{B}^3 \times [-\pi, \pi]}$ is not continuous in any reasonable sense if we want to extend it to $\overline{\mathbb{B}^3} \times [- \pi, \pi]$. Since $\Sigma$ intersects $\partial \mathbb{B}^3$ orthogonally, the tangent plane $T_p\Sigma$ passes through the origin for any $p \in \partial \Sigma$. Thus, as $v \in \mathbb{B}^3$ converges to  $p \in \Sigma$, we are essentially blowing up the equator disk $D_p: = \{ x \in \mathbb{B}^3 :\langle x, N(p) \rangle = 0 \}$ and the limit depends on the angle of convergence. More precisely, if
 \( v_n = |v_n|(\cos{(s_n)} p + \sin{(s_n) N(p)})\)
is a sequence in $\mathbb{B}^3$ converging to $p \in \partial \Sigma$, i.e., $|v_n|$ tends to one, $|v_n| < 1$, and $s_n$ tends to zero, then the limit of $\Sigma_{(v_n, t)}$ will be a spherical cap that intersects $\partial \mathbb{B}^3$ orthogonally
\[
\begin{split}
\lim_{n \to \infty}\Sigma_{(v_n, t)}  = \lim_{n \to \infty}(D_p)_{(v_n, t)} 
 = \partial \left(\pi^{-1}(B^+_{\frac{\pi}{2} - \theta + t}(- \sin{\theta} \, p - \cos{\theta} \, N(p), 0)) \right) \cap \mathbb{B}^3 \\
\end{split}
\]
where
\(\theta = \lim_{n \to \infty} \arctan{\frac{s_n}{1 - |v_n|}} \in \left[- \frac{\pi}{2}, \frac{\pi}{2}\right]\)
(see \S \ref{subsec: general free  boundary surface notation} for notation).
\begin{rem}
As $v \in \mathbb{B}^3$ converges to $p \in \mathbb{S}^2 \setminus \partial \Sigma$, $\Sigma_{(v, t)}$ converges to spherical caps
\[
\partial \left(\pi^{-1} \left(B^+_{\pi + t}(p, 0) \right) \right) \cap \mathbb{B}^3 \quad \text{or} \quad \partial \left(\pi^{-1} \left(B^+_{t}(-p, 0) \right) \right) \cap \mathbb{B}^3,
\]
depending on which connected component of $\mathbb{S}^2 \setminus \partial \Sigma$ contains $p$.
\end{rem}

We reparametrize the canonical family to make it continuous on $\overline{\mathbb{B}^3} \times [- \pi, \pi]$. This is done by ``blowing-up" $\overline{\mathbb{B}^3}$ along $\partial \Sigma$ as in \cite[\S 2.5]{marques2014} (see Section \ref{sec:canonicalfamily} for details).

\subsubsection{Construction of $4$-sweepout} 
Consider the annulus $C:=\pi^{-1}(T^{2}) \cap \mathbb{B}^{3}$, where  $T^{2} \subset \mathbb{S}^{3}$ is the Clifford torus. Due to the antipodal symmetry of $C$, $\Phi_4^C$ can be easily reparametrized into a $4$-sweepout using boundary identifications. 

Although any annulus with antipodal symmetry works for the boundary identification, we choose $C$ over other annuli because the topology of the $4$-parameter family $\Phi_4^C$ is well-controlled. $C$ inherits a two-sided channel surface structure from the Clifford torus, and this structure is preserved under conformal deformations. Moreover, by explicitly computing the radii of the channel (hemi-)spheres, we shall see $\Phi_4^C$ is a free boundary Simon-Smith family of surfaces with the topological bound $(0, 1)$. We can then find the min-max limit of $\Phi_4^C$ using the topological control for Simon-Smith  families \cite{franz2023topological}.

On the other hand, by a Lyusternik-Schnirelmann type argument (see Proposition \ref{prop:lowerfourthbound}), we obtain a lower bound on the $4$-width of the unit ball $ \omega_{4}(\mathbb{B}^{3}) > \pi.$ Hence, the Simon-Smith min-max width $\mathbf{L}_{SS}(\Lambda(\Phi_4^C))$ satisfies
\begin{equation} \label{main idea eq: lower bound on min-max width}
\pi < \omega_4(\mathbb{B}^3)  \leq \mathbf{L}_{SS}(\Phi_4^C).    
\end{equation}
\subsubsection{Simon-Smith family of surfaces} \label{subsubsec: Simon-Smith family of annuli} 

Applying the Simon-Smith min-max theorem to $\Phi_4^C$, by the topological control \cite[Theorem 1.8]{franz2023topological} and the Multiplicity One Theorem \cite[Theorem 1.3]{sarnataro2026existence}, we obtain that $\mathbf{L}_{SS}(\Phi_4^C)$ is achieved by a properly embedded free boundary minimal annulus $\Gamma$ with
\begin{equation} \label{main idea eq: Gamma}
\mathbf{L}_{SS}(\Phi_4^C) = \operatorname{Area}(\Gamma).    
\end{equation}
By the upper bound on Morse index for min-max surfaces \cite[Theorem 1.10]{franz2023equivariant} and the lower bound of Morse index of non-flat free boundary minimal surfaces from \cite[Theorem 1.1]{devyver2019index}, $\operatorname{Index}(\Gamma) = 4$. Now by the characterization of free boundary minimal annuli by Morse index \cite[Corollary 7.3]{devyver2019index} (also see \cite[Corollary 3.11]{tran2016index}), $\Gamma$ must be the critical catenoid $\mathbb{K}$. Thus, combining (\ref{main idea eq: lower bound on min-max width}) and (\ref{main idea eq: Gamma}), Theorem \ref{intro thm:upperbound4width} holds.

%\subsection{Part \uppercase\expandafter{\romannumeral 2}}
\subsubsection{$6$-parameter family of surfaces} By changing the rotational symmetry axis of $C$, we can obtain a $\mathbb{RP}^2$ family of such annuli congruent to $C$. For each annulus in this $\mathbb{RP}^2$ family, we can associate the $4$-sweepout we constructed above in \S \ref{subsubsec: Simon-Smith family of annuli} and obtain a $6$-parameter family of surfaces $\Phi_6$, whose parameter space is $Y : = \mathbb{RP}^4 \times \mathbb{RP}^2$ (see Section \ref{sec: 6-parameter family}).

\subsubsection{Existence of three free boundary minimal annuli} On a compact $3$-manifold $(M, g)$ with nonnegative Ricci and strictly convex boundary, we adapt the scheme developed by Chu-Li \cite[Section 3]{chu2024existence} to the $6$-parameter family $\Psi \coloneqq \Phi_6$ to show the existence of $3$ minimal annuli in $(M, g)$. We summarize the main steps below.
\begin{enumerate} [label=\textbf{Step \arabic*.}]
 
\item Without loss of generality, we can assume there are only finitely many free boundary minimal annuli in $(M, g)$. By slightly perturbing the metric, we may assume $(M, g)$ still has the same number of minimal annuli but only finitely many minimal disks.

\item \label{step: Repetitive min-max machinery}  Repetitive min-max: We apply the Simon-Smith min-max theorem (Theorem \ref{thm:simon-smith min-max}) to $\Psi$ to obtain an unstable multiplicity one minimal disk or annulus, by the multiplicity one Theorem \cite[Theorem 1.3]{sarnataro2026existence} and the topological control \cite[Theorem 1.8]{franz2023topological}. We remove a ``cap" near the free boundary minimal surface from the parameter space $Y$, obtaining a new family whose maximal area is strictly smaller than this free boundary minimal surface. The iteration will stop after finitely many iterations as there are only finitely many free boundary minimal disks and annuli in $(M, g)$. 

\item  Assemble pieces we obtained in the previous step to get a $6$-parameter family $\Xi$ which is homotopic to $\Psi$.

\item  Let $N$ be the number of minimal annuli detected in \ref{step: Repetitive min-max machinery}. Then the parameter space $\text{dmn}(\Xi)$ of $\Xi$ can be decomposed into a union of subsets $D_0,D_1,\cdots,D_N$, where the restriction $\Xi|_{D_0}$ consists entirely of disks (possibly with singularities), and cannot form a $4$-sweepout. For each $j=1,\cdots,N$, $D_j$ consists of surfaces close to some multiplicity one free boundary minimal annulus. $D_j$ is homologically trivial in the ambient parameter space. 
\item  Let $\lambda \coloneqq [\Psi]^*(\overline{\lambda}) \in H^1(Y; \mathbb{Z}_2)$ induced by the generator $\bar{\lambda} \in H^*(\mathcal{Z}_2(\mathbb{B}^3, \partial \mathbb{B}^3; \mathbb{Z}_2) ; \mathbb{Z}_2)$. We show there exists $\alpha \in H^1(Y ; \mathbb{Z}_2)$ such that $\lambda^4 \cup \alpha^2 \neq 0.$     
\item  If $N \leq 2$, by the Lyusternik–Schnirelmann vanishing lemma \cite[Lemma 5.3]{Haslhofer2019}, $\lambda^4 \cup \alpha^2 \neq 0$ cannot hold. Thus, there are at least $3$ free boundary minimal annuli in $(M, g)$.
\end{enumerate}

\subsection{Discussion}
\subsubsection{Lower bound on the $4$-width of the unit ball}
As in Chu \cite{chu2023free}, it is conjectured that $\omega_{4}(\mathbb{B}^{3}) = \text{Area}(\mathbb{K})$. While we could obtain that $\pi< \omega_{4}(\mathbb{B}^{3}) \le \text{Area}(\mathbb{K})$ in Theorem \ref{intro thm:upperbound4width}, we cannot exclude the possibility that $\omega_{4}(\mathbb{B}^{3})< \text{Area}(\mathbb{K})$.

The conjectured equality would follow immediately from either a uniqueness theorem or a lower bound of areas of free boundary minimal surfaces. However, neither statement is currently known.

Nevertheless, the index estimates available for min-max surfaces reduce the problem to a seemingly weaker question. Indeed, a min-max surface realizing $\omega_{4}(\mathbb{B}^{3})$ has Morse index at most $4$ by the index upper bound of Marques--Neves \cite{marques2016morse}. On the other hand, the results of Devyver \cite[Corollary 7.3]{devyver2019index} and Tran \cite[Corollary 3.11]{tran2016index} imply that a non-flat embedded free boundary minimal surface in the unit ball has Morse index at least $4$. This motivates the following question.
\begin{qe}\label{qe:fbindex4}
Does there exist an embedded free boundary minimal surface $\Sigma\subset \mathbb{B}^{3}$ such that
\[
\operatorname{Area}(\Sigma)<\operatorname{Area}(\mathbb{K})
\qquad\text{and}\qquad
\operatorname{Index}(\Sigma)=4?
\]
\end{qe}
A negative answer to Question \ref{qe:fbindex4} would imply
\[
\omega_{4}(\mathbb{B}^{3})=\operatorname{Area}(\mathbb{K}).
\]
Moreover, Devyver \cite[Corollary 7.2]{devyver2019index} and Tran \cite[Corollary 3.10]{tran2016index} proved that every free boundary minimal surface of Morse index $4$ in the unit ball has genus zero. Also, we know that $\Sigma$ has at least $3$ boundary components if it satisfies conditions in Question \ref{qe:fbindex4}. Consequently, the conjectured area-minimizing properties of genus-zero free boundary minimal surfaces formulated in Question 3 and Conjecture 1.21 of Karpukhin--Kusner--McGrath--Stern \cite{karpukhin2024embedded} would yield a negative answer to Question \ref{qe:fbindex4}, and hence the desired equality for the fourth width.
\subsubsection{Free boundary minimal surfaces with arbitrary topology and their Morse indices} \label{section:arbitrarytopology}
Motivated by our repetitive min-max theory in the free boundary setting in Theorem \ref{thm:threeminimalannuli}, we propose the following conjecture:
\begin{conj} \label{conj:riemannianballindex}
    Let $M$ be a compact $3$-manifold with $\text{Ric} \ge 0$ and strictly convex boundary, for every $g \ge 0$ and $\beta \ge 2$ and $(g,\beta)\neq(0,2)$, there exists a free boundary minimal surface $\Sigma$ with genus $g$ and $\beta$ boundary components satisfying
    \[
\operatorname{Index}(\Sigma) \le \beta_{1}(\Sigma)+3 = 2 g+\beta+2,
\]
where $\beta_{1}(\Sigma)$ is the first Betti number of $\Sigma$.
\end{conj}

 Conjecture \ref{conj:riemannianballindex} is further motivated by the construction of minimal surfaces with arbitrary genus $g$ on positively curved spheres with $2g+3$ parameter family of surfaces by Chu \cite{chu2025min}. Resolution of Conjecture \ref{conj:riemannianballindex} may provide a variational characterization of free boundary minimal surfaces in the unit ball and a new construction of those, which would be a free boundary analog of Morse index characterization of Lawson surfaces $\xi_{g,1}$ of Morse index $2g+3$ and nullity $6$ by Kapouleas and Wigyul in \cite{kapouleas2020index}. 
\subsection{Organization}
We divide the paper into two parts. 

In \textbf{Part \uppercase\expandafter{\romannumeral 1}}, we give a lower bound on $\omega_4(\mathbb{B}^3)$ (Section \ref{sec: lower bound on w_4(B^3)}), and construct a $4$-parameter family of  annuli (Section \ref{sec:canonicalfamily} - \ref{sec : topologysixparameter}) to give an upper bound on $\omega_4(\mathbb{B}^3)$ (Section \ref{sec: upper bounda of the 4th-width}). Section \ref{sec :boundary blow-up} and Appendix \ref{appendix} are devoted to the technical details of boundary blow-up in the canonical family.

In \textbf{Part \uppercase\expandafter{\romannumeral 2}}, we show the existence of $3$ properly embedded free boundary minimal annuli. In Section \ref{sec: Existence of 3 free-boundary minimal annuli}, we give the framework of the proof. We leave the technical ingredients to Section \ref{Sec:proofmin-max} - \ref{sec: topology of boundary two cap}.
\subsection{Acknowledgements} We would like to thank Daniel Ketover for his encouragement, support, and valuable discussions. We are grateful to Fernando Codá Marques, André Neves, and Tristan Rivière for valuable discussions and comments. D.K. is grateful to Tobias Holck Colding and William Minicozzi for their continuing encouragement and support. We would like to thank Adrian Chun-Pong Chu, Giada Franz, Yangyang Li, Lorenzo Sarnataro, Mario Schulz, Doug Stryker, Zhihan Wang, Xin Zhou, Jonathan Zhu and Jiahua Zou for valuable discussions. G.S. has been partially supported by NSF DMS-2405114.
\section{Preliminaries}
Let $(M, g)$ be an $n$-dimensional compact Riemannian manifold with nonnegative Ricci curvature and strictly convex boundary $\partial M$. We fix the following notations. Readers may refer to \cite{sun2024multiplicity, chu2024existence, marques2017existence}.

\subsection{Notation.}
\begin{itemize}
    \item $\mathbf{I}_k(M;\mathbb{Z}_2)$: the set of integral $k$-dimensional currents in $M$ with $\mathbb{Z}_{2}$-coefficients.
    
    \item $Z_{k}(M,\partial M; \mathbb{Z}_{2})$: the space of relative cycles consisting of elements $T \in \mathbf{I}_k(M;\mathbb{Z}_2)$ such that $\operatorname{spt}(\partial T) \subset \partial M$.
    
    \item $\mathcal{Z}_{k}(M,\partial M; \mathbb{Z}_{2})$: define $T, S \in Z_{k}(M,\partial M; \mathbb{Z}_{2})$ to be equivalent if $T - S \in \mathbf{I}_k(\partial M;\mathbb{Z}_2)$. $\mathcal{Z}_{k}(M,\partial M; \mathbb{Z}_{2})$ is the quotient space of $Z_{k}(M,\partial M; \mathbb{Z}_{2})$  under this equivalent relation.
    
    \item $\mathcal{Z}_{k}(M,\partial M; \nu; \mathbb{Z}_{2})$ with $\nu = \mathcal{F}, \bF, \bM$: The set $\mathcal{Z}_{k}(M,\partial M; \mathbb{Z}_{2})$ equipped with the three topologies given by the flat norm $\mathcal{F}$, the $\mathbb{F}$-metric, and the mass norm $\bM$ (See the detailed definition in \cite{pitts2014existence} and \cite{marques2020applications}). For the flat norm, there are two definitions that induce the same topology:
    \[
    \mathcal{F}(T) := \operatorname{inf}\{ \bM(P) + \bM(Q) : T = P + \partial Q \},
    \]
    and
    \[
    \mathcal{F}(T) := \operatorname{inf}\{ \bM(Q) : T = \partial Q \}.
    \]
    In this paper, we will use the second definition.
    \item $\mathcal{V}_k(M)$: the closure in the varifold weak topology of the space of $k$-dimensional rectifiable varifolds in $M$. The $\mathbf{F}$-metric on $\mathcal{V}_k(M)$ is defined in Pitts' book \cite{pitts2014existence} and induces the varifold weak topology on $\mathcal{V}_k(M)$. $|| V||$ is the Radon measure induced by $V \in \mathcal{V}_k(M)$.
    \item $|T|$: the varifold in $\mathcal{V}_k(M)$ induced by a current $T \in \mathcal{Z}_{k}(M,\partial M; \mathbb{Z}_{2})$, or a countable $k$-rectifiable set $T \subset M$. In the same spirit, given a map $\Phi$ into $\mathcal{Z}_{k}(M,\partial M; \nu; \mathbb{Z}_{2})$, the associated map into $\mathcal{V}_k(M)$ is denoted by $|\Phi|$.
    \item $[W]$: the $\mathbb{Z}_2$-relative cycle induced by a countably $k$-rectifiable set with $\mathcal{H}^k(W) < \infty$. In the same spirit, given a map $f$ whose images are countably $k$-rectifiable sets, the associated map into $\mathcal{Z}_{k}(M,\partial M; \nu; \mathbb{Z}_{2})$ is denoted by $[f]$.
    \item $[|U|]$ is the associated integral current of an open set of finite perimeter $U \subset M$.
    \item $[\mathcal{W}]:= \{ [\Sigma_i]\} \subset \mathcal{Z}_{k}(M,\partial M; \nu; \mathbb{Z}_{2})$ for a set $\mathcal{W}$ of varifolds $\{ V_i \} \subset \mathcal{V}_k(M)$, each associated with a countably $k$-rectifiable set $\Sigma_i$.
    \item $\mathbf{B}_{\epsilon}^{\nu}(\cdot)$: the open $\epsilon$-neighborhood of an element or subset of $\mathcal{Z}_{k}(M,\partial M; \nu; \mathbb{Z}_{2})$.

    \item $\mathbf{B}_{\epsilon}^{\mathbf{F}}(\cdot)$: the open $\epsilon$-neighborhood of an element or subset of $\mathcal{V}_k(M)$.
    \item $\mathcal{C}(M)$: the space of Caccioppoli sets in $M$, equipped with the metric induced by the Lebesgue measure of the symmetric difference. We also denote by $\partial ^* \Omega$ the reduced boundary of $\Omega \in \mathcal{C}(M)$.
    \item $\Gamma^{\infty}(M)$: the set of smooth Riemannian metrics on $M$.
    \item $B_r(p)$: a geodesic ball of radius $r$ centered at $p \in M^3$.
    \item $\mathfrak{g}(\Sigma)$, $\mathfrak{\beta}(\Sigma)$, $\mathfrak{b}(\Sigma)$: If $\Sigma \subset M$ is a properly embedded, connected surface, $\mathfrak{g}(\Sigma)$ is the genus of $\Sigma$, $\mathfrak{\beta}(\Sigma)$ is the number of boundary components of $\Sigma$, and $\mathfrak{b}(\Sigma) \coloneqq \mathfrak{\beta}(\Sigma) - 1$ is the boundary complexity of $\Sigma$. 
    \item $\operatorname{in}(S) (\operatorname{out}(S))$: the inside (outside) open region of an oriented hypersurface $S$, provided that $S$ separates $M$.
\end{itemize}
 We use a standard cubical complex structure and their refinements for Almgren-Pitts and Simon-Smith min-max theory as a parameter space as in \cite{pitts2014existence, marques2017existence, chu2024existence}. We denote the diffeomorphism group of $M$ equipped with the $C^{\infty}$-topology by $\text{Diff}^{\infty}(M, \partial M)$.
\subsection{Almgren-Pitts min-max theory} As in the case for closed manifolds, Almgren's isomorphism theorem \cite{almgren1962homotopy} applies to the modulo $2$ relative cycle space $\mathcal{Z}_{n}(M;\partial M, \mathbb{Z}_{2})$ in the free boundary setting, namely the modulo $2$ relative cycle space $\mathcal{Z}_{n}(M;\partial M, \mathbb{Z}_{2})$ is weakly homotopically equivalent to $\mathbb {R}P^\infty$ as in Guang-Li-Wang-Zhou \cite{guang2021min}. We denote the generator of the cohomology ring of the cycle space by $\bar \lambda \in H^1(\mathcal{Z}_{n}(M;\partial M;\mathbb{Z}_2))$ and we denote the cohomology ring of $\mathcal{Z}_{n}(M;\partial M;\mathbb{Z}_2)$ by $\mathbb{Z}_2[\bar\lambda]$. Let $\mathcal{P}_p$ be the set of all $\bF$-continuous maps $\Phi:X\to \mathcal{Z}_2(M;\partial M; \bF; \mathbb{Z}_2)$, where $X$ is a finite simplicial complex, such that $\Phi^*(\bar\lambda^p)\ne 0$ and we call elements of $\mathcal{P}_{p}$ by $p$-\emph{sweepouts}. Since every finite cubical complex is homeomorphic to a finite simplicial complex and vice versa by Buchstaber-Panov \cite{buchstaber2002torus}, let us regard $X$ to be a finite cubical complex.

We define the Almgren-Pitts min-max $p$-width as follows with the notation $\text{dmn}(\Phi)$ of the domain of $\Phi$:
\begin{equation} \label{sswidthdef}
    \omega_{p}(M) = \inf_{\Phi \in \mathcal{P}_{p}} \sup_{x \in \text{dmn}(\Phi)} \bM(\Phi(x)).
\end{equation}
\subsection{Simon-Smith min-max theory} \label{subsec: Simon-Smith min-max theory}
As in Chu-Li \cite[Section 2]{chu2024existence}, we consider the class of punctuate surfaces with boundary as a generalized notion of surfaces (with boundary).
\begin{defn}[Relative punctate surface] \label{def: Punctate surfaces}
Let $(M^3,\partial M, g)$ be a compact Riemannian $3$-manifold with boundary. 
A closed set $\Sigma \subset M$ is called a \emph{punctate surface with boundary} if:

\begin{enumerate}[label=(\roman*)]
\item The Hausdorff measure of $\Sigma$ satisfies $\mathcal{H}^2(\Sigma) \in (0,\infty)$;

\item There exists a finite set $P \subset \Sigma$ such that $\Sigma \setminus P$ is a smooth, embedded, orientable surface with boundary satisfying
\[
\partial (\Sigma \setminus P) \subset \partial M;
\]

\item $\Sigma \setminus P$ meets $\partial M$ transversely;

\item We denote by $\Sigma_{\text{iso}}$ the set of isolated points of $\Sigma$, 
the complement $M  \setminus (\Sigma \setminus \Sigma_{\text{iso}})$ consists of two open sets whose common relative boundary is $\Sigma \setminus \Sigma_{\text{iso}}$.
\end{enumerate}
\end{defn}
\begin{rem}
    In the free boundary setting, we allow the singularities to occur on the boundary, i.e. $P\cap \partial M$ can be nonempty.
\end{rem}

We denote the space of punctate surfaces with boundary in $M$ by $\mathcal{S}(M)$. Let us denote the punctate set of $\Sigma$ by $P$. Then we say $\Sigma \in \mathcal{S}(M)$ is connected if $\Sigma \setminus \Sigma_{iso}$ is connected. We define the genus of a connected punctate surface $\Sigma$ by
\[
\mathfrak{g}(\Sigma) := \lim_{r \rightarrow 0} \mathfrak{g}(\Sigma \setminus B_{r}(P)),
\]
and the number of boundary components of $\Sigma$ by
\[
\mathfrak{\beta}(\Sigma) := \#\pi_0\left(\overline{\partial (\Sigma \setminus P)}\right).
\]
It is not hard to see that the definitions above are independent of the choice $r \to 0$ and the punctate set $P$. We consider a more generalized class than punctate surfaces, $\mathcal{S}^{*}(M)$ as
\[
\mathcal{S}^{*}(M) = \mathcal{S}(M) \cup \{ \Sigma \subset M \, | \, \Sigma \text{ is a closed set satisfying } \mathcal{H}^{2}(\Sigma)=0 \}. 
\]
For a punctate surface $\Sigma \in \mathcal{S}(M)$ such that $\Sigma \setminus \Sigma_{iso}$ is a disjoint union of connected punctate surfaces $\{\Gamma_1, \cdots, \Gamma_N \}$, we define the \emph{genus} $\mathfrak{g}$ and the \emph{boundary complexity} $\mathfrak{b}$ of $\Sigma$ as follows.
\begin{align*}
    \mathfrak{g}(\Sigma)&:=\sum_{j=1}^N\fg(\Gamma_j) \quad \mathfrak{b}(\Sigma):=\sum_{j \in B} (\beta(\Gamma_{j})-1),
\end{align*}
where $\Gamma_{j}$ is a connected component with nonempty boundary if $j \in B$. 

For nonnegative integers $\fg_0, \fb_{0}\geq 0$, we say that $\Sigma \in \mathcal{S}(M)$ satisfies the \emph{topological bound} $(\fg_0, \fb_{0})$ if
\begin{align}
    \label{topbounddef} \mathfrak{g}(\Sigma) &\leq \fg_0, \quad \mathfrak{g}(\Sigma) + \mathfrak{b}(\Sigma) \leq \fg_{0}+\fb_{0},
\end{align}
Note that our definition of the topological bounds (\ref{topbounddef}) are inspired by the topological control results of min-max surfaces in Franz-Schulz \cite[Definition 1.6]{franz2023topological}.
\begin{defn}[Free boundary Simon-Smith family] \label{def:fbssfamily}
Let $(M,\partial M)$ be a compact Riemannian $3$-manifold with boundary, 
and let $X$ be a cubical subcomplex of some $I(m,j)$. 
A map $\Phi : X \to \mathcal{S}^*(M)$ is called a 
\emph{free boundary Simon--Smith family} provided that:

\begin{enumerate}[label=(\roman*)]
\item The map $x \mapsto \mathcal{H}^2(\Phi(x))$ is continuous.

\item For any $x_0 \in X$ and any open set $U \supset \Phi(x_0)$,
there exists a neighborhood $O \subset X$ of $x_0$ such that
$\Phi(x) \in U$ for all $x \in O$.

\item For each $x_0 \in X$ with $\Phi(x_0) \in \mathcal{S}(M)$,
there exists a punctate set $P_{\Phi(x_0)} \subset \Phi(x_0)$ such that
\[
N_P(\Phi) :=\sup_{x : \Phi(x) \in \mathcal{S}(M)}|P_{\Phi(x)}|< \infty.
\]

\item For each $x \in X$ with $\Phi(x) \in \mathcal{S}(M)$,
the regular part $\Phi(x)\setminus P_{\Phi(x)}$
is a smooth embedded orientable surface with boundary satisfying
\[
\partial (\Phi(x)\setminus P_{\Phi(x)}) \subset \partial M.
\]

\item For the punctate set $P_{\Phi(x_0)}$ chosen in (iii),
for any $x_0 \in X$ with $\Phi(x_0) \in \mathcal{S}(M)$ and any open set $U \subset \subset M \setminus P_{\Phi(x_0)}$,
the convergence $\Phi(x) \to \Phi(x_0)$ is smooth on $U$
whenever $x \to x_0$.
\end{enumerate}
In this case, we call $X$ the parameter space of $\Phi$.
Moreover, for non-negative integers $\mathfrak{g}_0$, $\mathfrak{b}_{0}$, we say that $\Phi$ is a free boundary Simon--Smith family of topological bounds $(\mathfrak{g}_0,\mathfrak{b}_{0})$ if for each $x \in X$ with $\Phi(x) \in \mathcal{S}(M)$, the surface $\Phi(x)$ satisfies the \emph{topological bound} $(\fg_0, \fb_{0})$.
\end{defn}
We define the homotopy class of Simon-Smith family on $(M, \partial M, g)$.
\begin{defn}[Homotopy class] \label{def:homotopyclasssimonsmith}
Let $M$ be a compact Riemannian $3$-manifold with boundary, and let $\Phi_0,\Phi_1$ be two (relative) Simon-Smith families parametrized by the same parameter space $X$. Then $\Phi_0$ and $\Phi_1$ are called \emph{homotopic} if there exists a continuous map
\[
\varphi:[0,1]\times X \to \text{Diff}^\infty(M,\partial M).
\]
such that:
\begin{enumerate}[label=(\roman*)]
\item $\varphi(0,x)=\text{Id}$ for all $x\in X$;
\item $\varphi(1,x)\big(\Phi_0(x)\big)=\Phi_1(x)$ for all $x\in X$.
\end{enumerate}
\end{defn}
We call the set of all families homotopic to a Simon-Smith family $\Phi$ by \emph{the homotopy class associated to} $\Phi$ and denote this by $\Lambda(\Phi)$. Note that we can assume that $N_{P}(\Phi) < \infty$, the properly embeddedness is preserved and the upper bounds of topological bounds are bounded within $\Lambda(\Phi)$.

Now we define the min-max width within the homotopy class $\Lambda(\Phi)$ as follows:
\begin{equation} \label{sswidthdef}
    \bL(\Lambda) = \inf_{\Phi \in \Lambda} \sup_{x \in X} \mathcal{H}^{2}(\Phi(x)).
\end{equation}

As in \cite[Section 3]{marques2016morse} (also see \cite[Section 2]{chu2024existence}), we give the following definitions.

\begin{defn} \label{def: min-max sequence}(Minimizing sequence and min-max sequence).
A sequence $\{\Phi_i\} \subset \Lambda$ is said to be minimizing if 
\[
\lim_{i \to \infty} \sup_{x \in X} \mathcal{H}^2(\Phi_i(x)) = \mathbf{L}(\Lambda).
\]
If $\{\Phi_i\}$ is a minimizing sequence in $\Lambda$ and $\{ x_i \} \subset X$ satisfies
\[
\lim_{i \to \infty} \mathcal{H}^2(\Phi_i(x_i)) = \mathbf{L}(\Lambda),
\]
then $\{ \Phi_i(x_i) \}$ is called a \textit{min-max sequence}. For a minimizing sequence $\{\Phi_i\}$, we define its critical set $\mathbf{C}(\{ \Phi_i \})$ to be the set of all subsequential varifold-limits of its \textit{min-max sequences}:
\[
\mathbf{C}(\{ \Phi_i \}) := \{ V = \lim_{j \to \infty} | \Phi_{i_j}(x_{i_j})| : x_{i_j} \in X, ||V|| (M) = \mathbf{L}(\Lambda)\}.
\]
\end{defn}
Moreover, $\{ \Phi_{i} \}$ is called \emph{pulled-tight} if every varifold in $\bC(\{ \Phi_{i} \})$ is stationary.

\begin{defn}
In a compact $3$-manifold $M$ with boundary, for $L > 0$, we define $W_L(M)$ to be the set of all varifolds $W \in \mathcal{V}_2(M)$, with $||W||(M) = L$, of the form
\[
W = m_1 |\Gamma_1| + \cdots + m_l |\Gamma_l|,
\]
where $\{m_j\}$ is a set of positive integers and $\{ \Gamma_j \}$ is a set of disjoint properly embedded, connected, free boundary minimal surfaces in $(M, g)$. For nonnegative integers $\fg_0, \fb_{0}\geq 0$, let $\mathcal{W}_{L, \le \fg_{0}, \le \fg_{0} + \fb_{0}} \subseteq \mathcal{W}_{L}$ be the subset consisting of $W \in W_L(M)$ whose support, $\operatorname{supp}||W||$, satisfies the \emph{topological bound} $(\fg_0, \fb_{0})$.
\end{defn}

Now we state the free boundary Simon-Smith min-max theorem and will prove it in Section \ref{Sec:proofmin-max}.
\begin{rem}
As in Franz-Schulz \cite[Remark 1.5]{franz2023topological}, we focus on the topological bounds without multiplicities since this is sufficient for our applications. We also note that every properly embedded free boundary minimal surface is unknotted by Chu-Franz \cite[Theorem 1.1]{chu2025unknottedness}.
\end{rem}

\begin{thm}[Simon-Smith min-max theorem]\label{thm:simon-smith min-max}
        Given a compact orientable Riemannian $3$-manifold $(M, g)$ with mean convex boundary $\partial M$, a cubical subcomplex $X$ of some $I(m, k)$, a Simon-Smith family $\Phi: X \to \mathcal{S}^*(M)$ satisfying a topological bound $(\fg_{0}, \fb_{0})$, and a positive real number $r > 0$, if $L := \mathbf{L}(\Lambda(\Phi)) > 0$, 
        then there exists a minimizing sequence $\{\Phi_{i}\}$ in $\Lambda(\Phi)$ such that:
	\begin{enumerate}[label=\normalfont(\arabic*)]
            \item\label{item:minMaxPulltight} The sequence $\{\Phi_i\}$ is pulled-tight and moreover,
            \[
                \bC(\{\Phi_{i}\})\cap \mathcal{W}_{L, \leq \mathfrak{g}_0,\le \fg_{0}+\fb_{0}} \neq \emptyset\,.
            \]
            \item\label{item:minMaxMultiplicityone} There exists a $W \in \mathcal{W}_{L, \leq \mathfrak{g}_0, \le \fg_{0}+\fb_{0}} \cap \mathbf{C}(\{ \Phi_i \})$ such that every connected component $\Gamma_j$ in $W$ satisfies:
		\begin{enumerate}[label=\normalfont(\alph*)]
			\item If $\Gamma_j$ is unstable and two-sided, then $m_j = 1$.
			\item If $\Gamma_j$ is one-sided, then its connected double cover is stable.
		\end{enumerate}
            \item \label{item:minMaxW} Furthermore, there exists $\eta > 0$ such that for all sufficiently large $i$,
		\[
			\mathcal{H}^2(\Phi_{i}(x)) \geq L - \eta \implies |\Phi_i(x)| \in \bB^{\bF}_{r}(\mathcal{W}_{L, \leq \mathfrak{g}_0,\le \fg_{0} + \fb_{0}} \cap \bC(\{\Phi_{i}\}))\,.
		\]
	\end{enumerate}
    \end{thm}

\subsection{Relative free boundary Simon-Smith min-max theory} We introduce the relative Simon-Smith min-max theory in the  free boundary setting. Let us fix parameter spaces, cubical complex $X$ and its subcomplex $Z \subset X$. We develop the relative $(X,Z)$-homotopy class in the free boundary setting (See Zhou \cite{zhou2020multiplicity} where the class is first introduced). Let us define $\Lambda_{Z}(\Phi)$ analogously with Chu-Li \cite[Section 2.3]{chu2024existence}. Now we define the min-max width within the homotopy class $\Lambda_{Z}(\Phi)$ as follows:
\begin{equation} \label{sswidthdef}
    \bL(\Lambda_{Z}) = \inf_{\Phi \in \Lambda_{Z}} \sup_{x \in X} \mathcal{H}^{2}(\Phi(x)).
\end{equation}
We define the minimizing sequence, \textit{min-max sequence}, and the critical set for $\Lambda_Z$ in the same way as Definition \ref{def: min-max sequence}. The relative Simon-Smith width is always larger than equal to the Simon-Smith as in \cite[Lemma 2.20]{chu2024existence}. Now we state the relative Simon-Smith min-max theorem in the free boundary setting.
\begin{thm} [Relative Simon-Smith min-max theorem] \label{thm:relative simon-smith min-max}
     Given a compact orientable Riemannian $3$-manifold $(M, g)$ with a strictly mean convex boundary $\partial M$, a pair of cubical subcomplex $Z \subset X$ of some $I(m, k)$, a Simon-Smith family $\Phi: X \to \mathcal{S}^*(M)$ satisfying the topological bound $(\fg_{0},\fb_{0})$, and a positive real number $r > 0$, if 
     \[L := \mathbf{L}(\Lambda_{Z}(\Phi)) > \sup_{z \in Z} \mathcal{H}^{2}(\Phi(z)),\]
     then there exists a minimizing sequence $\{\Phi_{i}\}$ in $\Lambda_{Z}(\Phi)$ such that conclusions \normalfont (1),  \normalfont (2) and \normalfont (3) in Theorem~\ref{thm:simon-smith min-max} still hold.
\end{thm}
We will provide the proof of Theorem \ref{thm:relative simon-smith min-max} in Section \ref{Sec:proofmin-max} with the proof of Theorem \ref{thm:simon-smith min-max}. As in Chu-Li \cite{chu2024existence}, we have that Simon-Smith families are sweepouts in Almgren-Pitts min-max theory.
\begin{prop} [Simon-Smith families are Almgren-Pitts sweepouts]\label{prop:sslargerap}
Every element $\Sigma$ in $\mathcal{S}^{*}(M)$ is associated with a unique relative 2-cycle $[\Sigma]$ in $\mathcal{Z}_{2}(M;\partial M; \mathbb{Z}_{2})$. In particular, if $\mathcal{H}^{2}(\Sigma) =0$, then $[\Sigma]=0$. Moreover, for a Simon-Smith family $\Phi$, the induced map
\[
[\Phi]: X \rightarrow \mathcal{Z}_{2}(M;\partial M; \bF; \mathbb{Z}_{2})
\]
is continuous with respect to $\bF$-topology.
\end{prop}

\part{The topological family of annuli and an upper bound on $\omega_4(\mathbb{B}^3)$}

\section{Lower bound on $\omega_4(\mathbb{B}^3)$} \label{sec: lower bound on w_4(B^3)}

The first three widths of the unit ball are given by
\[
\omega_{1}(\mathbb{B}^{3})=\omega_{2}(\mathbb{B}^{3})=\omega_{3}(\mathbb{B}^{3})= \pi,
\]
which follows from straightforward constructions of first three sweepouts and the regularity theory of min-max minimal surfaces. Now we prove the fourth width of a Euclidean unit ball $\mathbb{B}^{3}$ is strictly larger than $\pi$, the area of a unit disk. The proof uses the gap theorem for free boundary minimal surfaces by Ketover \cite[Proposition 2.1]{ketover2016free} and is a modification of Nurser's argument to show $\omega_{5} (\mathbb{S}^{3})>4 \pi$ \cite[p. 44-51]{nurser2016low}.
\begin{prop} \label{prop:lowerfourthbound} $\omega_{4}(\mathbb{B}^{3}) > \pi$.
\begin{proof}
Suppose not. There exists a $4$-dimensional cubical subcomplex $X$ and a sweepout $\Pi\in[X, \mathcal{Z}_2(\mathbb{B}^3;\partial \mathbb{B}^3; \mathbb{Z}_2)]$ such that
\[
 L(\Pi) <\pi + \delta,
\]
where $\delta>0$ is chosen as in \cite[Proposition 2.1]{ketover2016free} such that if $\Sigma$ is an embedded free boundary minimal surface in $\mathbb{B}^{3}$ and $|\Sigma| < \pi + \delta$, then $\Sigma$ is a flat free boundary disk. Then $L(\Pi)$ is realized by an embedded free boundary minimal surface by Almgren-Pitts min-max theorem, and hence by an equatorial disk by  \cite[Proposition 2.1]{ketover2016free}. Thus, $L(\Pi) = \pi$ and we will show that this is impossible.

After pull-tight, we can choose a critical sequence $S = \{ \phi_i \}_{i \in \mathbb{N}} \in \Pi$ of width $\pi$. Take $\Phi_i : X \to \mathcal{Z}_2(\mathbb{B}^3;\partial \mathbb{B}^3; \mathbb{Z}_2)$ as an Almgren extension of $\phi_i$. For sufficiently large $i$, $\Phi_{i}$ is a $4$-sweepout, so if $\bar\lambda\in H^1(\mathcal{Z}_{2}(\mathbb{B}^{3};\partial \mathbb{B}^{3}; \mathbb{Z}_{2});\mathbb Z_2)$ and $\lambda = \Phi_{i}^{*}(\bar{\lambda})$, then 
\begin{equation}\label{eq:lambda4-nonzero}
\lambda^4\neq 0\quad\text{in }H^4(X;\mathbb Z_2).
\end{equation}

Let $\mathcal{T}_{D}\subset \mathcal{Z}_{2}(\mathbb{B}^{3};\partial \mathbb{B}^{3}; \mathbb{Z}_{2})$ be the subset of cycles given by (unoriented) equatorial disks, and let $\mathcal{T}_{V}$ be the corresponding set of associated varifolds. Then
\[
\mathcal{T}_{D} \simeq \mathcal{T}_{V} \simeq \mathbb{RP}^2.
\]
For a small constant $\epsilon>0$, set the $\epsilon$-tubular neighborhood $\mathcal{U}_{\epsilon} (\mathcal{T}_D) := \{ T \in \mathcal{Z}_2(\mathbb{B}^3;\partial \mathbb{B}^3; \mathbb{Z}_2) : \mathcal{F}(T, \mathcal{T}_D) < \epsilon \}$. By \cite[Lemma 52]{nurser2016low}, we can choose $\eta = \eta(\epsilon)>0$ so that
\begin{equation}
\label{eq:eta-eps}
\bF(|T|,\mathcal{T}_{V})\le 2\eta \ \Longrightarrow\ \mathcal{F}(T,\mathcal{T}_{D}\cup\{0\})<\varepsilon \qquad\text{for all }T\in\mathcal{Z}_{2}(\mathbb{B}^{3};\partial \mathbb{B}^{3}; \mathbb{Z}_{2}).
\end{equation}
Take $k_i \in \mathbb{N}$ so that $\text{dmn}(\phi_i) = X(k_i)_0$. Define $Y_i \subset X$ to be the cubical subcomplex of $X(k_i)$ consisting of all cells $\alpha$ such that $\bF(|\phi_i(x)|,\mathcal{T}_{V})\ge \eta$ for every vertex $x\in\alpha_0$. Note that $\mathbf{F}(|\Phi_i(x)|, \mathcal{T}_V) < 2 \eta$ for every $x \in X \setminus Y_i$. Set $Z_i:=\overline{X\setminus Y_i}$ and refine the complexes, then $X=Z_i\cup Y_i$. We denote the inclusions by $i_1\colon Z_i\hookrightarrow X$ and $i_2\colon Y_i\hookrightarrow X$. Lyusternik–Schnirelmann vanishing lemma \cite[Lemma 5.3]{Haslhofer2019} and (\ref{eq:lambda4-nonzero}) implies $(i_1^*\lambda)^3\neq 0 \text{ in }H^3(Z_i;\mathbb Z_2)$ or $i_2^*\lambda\neq 0 \text{ in }H^1(Y_i;\mathbb Z_2)$ for all $i \in \mathbb{N}$. 

Suppose $i_2^*\lambda\neq 0$. Then $(\Phi_i)|_{Y_i}$ is a $1$-sweepout whose width is at most $\pi$,  and must be $\pi$ since $w_1(\mathbb{B}^3) = \pi$. Also, all the critical varifold has a $\bF$-distance from $\mathcal{T}_{v}$ at least $\eta$ by construction. On the other hand, any smooth critical varifold has mass $\pi$, and hence is an equatorial disk, a contradiction. Thus the critical set contains no varifold with smooth support. By the pull-tight argument, we can deform a sweepout to have strictly lower mass than $\pi$, and this gives a contradiction. 

Now we prove that $(i_1^*\lambda)^3\neq 0$ cannot hold by showing that the restriction $\Phi_i|_{Z_i}$ cannot be a $3$-sweepout. Suppose not. By (\ref{eq:eta-eps}), we have $\Phi_i(Z_i)\subset \mathcal U_\varepsilon(\mathcal{T}_{D} \cup \{0\})$, and we can split $Z_i$ disjointly into the parts for which $\Phi_i(x)$ is close to $\mathcal{T}_D$ and $\{0 \}$ as $Z_i^{(D)} \sqcup Z_i^{(0)}$. Note that $(i_1^*\lambda)|_{Z_i^{(0)}} = 0$ as there are no $1$-sweepouts close to $0$. Thus, $(i_1^*\lambda)^3|_{Z_i^{(D)}}\neq 0$ and $\Phi_i|_{Z_i^{(D)}}$ is a $3$-sweepout. We abuse notation by using $Z_i$ to denote $Z_i^{(D)}$. Thus, we have a $3$-sweepout $\Phi_i: Z_i \to \mathcal{Z}_2(\mathbb{B}^3, \partial \mathbb{B}^3 ; \mathbb{Z}_2)$ with $\mathcal{F}(\Phi_i(x), \mathcal{T}_D) < \epsilon$. 

We take $\epsilon>0$ as in the projection lemma \cite[Lemma 9.10]{marques2014} or \cite[3.3.6, p.49-51]{nurser2016low} so that for sufficiently large $i$, we can find a continuous map $\tilde{\Phi}_i: Z_i \to \mathcal{T}_D$ such that $\tilde{\Phi}_i$ is homotopic to $\Phi_i|_{Z_i}$. Since $H^3(\mathcal{T}_{D}; \mathbb{Z}^2) = H^3(\mathbb{RP}^2; \mathbb{Z}_2) =0$, $\tilde{\Phi}_i$ cannot be a $3$-sweepout and we obtain a contradiction. Hence $\omega_{4}(\mathbb{B}^{3}) > \pi$.
\end{proof}
\end{prop}

\section{$4$-parameter canonical family of surfaces} \label{sec:canonicalfamily}
In this section, we introduce a $4$-parameter canonical family of annuli (with its degenerations) in the unit ball $\mathbb{B}^{3}$. This is a free boundary analog of the $5$-parameter canonical family of surfaces on $\mathbb{S}^{3}$ by Marques-Neves \cite{marques2014}. 

To each properly embedded free boundary surface in $\mathbb{B}^3$, we associate a continuous $4$-parameter family of surfaces. Then for any properly embedded annulus with antipodal symmetry that intersects $\partial \mathbb{B}^3$ orthogonally, this associated family of surfaces can be reparametrized into an Almgren-Pitts $4$-sweepout using boundary identifications. By Proposition \ref{prop:lowerfourthbound} above, this $4$-sweepout will give a free boundary minimal surface other than the disk.

\subsection{Notation and definitions} \label{subsec: general free  boundary surface notation} We use the following notation:
\begin{itemize}
    \item $\mathbb{B}^3 \subset \mathbb{R}^3$ is the unit ball. $\mathbb{S}^n$ is the unit $n$-sphere for $n \in \mathbb{N}$,  $\mathbb{S}^3_{+} \coloneqq \mathbb{S}^3 \cap \{ x_4 > 0 \}$ is the upper hemisphere of $\mathbb{S}^3$.
    \item  For $n \in \mathbb{N}$, $B^n_R (Q) = \{ x \in \mathbb{R}^n: |x - Q| < R \}$ and $B_r(p) = \{ x \in \mathbb{S}^3 : d(x, p) < r \}$, where $Q \in \mathbb{R}^n, \, p \in \mathbb{S}^3, \, R, r> 0$, and $d$ is the spherical distance, $B^+_r(p) \coloneqq B_r(p) \cap \mathbb{S}^3_+$ is the half geodesic ball in $\mathbb{S}^3$.
\end{itemize}
Throughout the paper, we may abbreviate notation by using $B_r(v, 0)$, $B^+_r(v, 0)$, $B_r^4(v, 0)$ for the geodesic ball $B_r\left((v, 0)\right)$, half geodesic ball $B^+_r\left((v, 0)\right)$, Euclidean ball $B^4_r((v, 0))$ respectively, where $(v, 0) \in \mathbb{R}^4$, $v \in \mathbb{R}^3$.

\subsection{Canonical family} \label{subsec: Canonical family}
For every $v \in \mathbb{B}^{3}$, we consider the conformal map $F_v : \mathbb{B}^3\to \mathbb{B}^3$ by
 \begin{equation}
   F_v(x) = \frac{(1 - |v|^2)x - (1 - 2 \langle v, x \rangle + |x|^2) v}{1 - 2 \langle v, x \rangle + |x|^2 |v|^2},\quad v \in\mathbb{B}^3.
 \end{equation}
It is easy to observe that $F_v$ is a conformal diffeomorphism of $\overline{\mathbb{B}^3}$ with inverse $F_{-v}$ and $F_v({\partial \mathbb{B}^3}) = \partial \mathbb{B}^3$. Define the inverse stereographic projection $\pi: \mathbb{R}^3 \cup \{ \infty \} \to \mathbb{S}^3$ by
\begin{equation*}
  \pi(x_1, x_2, x_3) = \left(\frac{2x_1}{|x|^2 + 1}, \frac{2x_2}{|x|^2 + 1}, \frac{2x_3}{|x|^2 + 1}, \frac{1 - |x|^2}{1 + |x|^2}\right).
\end{equation*}
Note that $\pi|_{\mathbb{B}^3}: \mathbb{B}^3 \to \mathbb{S}_+^3$ is a diffeomorphism and $\pi(\partial \mathbb{B}^3) = \partial \mathbb{S}_+^3$.

Let $\Sigma \subset \mathbb{B}^3$ be a properly embedded surface that intersects $\partial \mathbb{B}^3$ orthogonally. Throughout the paper, for notational convenience, for any properly embedded surface $\Sigma \subset \mathbb{B}^3$, we say $\Sigma$ is orthogonal if $\Sigma$ intersects $\partial \mathbb{B}^3$ orthogonally. We make several definitions regarding the geometry of a tubular neighborhood of $\partial \Sigma$ in $\partial \mathbb{B}^3 = \mathbb{S}^2$.
\begin{itemize}
    \item $A$ and $A^*$ denote the disjoint connected components of $\overline{\mathbb{B}^3} \setminus \Sigma = A \cup A^*$; $U = \partial A \, \setminus \Sigma$, $U^* = \partial A^* \, \setminus \Sigma$ are the boundary pieces of $A$ and $A^*$ on $\mathbb{S}^2$. 
    
    \item $N$ denote the unit normal to $\Sigma$ that points into $A^*$.
    \item Denote
    \[
    D^2_+(r) = \{ s = (s_1, s_2) \in \mathbb{R}^2 : |s| < r, s_1 \geq 0 \}.
    \]
    \item If $\epsilon > 0$ is sufficiently small, the map $\Lambda: \partial \Sigma \times D^2_+(3 \epsilon) \to \overline{\mathbb{B}^3}$ given by
    \begin{equation} \label{eq: def of Lambda}
        \Lambda(p, s) = (1 - s_1)(\cos{s_2} \,p + \sin{s_2} \, N(p))
    \end{equation}
    is a diffeomorphism onto a neighborhood of $\partial \Sigma$ in $\overline{\mathbb{B}^3}$.
    \item let $\Omega_r = \Lambda(\partial \Sigma \times D^2_+(r))$ for all $r \leq 3 \epsilon$.

    Consider the continuous map $T: \overline{\mathbb{B}^3} \to \overline{\mathbb{B}^3}$ such that
    \begin{itemize}
        \item $T$ is the identity on $\overline{\mathbb{B}^3} \setminus \Omega_{3 \epsilon}$;
        \item on $\Omega_{3 \epsilon}$, we have
        \[
        T(\Lambda(p, s)) = \Lambda(p, \phi(|s|)s),
        \]
        where $\phi$ is smooth, zero on $[0, \epsilon]$, strictly increasing on $[\epsilon, 2 \epsilon]$, and one on $[2 \epsilon, 3 \epsilon]$.
    \end{itemize}
\end{itemize}
The map $T$ collapses a tubular neighborhood of $\partial \Sigma$ onto $\partial \Sigma$.

Define
\[
A_v = F_v(A), \quad A_v^* = F_v(A^*), \quad \Sigma_v = F_v(\Sigma) = \partial A_v.
\]
We choose the unit normal to $\pi (\Sigma_v)$ to be \begin{equation} \label{eq: normal N_v}
 N_v =D(\pi \circ F_v)(N) / |D(\pi \circ F_v)(N)|,   
\end{equation} and let $d_v: \overline{\mathbb{S}_+^3} \to \mathbb{R}$ be the signed distance to $\pi(\Sigma_v) \subset \overline{\mathbb{S}_+^3}$:
\[
d_v(x) = \begin{cases}
    & d(x, \pi(\Sigma_v)) \quad \text{if }x \notin \pi(A_v), \\
    & -d(x, \pi(\Sigma_v)) \quad \text{if }x \in \pi(A_v)
\end{cases}
\]

The canonical family of $\Sigma$ is the $4$-parameter family of $2$-rectifiable subsets of $\mathbb{B}^3$ given by
\[
\Sigma_{(v, t)} = \partial A_{(v, t)} \cap \mathbb{B}^3, \quad \text{where } A_{(v, t)} = \pi^{-1}\bigl( \{x \in \mathbb{S}^3_+ : d_v(x) < t \} \bigr)  
\]

Following the discussion in \S \ref{subsubsec: Boundary blow-up}, we reparametrize and extend the canonical family to all of $\overline{\mathbb{B}}^3 \times [- \pi, \pi]$ as Marques-Neves did in \cite{marques2014} and produce, out of the canonical family, a $4$-dimensional family of integral currents that is continuous in the flat topology (Theorem \ref{thm: flattop conti blow-up}). Before we implement the construction, we first define the extended Gauss map and the angle function as in \cite[Section 3]{marques2014}.

\subsection{Extended Gauss map.} For every $p \in \partial \Sigma$ and $k \in [-\infty, \infty]$, consider
\begin{equation} \label{eq: Q^bar_p, k}
    \overline{Q}_{p, k} = - \frac{k}{\sqrt{1 + k^2}}p - \frac{1}{\sqrt{1 + k^2}} N(p) \in \mathbb{S}^2.
\end{equation}
This induces a function $\overline{Q}: \overline{\Omega}_{\epsilon} \to \mathbb{S}^2$ such that
\[
\overline{Q}(\Lambda(p, s)) = \overline{Q}_{p, k}, \quad \text{where }k = \frac{s_2}{\sqrt{\epsilon^2 - s_2^2}}.
\]
We extend this map in the following way:
\begin{equation} \label{def:Q^bar}
    \overline{Q}: \mathbb{S}^2 \cup \overline{\Omega}_{\epsilon} \to \mathbb{S}^2, \quad \overline{Q}(v) = \begin{cases}
      &  - T(v) \quad \text{if }v \in U^* \setminus \overline{\Omega}_{\epsilon}, \\
      &  T(v) \quad \text{if }v \in U \setminus \overline{\Omega}_{\epsilon}, \\
      & \overline{Q}(v) \quad \text{if }v \in \overline{\Omega}_{\epsilon}.
    \end{cases}
\end{equation}

\subsection{Angle function.} For every $k \in [ -\infty, + \infty]$, consider 
\begin{equation} \label{eq: r^bar_k}
\overline{r}_k = \frac{\pi}{2} - \arctan{k} \in [0, \pi] 
\end{equation}

Also consider $\overline{r}: \overline{\Omega}_{\epsilon} \to [0, \pi]$ given by
\[
\overline{r}(\Lambda(p, s)) = \overline{r}_k, \quad \text{where }k = \frac{s_2}{\sqrt{\epsilon^2 - s_2^2}}.
\]
We extend this function in the following way:
\begin{equation} \label{def:r^bar}
    \overline{r}: \mathbb{S}^2 \cup  \overline{\Omega}_{\epsilon} \to [0, \pi], \quad \overline{r}(v) = \begin{cases}
        0 & \text{if }v \in U^* \setminus \overline{\Omega}_{\epsilon}, \\
        \pi & \text{if }v \in U \setminus \overline{\Omega}_{\epsilon}, \\
        \overline{r}(v) & \text{if }v \in \overline{\Omega}_{\epsilon}.
    \end{cases}
\end{equation}
Now we reparametrize the canonical family to give the continuous $4$-parameter family of currents we need.

\begin{thm} \label{thm: flattop conti blow-up}
The map below is well-defined and continuous in the flat topology:
\[
C: \overline{\mathbb{B}^{3}} \times [- \pi, \pi] \to \mathcal{Z}_2(\mathbb{B}^3,\partial \mathbb{B}^{3}; \mathbb{Z}_2),
\]
\[
C(v,t)=
\begin{cases}
\partial [| A_{(T(v),t)} |],
& \text{if } v\in \mathbb{B}^{3}\setminus \overline{\Omega}_{\epsilon}, \\[4pt]
\partial[|\pi^{-1}\left(
 B^+_{\overline r(v)+t}(\overline Q(v),0)
\right)|],
& \text{if } v\in \partial\mathbb{B}^{3}
\cup
\overline{\Omega}_{\epsilon}.
\end{cases}
\]
 Furthermore, we have
 $C(v, \pi) = C(v, - \pi) = 0$ for every $v \in \overline{\mathbb{B}^{3}}$.
\end{thm}
 We extend $C(v,t) =0$ for $t \in (-\infty, -\pi] \cup [\pi, \infty)$. We postpone the proof of Theorem \ref{thm: flattop conti blow-up} to Section \ref{sec :boundary blow-up}. The proof includes technical details of the free boundary analog of boundary blow-up in $\mathbb{S}^{3}$ as in \cite[Section 5]{marques2014}.
\begin{comment}
\begin{rem} \label{rem: Area bound}
For $v \in \partial \mathbb{B}^3$, let $\alpha = \alpha(v, t) = \overline r(v)+t$. For $\alpha \in (0, \pi) \setminus \{ \frac{\pi}{2}\}$, we observe that
\[
C(v, t) = \partial B^3_{|\tan{\alpha}|}(\sec{(\alpha)} \,\overline Q(v) ) \cap \mathbb{B}^3
\]
is an orthogonal spherical cap in $\mathbb{B}^{3}$. For $\alpha = \frac{\pi}{2}$,
\[
C(v, t) = \{ \langle x, \overline Q(v) \rangle = 0 \} \cap \mathbb{B}^3\] 
is an equator disk. By direct computations, for any $v \in \partial \mathbb{B}^3$ and $\alpha \neq \frac{\pi}{2} $,
we have the area bound
\[
\begin{split}
\mathbf{M}(C(v, t)) & = \mathcal{H}^2(\partial B^3_{|\tan{\alpha}|}(\sec{(\alpha)} \,\overline Q(v) ) \cap \mathbb{B}^3) \\
& = 2 \pi \frac{\sin^2{\alpha}}{1 + \sin{\alpha}} < \pi
\end{split}
\]
and $\mathcal{H}^2(C(v, t)) = \pi$ iff  $ \alpha(v, t) =  \pi/2$.
\end{rem}

\end{comment}

\subsection{$4$-parameter family of surfaces} Proceeding the same as \cite[Section 6]{marques2014}, we reparametrize $C(v, t)$ to close up the parameter space. Choose a continuous extension $\bar {r}:\overline{\mathbb{B}^{3}}\to[0,\pi]$ of $r$. We define $\Phi : \overline{\mathbb{B}^{3}} \times [- 1, 1] \to \mathcal{Z}_2(\mathbb{B}^3 ;\partial \mathbb{B}^{3}; \mathbb{Z}_2)$ by:
\[
\Phi(v,t)
=
C\left(v,2\pi t+\gamma(|v|)\left(\frac{\pi}{2}-\bar r(v)\right)\right),
\qquad
(v,t)\in \overline{\mathbb B^3}\times[-1,1].
\]
where $\gamma:[0, 1] \to [0, 1]$ is given by:
\[
\gamma(s) = \begin{cases}
    0 & \text{if }s \leq \frac{1}{2}, \\
    2s - 1 & \text{if }s \geq \frac{1}{2}.
\end{cases}
\]
$\Phi$ is clearly continuous in the flat topology with no concentration of mass since $C(v, t)$ is. Now observe that on $\partial( \overline{\mathbb{B}^{3}} \times [-1, 1]) = \partial \mathbb{B}^{3} \times [-1, 1] \cup\overline{\mathbb{B}^{3}} \times \{-1\} \cup \overline{\mathbb{B}^{3}} \times \{1\}$, we have:
\begin{equation} \label{eq: Phi}
    \Phi(v, t) = \begin{cases} \partial [|\pi^{-1}\left(B^+_{2 \pi t + \frac{\pi}{2}}(\overline{Q}(v), 0)\right)|] & \text{if }t \in (-1, 1), \\
    0 & \text{if }t = -1, 1.  
    \end{cases}
\end{equation}
Again, $\Phi(v, 1) = \Phi(v, -1) = 0$ for all $v \in \overline{\mathbb{B}}^3$. so we can think of $\Phi$ as a map from $\overline{\mathbb{B}^4}$. \begin{comment}
Also note that by Remark \ref{rem: Area bound}, we have
\[
\max\{\mathbf{M}(\Phi(x)) : {x \in \partial\mathbb{B}^3} \} = \pi.
\]    
\end{comment}

\subsection{Construction of a $4$-sweepout} Now we choose $\Sigma \subset \mathbb{B}^3$ to be any orthogonal annulus with antipodal symmetry. Similar to \cite[Section 3.4.2]{nurser2016low}, for any $v \in \partial \mathbb{B}^3$, we have:
\begin{equation} \label{eq: boundary identificarions}
\begin{split}
    \Phi(-v, -t) & = \partial\left[\left|\pi^{-1}\left( B^+_{-2 \pi t + \frac{\pi}{2}}(\overline Q(-v),0)\right)\right|\right] \\
& = \partial\left[\left|\pi^{-1}\left(
 B^+_{2 \pi t + \frac{\pi}{2}}(-\overline Q(-v),0)\right)\right|\right].
\end{split}
\end{equation}
On the other hand, due to the antipodal symmetry of $\Sigma$ and the fact that $N(p) = - N(-p)$ for any $p \in \partial \Sigma$, we have
\begin{equation} \label{eq:Q}
    \overline{Q}(v) = - \overline{Q}(-v)
\end{equation}
for any $v \in \partial \mathbb{B}^3$. By (\ref{eq: boundary identificarions}), (\ref{eq:Q}), we can take the boundary identification
\[
\Phi(v, t) = \Phi(-v, -t) \quad \text{for }(v, t) \in \partial \mathbb{B}^{3} \times [-1, 1] \cup\overline{\mathbb{B}^{3}} \times \{-1\} \cup \overline{\mathbb{B}^{3}} \times \{1\}.
\] 
Thus, we can define a $\mathbb{RP}^{4}$-family of surfaces as
\[
\Phi^{\Sigma}_{4}([(v,t)]) = C\left(v, 2\pi t + \gamma(|v|) \left(\frac{\pi}{2}-\bar{r}(v)\right)\right).
\]
The $4$-parameter family $\Phi_4^{\Sigma}$ then becomes an Almgren-Pitts $4$-sweepout as shown below.
\begin{prop} \label{prop: 4sweepout}
Suppose $\Sigma \subset \mathbb{B}^3$ is an annulus with antipodal symmetry that intersects $\partial \mathbb{B}^3$ orthogonally. Then  $\Phi_4 = \Phi^{\Sigma}_{4}$ is an Almgren-Pitts $4$-sweepout $\Phi_{4}: \mathbb{RP}^4 \to \mathcal{Z}_2(\mathbb{B}^3, \partial \mathbb{B}^{3};\mathbb{Z}_2)$.
\end{prop}
\begin{proof}
Let \(\bar\lambda\in H^1(\mathcal Z_2(\mathbb B^3,\partial\mathbb B^3;\mathbb Z_2);\mathbb Z_2)\)
be the generator and \(\lambda\in H^1(\mathbb RP^4;\mathbb Z_2)\) be the generator.
It suffices to show that \(\Phi_4^*\bar\lambda=\lambda\), since then
\(\Phi_4^*(\bar\lambda^4)=(\Phi_4^*\bar\lambda)^4=\lambda^4\neq0.\)

Fix \(v_0\in\partial\mathbb B^3\) and consider the loop $\gamma(t)=[(v_0,t)]$ for $t\in[-1,1]$ where the endpoints are identified by the boundary relation. $\gamma$ represents the nontrivial element of \(H_1(\mathbb RP^4;\mathbb Z_2)\).
By the definition of $\Phi_4$, the family $\Phi_4\circ\gamma (t) = \partial [|\pi^{-1}\left(B^+_{2 \pi t + \frac{\pi}{2}}(\overline{Q}(v_0), 0)\right)|]$ is a standard sweepout of \(\mathbb B^3\) by
orthogonal spherical caps. Hence, \(\Phi_4^*\bar\lambda\neq0\). 
\begin{comment}
Therefore we have
\[
\langle \Phi_4^*\bar\lambda,[\gamma]\rangle=1.
\]    
\end{comment}

Since \(H^1(\mathbb RP^4;\mathbb Z_2)\cong\mathbb Z_2\), we have
\(\Phi_4^*\bar\lambda=\lambda\). Consequently,
\(\Phi_4^*(\bar\lambda^4)=\lambda^4\neq0\)
and \(\Phi_4\) is a \(4\)-sweepout.
    \end{proof}

Together with Proposition~\ref{prop:lowerfourthbound}, Proposition \ref{prop: 4sweepout} yields that
\begin{equation} \label{eq: lower bound of the 4-sweepout}
\pi<\omega_4(\mathbb B^3)\le \mathbf{L}_{AP}(\Lambda(\Phi_4^{\Sigma}))   
\end{equation}
where $\mathbf{L}_{AP}(\Lambda(\Phi_4^{\Sigma}))$ is the Almgren-Pitts min-max width of the $4$-sweepout $\Phi_4^{\Sigma}$. 

In the next section, we are going to take an appropriate choice of the annulus $\Sigma$ such that $\Phi_4^{\Sigma}$ becomes a Simon-Smith family so that we can use the topological control for Simon-Smith families to characterize the Simon-Smith min-max width of $\Phi_4^{\Sigma}$.

\section{$6$-parameter family of annuli} \label{sec: 6-parameter family}
In this section, we define a $6$-parameter Simon-Smith family of surfaces based on  previous constructions.

\subsection{Definitions} First, we take a specific family of annuli in $\mathbb{B}^3$. Consider $C:=\pi^{-1}(T^{2}) \cap \mathbb{B}^{3}$, where  $T^{2} \subset \mathbb{S}^{3}$ is the Clifford torus given by
 \[
        T^2 = \left\{x \in \mathbb{R}^4 : x^2_1 + x^2_2 = x^2_3 + x^2_4 = \frac{1}{2}\right\}\,.
\]
Observe that $C$ is a rotationally symmetric annulus with antipodal symmetry that intersects $\partial \mathbb{B}^3$ orthogonally. Also $C$ is uniquely determined by its rotational symmetry axis, the $e_{3}$-axis, and can be parametrized as
\[
C :=C_{e_3}
=
\left\{
\frac{1}{\sqrt{2}+\sin\beta}(\cos\alpha,\sin\alpha,\cos\beta):\alpha\in[0,2\pi),\ \beta\in[0,\pi]\right\}.
\]
Given any $a\in \mathbb{S}^2$, we can choose a rotation $R_a\in SO(3)$ such that $R_a(e_3)=a$, and define
\begin{equation}\label{eq:Ka-def}
C_a:=R_a(C_{e_3})\subset \mathbb{B}^3.
\end{equation} 
$C_a$ satisfies all the properties as $C_{e_3}$, except that the rotational symmetry axis is changed to the line $\{ ta: t \in \mathbb{R} \}$. Note that $C_a$ is independent of the choice of $R_a$. Moreover, $C_{-a}=C_a$ as unoriented surfaces of $\mathbb{B}^3$. For notational convenience, we will name any such surface $C_a$ the \emph{Clifford annulus}. We denote the family of oriented Clifford annuli by
\begin{equation}
\tilde{\mathcal{C}} := \{C_a\}_{a\in \mathbb{S}^2} \cong \mathbb{S}^2.  
\end{equation}
For each $a \in \mathbb{S}^{2}$, identifying $\{ C_a, C_{-a} \}$ as the unoriented Clifford annulus $C_{[a]}$, $\tilde{\mathcal{C}}$ descends to the family of unoriented Clifford annuli,
\begin{equation}
\mathcal{C} = \{C_{[a]} \}_{[a] \in \mathbb{RP}^2} \cong \mathbb{RP}^2.
\end{equation}
\begin{notation} \label{subsec: orientation of K}
 For each $a\in \mathbb{S}^2$, the oriented Clifford annulus $C_a$ separates $\mathbb{B}^3$ into two components, $A_a$ and $A_a^*$. Throughout the paper, we use the convention that $A_a$ is the topologically solid ball region, and $A_a^{*}$ is the topologically torus region, and we choose the orientation on $C_a$ such that the outer unit normal $N_a$ points towards $A_a^*$. The antipodal map $\zeta: a\mapsto -a$ leaves the underlying surface unchanged. More importantly, $\zeta$ preserves the normal field $N_a$, and $A_{a} = A_{-a}$, $A_a^*= A_{-a}^{*}$.
\end{notation} 
\begin{rem} \label{rem: Continuous orientation}
Since the antipodal map preserves the orientation of a Clifford annulus, we can assign an orientation to each unoriented Clifford annulus in $\mathcal{C}$ continuously. This is different from \cite{chu2024existence}, as there is no continuous way to assign an orientation to the space of unoriented Clifford tori. Thus, our parameter space is $\mathbb{RP}^4 \times \mathbb{RP}^2$, the product of $4$-parameter family $\mathbb{RP}^4$ (recall Section \ref{sec:canonicalfamily}) and the family of unoriented Clifford annuli, $\mathcal{C} \cong \mathbb{RP}^2$. We also note that due to the same reasoning, the $\mathbb{RP}^2$ family of Clifford annuli is not a $2$-sweepout of $\mathbb{B}^3$.
\end{rem}

\begin{comment}
By Remark \ref{rem: Continuous orientation}, we can assign an orientation to each $C_{[a]} \in \mathcal{C}$ such that $\mathcal{C}$ becomes a $\mathbb{RP}^2$ family of oriented smooth surfaces, that is continuous in the smooth topology. Without loss of generality, we also assume that each $C_{[a]}$ has the same orientation as Notation \ref{subsec: orientation of K} above.    
\end{comment}
\subsection{Topological characterization of $6$-parameter family} 
Now we extend the $4$-parameter family from Section \ref{sec:canonicalfamily} to a $6$-parameter family. For $[a]\in\mathbb{RP}^2$, applying the construction of the previous section to $C_{[a]}$, we obtain a canonical family
\[
\Phi^{C_{[a]}}_4:\mathbb{RP}^4\to \mathcal Z_2(\mathbb B^3,\partial\mathbb B^3;\mathbb Z_2).
\]
We then define
\begin{equation} \label{def: 6 parameter family}
\Phi_6:\mathbb{RP}^4\times\mathbb{RP}^2\to \mathcal Z_2(\mathbb B^3,\partial\mathbb B^3;\mathbb Z_2), \qquad \Phi_6(z,[a])=\Phi^{C_{[a]}}_4(z).    
\end{equation}

\begin{lem}\label{lem:Ca-continuous}
The map
\[
(z,[a])\mapsto \Phi_6(z,[a])
\]
from $\mathbb{RP}^4\times\mathbb{RP}^2$ to $\mathcal Z_2(\mathbb B^3,\partial\mathbb B^3;\mathbb Z_2)$ is continuous in the flat topology.
\end{lem}

\begin{proof}
It follows from the continuity of $\Phi^{C_{[e_3]}}_4$ by Theorem \ref{thm: flattop conti blow-up} and the smooth dependence of $C_{[a]}$ on $[a] \in \mathbb{RP}^2$.
\end{proof}

\begin{thm}\label{thm:genuszerotwoboundary}
        The $6$-parameter family 
        \[
            \Psi \coloneqq \Phi_6:Y \rightarrow \mathcal{S}^*(\mathbb{B}^{3})
        \]
         is a free boundary Simon-Smith family parametrized by $Y:=\mathbb{RP}^{4} \times \mathbb{RP}^{2}$. Moreover, $\Psi$ satisfies the topological bound $(0,1)$.
    \end{thm}
\begin{comment}
We will prove Theorem \ref{thm:genuszerotwoboundary} in Section \ref{sec : topologysixparameter} by using the channel surface structure on both sides of the Clifford annulus. The Clifford annulus inherits a two-sided channel surface structure from the Clifford torus, and this structure is preserved under conformal deformations. This will help us characterize the topology of the $6$-parameter family $\Phi
_6$ and show Theorem \ref{thm:genuszerotwoboundary}.    
\end{comment}

\section{Topology of surfaces in the $6$-parameter family} \label{sec : topologysixparameter}
In this section, we prove Theorem \ref{thm:genuszerotwoboundary}. By Remark \ref{rem: Continuous orientation}, it suffices to consider the Clifford annulus $\Sigma = C = C_{[e_3]}$ and study $\Sigma_{(v, t)}$. Recall that $\Sigma$ can be parametrized by
\[
\Sigma=\left\{\frac{1}{\sqrt{2}+\sin\beta}(\cos\alpha,\sin\alpha,\cos\beta):\alpha\in[0,2\pi),\ \beta\in[0,\pi]\right\}.
\]
By Section \ref{subsec: orientation of K}, we take the outer unit normal $N$ to point into $A^{*} = A^{*}_{e_3}$, the solid torus region, and $A = A_{e_3}$ is the solid ball region. Before proving Theorem \ref{thm:genuszerotwoboundary}, we describe interior and exterior channel surface structures of $F_{v}(\Sigma)$. From the interior region $A$, $\Sigma$ is a channel surface with channel spheres $\mathcal S_\beta\subset \mathbb R^3$ with the center
\[
C_{\mathbb{R}^{3}}(\beta)=\left(0,0,\frac{\sqrt2\cos\beta}{1+\sqrt2\sin\beta}\right)
\]
and radius
\[
R_{\mathbb{R}^{3}}(\beta)=\frac{1}{1+\sqrt2\sin\beta},\qquad \beta\in[0,\pi].
\]
Then $\Sigma$ is the envelope of the one-parameter family
$\{\mathcal S_\beta\}_{\beta\in[0,\pi]}$.
Equivalently, $\mathcal S_\beta$ is given by
\[
\left|x-\left(0,0,\frac{\sqrt2\cos\beta}{1+\sqrt2\sin\beta}\right)\right|^2=\frac{1}{(1+\sqrt2\sin\beta)^2}.
\]

On the other hand, from the exterior region $A^{*}$, $\Sigma$ is an envelope of spherical caps that are orthogonal to $\partial \mathbb{B}^{3}$. More precisely, it is an envelope of the following $1$-parameter family of orthogonal spherical caps
\begin{equation} \label{channelcapdef}
    \mathcal{H}_\alpha=\overline{\mathbb B}^3\cap\left\{x\in\mathbb R^3:\left|x-\sqrt2(\cos\alpha,\sin\alpha,0)\right|^2=1\right\}.
\end{equation}
Conformal maps preserve the structure of channel surfaces. Viewing from the interior region $\pi(F_v(A))$, $\pi(F_{v}(\Sigma))$ is an envelope of geodesic spheres $\pi(F_{v}(\mathcal{S}_{\beta}))$; from the exterior region $\pi(F_v(A^{*}))$, it is an envelope of half geodesic spheres $\pi(F_{v}(\mathcal{H}_{\alpha}))$. 
In the following lemma, we characterize the centers of these geodesic spheres and half geodesic spheres. This good property of the center curves is the main motivation we work with the Clifford annulus. For notational convenience, Let us call a surface in $\mathbb{S}^3$ which is an envelope of channel hemispheres \emph{half-channel surface}.
 
\begin{lem} \label{lem:halfchannel} Fix any $v \in \mathbb{B}^3$. For $\alpha \in [0, 2 \pi)$, the center $C_{\mathbb{S}^{3}}(\alpha)$ of $\pi(F_{v}(\mathcal{H}_{\alpha}))$ is on $\partial \mathbb{S}^{3}_{+}$. In particular, for each $\alpha$, $\pi(F_{v}(\mathcal{H}_{\alpha}))$ is a half geodesic sphere. 

For $\beta \in (0, \pi)$, the center $C_{\mathbb{S}^3}(\beta)$ of $\pi(F_v(\beta))$ is on $\mathbb{S}_+^3$ and  $C_{\mathbb{S}^3}(\mathcal{S}_{\beta}) \in \partial \mathbb{S}_+^3$ for $\beta =0, \pi$.
\end{lem}
\begin{proof}
By (\ref{channelcapdef}), for each $\alpha \in [0, 2 \pi)$, $\mathcal{H}_{\alpha}$ is a spherical cap which is orthogonal to $\partial \mathbb{B}^{3}$. By the conformality of $F_v$ and $\pi$, $\pi(F_{v}(\mathcal{H}_{\alpha})) \subset \overline{\mathbb{S}_+^3}$ is orthogonal to $\partial \mathbb{S}^{3}_{+}$ and $\pi(F_{v}(\mathcal{H}_{\alpha}))$ is a half geodesic sphere with center on $\partial \mathbb{S}^{3}_{+}$. Similarly, the center curve $\{\pi(C_{\mathbb{R}^3(\beta))} \}_{\beta \in [0, \pi]} \subset \mathbb{S}^3_+$ has its interior in $\mathbb{S}^3_+$ and orthogonal to $\partial \mathbb{S}^3_+$ at the two endpoints.
\end{proof} 
\subsection{Channel surface structure in the interior region $A$} We first analyze the spherical radius of channel spheres of $\Sigma_{(v,0)}=F_{v}(\Sigma)$. This will give us the topology of spherical distance set $\Sigma_{(v,t)}$ when $t<0$, i.e. when $\Sigma_{(v,0)}$ is shrinking to the center curve.  We denote the center and radius of $F_{v}(\mathcal{S}_{\beta})$ by $C_{\mathbb{R}^{3},v}(\beta)$ and $R_{\mathbb{R}^{3},v}(\beta)$, respectively. Also, denote the center and radius of $\pi(F_{v}(\mathcal{S}_{\beta}))$ by $C_{\mathbb{S}^{3},v}(\beta)$ and $R_{\mathbb{S}^{3},v}(\beta)$, respectively. We abbreviate the notation by omitting $\mathbb{S}^{3}$ as $C_{v}(\beta)$ and $R_{v}(\beta)$. Now we compute $C_{v}(\beta)$ and $R_{v}(\beta)$.
\begin{lem} \label{lem:betaradius}
    The radius of $R_{v}(\beta)$ of $\pi(F_{v}(\mathcal{S}_{\beta}))$ is
    \[
     R_{v}(\beta) = \arccot \left( \frac{1+|v|^2-2\sqrt2\,v_3\cos\beta}{1-|v|^2} \right).
    \]
In particular, if $v_3\neq0$, then $R_v$ is strictly monotone on $[0,\pi]$, hence its maximum and minimum are attained at the endpoint channel spheres $\beta=0,\pi$. If $v_3=0$, then $R_v$ is a constant.
\end{lem}
\begin{proof}
Observe that
\begin{equation} \label{eq: channel sphere of the Clifford Torus}
    \pi(\mathcal {S}_\beta)  =\{y\in \mathbb{S}^3:y\cdot w(\beta)=\cos(\pi/4)\}  = \partial B_{\frac{\pi}{4}}(w(\beta)),
\end{equation}
where
\(w(\beta)=(0,0,\cos\beta,\sin\beta)\in \mathbb{S}^3\) for \(\beta \in [0, \pi]\).

Let us take $\tilde{v}=(v,0) \in \mathbb{R}^4$. By (\ref{eq: channel sphere of the Clifford Torus}) and the fact that $(F_v)^{-1} = F_{-v}$, $z\in (\pi\circ F_v)(\mathcal{S}_\beta)$ iff
\[
(\pi \circ F_{-v} \circ \pi^{-1})(z)\cdot w(\beta)=\frac1{\sqrt2}.
\]
By direct computations, we have
\[
(\pi \circ F_{-v} \circ \pi^{-1})(z)=\frac{(1-|v|^{2})z+2(1+\tilde{v}\cdot z)\tilde{v}}{1+2\tilde{v}\cdot z+|v|^{2}},
\]
and hence
\begin{equation} \label{eq: channel geodeisc sphere}
z\cdot\left((1-|v|^{2})w(\beta)+2\left(v_3\cos\beta-\frac1{\sqrt2}\right)\tilde{v}\right)=\frac{1+|v|^{2}}{\sqrt2}-2v_3\cos\beta.
\end{equation}
Observe that 
\begin{equation*}
\left|(1-|v|^{2})w(\beta)+2\left(v_3\cos\beta-\frac1{\sqrt2}\right)\tilde{v}\right|^2 = 4\left(v_3 \cos{\beta} - \frac{1 + |v|^2}{2 \sqrt{2}}\right)^2 + \frac{1}{2}(1 - |v|^2)^2 > 0, \quad \text{for }(v, \beta) \in \mathbb{B}^3 \times [0, \pi].
\end{equation*}
Thus, \((\pi\circ F_v)(\mathcal{S}_{\beta})\) is the geodesic sphere given by
\[
(\pi\circ F_v)(\mathcal{S}_{\beta})=\{y\in \mathbb{S}^3:y\cdot C_v(\beta)=\cos R_v(\beta)\},
\]
where
\begin{equation} \label{eq: center of the interior channel surface}
    C_v(\beta) = \frac{(1-|v|^{2})w(\beta)+2\left(v_3\cos\beta-\frac1{\sqrt2}\right)\tilde{v}}{\left|(1-|v|^{2})w(\beta)+2\left(v_3\cos\beta-\frac1{\sqrt2}\right)\tilde{v}\right|}
\end{equation}
and
\begin{equation}
\cos R_v(\beta)=\frac{\frac{1+|v|^2}{\sqrt2}-2v_3\cos\beta}{\left|(1-|v|^{2})w(\beta)+2\left(v_3\cos\beta-\frac1{\sqrt2}\right)\tilde{v}\right|}.    
\end{equation}
In particular, from (\ref{eq: center of the interior channel surface}), it is straightforward to see that $C_v(\beta) \in \mathbb{S}_+^3$ for $\beta \in (0, \pi)$, and $C_v(0), C_v(\pi) \in \partial \mathbb{S}_+^3$.
Then we can compute
\[
\cot R_v(\beta)=\frac{\cos R_v(\beta)}{\sin R_v(\beta)}=\frac{1+|v|^2-2\sqrt2\,v_3\cos\beta}{1-|v|^2}.
\]
Now the proposition follows from the monotonicity of $\cos \beta$.
\end{proof}

\subsection{Half-channel surface structure in the exterior region $A^{*}$} 
Now we analyze the half-channel surface structure and see the behavior of the spherical radius of the channel half geodesic spheres of $F_{v}(\Sigma)$ from $A^{*}$. This gives the topology of spherical distance set $\Sigma_{(v,t)}$ when $t>0$, i.e., when the annulus $\Sigma_{(v,0)}$ is expanding. Note that  $\pi(F_{v}(\mathcal{H}_{\alpha}))$ is a hemisphere by Lemma \ref{lem:halfchannel}. Let us denote the center and radius of $\pi(F_{v}(\mathcal{H}_{\alpha}))$ by $C_{\mathbb{S}^{3},v}(\alpha)$ and $R_{\mathbb{S}^{3},v}(\alpha)$, respectively. We abbreviate the notation as before to $C_{v}(\alpha)$ and $R_{v}(\alpha)$ and we compute $C_{v}(\alpha)$ and $R_{v}(\alpha)$. By the same computations as Lemma \ref{lem:betaradius}, we have the following lemma.
\begin{lem} \label{lem:alpharadius}
    The radius of $R_{v}(\alpha)$ of $\pi(F_{v}(\mathcal{H}_{\alpha}))$ is
    \[
     R_{v}(\alpha) = \arccot \left( \frac{1+|v|^2-2\sqrt2\,(v_1\cos\alpha+v_2\sin\alpha)}{1-|v|^2} \right).
    \]
In particular, if $(v_{1},v_{2}) \neq (0,0)$, then $R_{v}(\alpha)$ has exactly one maximum and one minimum point in $[0, 2 \pi)$, and is monotonic between these two extreme points. Otherwise, \(R_v\) is a constant.
\end{lem}

\subsection{$6$-parameter family of genus zero surfaces with at most two boundary components} We prove the following topological control of the $4$-parameter Simon-Smith family (recall Section \ref{sec:canonicalfamily}) based on the channel surface structures we discussed above, and the monotonicity of the radii $R_{v}(\alpha)$ and $R_{v}(\beta)$ in Lemma \ref{lem:alpharadius} and \ref{lem:betaradius}, respectively. 
   \begin{prop}\label{lem:behaviourSigmavt}
        Let $(v,t)\in {\mathbb B^3}\times [-\pi,\pi]$. We have the topological characterizations of the set
        \[
            \Sigma_{(v,t)}:=\partial \left(\pi^{-1}\left( \{x \in \mathbb{S}^3_+ : d_v(x) < t \} \right) \right) \cap \mathbb{B}^3\,.
        \]
        
    If $v=0$, then $\Sigma_{(v,t)}$ is an annulus for $t\in (-\pi/4,\pi/4)$, a great circle for $t=\pi/4$, a straight line for $t= - \pi/4$, and empty for other $t$.
             
    If $v\ne 0$ and $v_{3} =0$, then there exist
        $-\pi<t_1<0<t_2<t_{3}<\pi,$ depending on $v$,
        such that:
        \begin{enumerate}[label=\normalfont(\arabic*)]
            \item For $t_1<t<t_2$, $\Sigma(v,t)$ is an embedded annulus.
            \item For $t=t_1$, $\Sigma_{(v,t)}$ is a straight line.
            \item For $t=t_2$, $\Sigma_{(v,t)}$ is a topological disk with two points identified such  that it is smooth except at a point.
            \item For $t_2<t<t_3$, $\Sigma_{(v,t)}$ is a topological disk, smooth except at two points.
            \item For $t= t_3$, $\Sigma_{(v,t)}$ is a point on $\partial \mathbb{B}^{3}$.
            \item For $t< t_1$ and $t> t_3$, $\Sigma_{(v,t)}$ is empty.
        \end{enumerate}
        
        If $v\ne 0$ and $(v_{1},v_{2})=(0,0)$, then there exist
        $-\pi<t_1<t_2<0<t_3<\pi,$ depending on $v$,
        such that:
        \begin{enumerate}[label=\normalfont(\arabic*)]
            \item For $t_2<t<t_3$, $\Sigma_{(v,t)}$ is an embedded annulus.
            \item For $t=t_2$, $\Sigma_{(v,t)}$ is a topological disk which is smooth except at a point on $\partial\mathbb{B}^3$.
            \item For $t=t_3$, $\Sigma_{(v,t)}$ is an embedded circle.
            \item For $t_1<t<t_2$, $\Sigma_{(v,t)}$ is a topological disk which is smooth except at a point.
            \item For $t= t_1$, $\Sigma_{(v,t)}$ is a point on $\partial \mathbb{B}^{3}$.
            \item For $t< t_1$ and $t> t_3$, $\Sigma_{(v,t)}$ is empty.
        \end{enumerate}
        
        Otherwise, there exist
        \[
            -\pi<t_1<t_2<0<t_3<t_4<\pi\,,
        \] depending on $v$, such that:
        \begin{enumerate}[label=\normalfont(\arabic*)]
            \item For $t_2<t<t_3$, $\Sigma_{(v,t)}$ is an embedded annulus.
            \item For $t=t_2$, $\Sigma_{(v,t)}$ is a topological disk which is smooth except at a point on $\partial\mathbb{B}^3$.
            \item For $t=t_3$, $\Sigma_{(v,t)}$ is a topological disk with two points identified such  that it is smooth except at a point. 
            \item For $t_1<t<t_2$, $\Sigma_{(v,t)}$ is a topological disk which is smooth except at a point.
            \item For $t_3<t<t_4$, $\Sigma_{(v,t)}$is a topological disk, smooth except at two points.
            \item For $t= t_1, t_{4}$, $\Sigma_{(v,t)}$ is a point on $\partial \mathbb{B}^{3}$.
            \item For $t< t_1$ and $t> t_4$, $\Sigma_{(v,t)}$ is empty.
        \end{enumerate}
    \end{prop}
    
\begin{proof}
Since the inverse stereographic projection $\pi$ preserves the topology, it suffices to show the proposition for $\pi(\Sigma(v, t))$.

\textbf{Case $1$:} If $t > 0$, by our choice of the orientation for $\Sigma$, this corresponds to expanding the annulus $\pi(F_v(\Sigma))$. As $\pi(F_v(\Sigma))$ can be viewed as a half-channel surface from the exterior region $\pi(F_v(A))$, it corresponds to shrinking the half geodesic spheres:
\begin{equation} \label{eqn:expandingchannel}
    \partial B^+_{R_{v}(\alpha)-t}(C_{v}(\alpha))=\{y\in \mathbb{S}^3_{+}:y\cdot C_v(\alpha)=\cos (R_v(\alpha)-t)\}.
\end{equation}
By the monotonicity of the radius $R_{v}(\alpha)$ from Lemma \ref{lem:alpharadius}, and since the monotonicity is preserved over the extension by the constant radius, arguing the same as \cite[Proposition 5.2 Claim 5.3]{chu2024existence}, we obtain the desired topological characterization.

\textbf{Case $2$:} If $t < 0$, by our choice of the orientation for $\Sigma$, this corresponds to shrinking the annulus $\pi (F_v(\Sigma))$ towards the interior region $\pi(F_v(A^*))$. Thus, we track the topology of $\pi(\Sigma_{(v, t)})$ by shrinking the geodesic spheres:
\begin{equation} \label{eqn:shrinkingchannel}
\partial B_{R_{v}(\beta)+t}(C_{v}(\beta))=\{y\in \mathbb{S}^3:y\cdot C_v(\beta)=\cos (R_v(\beta)+t)\}.
\end{equation}
By Lemma \ref{lem:betaradius}, arguing similarly as Case $1$, we obtain the desired topological characterization when $t<0$.
\end{proof}
\begin{comment}
Let us now prove Theorem \ref{thm:genuszerotwoboundary}.    
\end{comment}
\begin{proof}[Proof of Theorem \ref{thm:genuszerotwoboundary}]
    By Proposition \ref{lem:behaviourSigmavt} and the definition of $\Psi$, any surface in $\Psi$ is one of the following:
    \begin{itemize}
        \item properly embedded annuli,
        \item topological disk with two points identified, which is smooth except at a point,
        \item topological disk smooth except at a point,
        \item embedded circle,
        \item straight line segment,
        \item empty set,
    \end{itemize}
which satisfies the topological bound $(0, 1)$ and conditions \normalfont{(i)} - \normalfont{(iv)} of Definition \ref{def:fbssfamily}. Arguing the same as \cite[Section 5.3.2]{chu2024existence}, we can show $\Psi$ also satisfies Definition \ref{def:fbssfamily} \normalfont{(v)}. Thus, $\Psi$ is a free boundary Simon-Smith family with the topological bound $(0, 1)$ and we have finished the proof.
\end{proof}

\section{The upper bound of $\omega_4(\mathbb{B}^3)$} \label{sec: upper bounda of the 4th-width}
Now we can finally show Theorem \ref{intro thm:upperbound4width}. By extracting our $4$-parameter family from the $6$-parameter family in Theorem \ref{thm:genuszerotwoboundary}, and applying Simon-Smith min-max theorem (Theorem \ref{thm:simon-smith min-max}), we obtain the following existence theorem of free boundary minimal annulus in the unit ball $\mathbb{B}^{3}$.
\begin{thm} \label{multioneannulus} In the Euclidean unit ball $\mathbb{B}^{3}$, $\Phi_4 \coloneqq \Phi_4^{C_{[e_3]}}:\mathbb{RP}^4\to \mathcal{S}^{*}(\mathbb{B}^{3})$ is a free boundary Simon-Smith family with the topological bound $(0,1)$ such that the Simon-Smith min-max width $\mathbf{L}_{SS}(\Lambda(\Phi_4))$ is attained by a multiplicity one properly embedded free boundary minimal annulus.
\begin{proof}
By \cite[Lemma 2.4]{FraserLi2014} and \cite[Lemma C.1]{chodosh2017minimal}, any properly embedded free boundary minimal surface in $\mathbb{B}^{3}$ is connected and orientable. Since every Simon-Smith family is an Almgren-Pitts sweepout and $\Phi_4$ is a $4$-sweepout, we have
\begin{equation} \label{eq: inequality for the min-max width}
\mathbf{L}_{SS}(\Lambda(\Phi_4)) \geq \mathbf{L}_{AP}(\Lambda(\Phi_4)) \geq \omega_4(\mathbb{B}^3) > \pi   
\end{equation}
by (\ref{eq: lower bound of the 4-sweepout}).
Together with the multiplicity one theorem \cite[Theorem 1.3]{sarnataro2026existence} and the fact that every free boundary minimal surface in $\mathbb{B}^3$ is unstable, $\mathbf{L}_{SS}(\Lambda(\Phi_4))$ is attained by a multiplicity one free boundary minimal surface $\Gamma$.
Since $\Phi_4$ has a topological bound $(0,1)$ by Theorem \ref{thm:genuszerotwoboundary}, by the topological control \cite[Theorem 1.8]{franz2023topological}, $\Gamma$ is either a minimal disk or annulus.

Since $\mathcal{H}^2(\Gamma) = \mathbf{L}_{SS}(\Lambda(\Phi_4)) > \pi$,
by Nitsche's uniqueness result of free boundary minimal disk in \cite{nitsche1985stationary}, we can rule out the case that $\Gamma$ is a free boundary minimal disk. Hence $\Gamma$ is a multiplicity one properly embedded free boundary minimal annulus.
\end{proof}
\end{thm}
\begin{comment}
By combining Theorem \ref{multioneannulus} with an upper bound of Morse index for min-max free boundary minimal surfaces \cite{franz2023equivariant}, and the characterization of free boundary minimal annuli by Morse index \cite[Corollary 7.3]{devyver2019index} (also see \cite[Corollary 3.11]{tran2016index}),     
\end{comment}
Now we prove that $\Gamma$ must be the critical catenoid below. Notice that instead of trying to control the areas of surfaces in the $4$-parameter family $\Phi_4$ to derive an upper bound for $\omega_4(\mathbb{B}^3)$, our arguments here only rely on the topology and Morse index bounds.
\begin{thm} \label{thm: width is the critical catenoid}
    The min-max surface $\Gamma$ in Theorem \ref{multioneannulus} above is the critical catenoid $\mathbb{K}$ in $\mathbb{B}^{3}$.
    \begin{proof}
    By \cite[Theorem 1.10]{franz2023equivariant},  $\operatorname{Index}(\Gamma) \le 4$. On the other hand, by \cite[Theorem 1.1]{devyver2019index}, every free boundary surface other than the flat disk must have index $\geq 4$. Thus, $\operatorname{index}(\Gamma) = 4$. Now by the characterization of minimal annuli by Morse index \cite[Corollary 7.3]{devyver2019index} (also see \cite[Corollary 3.11]{tran2016index}), we obtain that $\Gamma$ is the critical catenoid.
    \end{proof}
\end{thm}
\begin{proof}[Proof of Theorem \ref{intro thm:upperbound4width}] By (\ref{eq: inequality for the min-max width}) and Theorem \ref{thm: width is the critical catenoid} , we have
\begin{equation} \label{eqn:upper4width}
    \pi < \omega_{4}(\mathbb{B}^{3}) \le \mathbf{L}_{SS}(\Lambda(\Phi)) = Area(\mathbb{K}), 
\end{equation}
which are the desired bounds we want.
\end{proof}

\section{Boundary blow-up} \label{sec :boundary blow-up}
In this section, we prove Theorem \ref{thm: flattop conti blow-up} to extend the canonical family to $\overline{\mathbb{B}^3} \times [-\pi,\pi]$ by the boundary blow-up argument. In our setting, the blow-up limits of (relative) cycles are spherical caps orthogonal to $\partial \mathbb{B}^3$, instead of geodesic spheres in \cite{marques2014}. So we will modify the boundary blow-up argument of Marques-Neves \cite[Section 5]{marques2014} to  adapt to our proof. Throughout this section, we consider an orthogonal properly embedded surface $\Sigma \subset \mathbb{B}^3$, and the reader may refer to \S \ref{subsec: general free  boundary surface notation} for notation and conventions we put on $\Sigma$. We now prove Theorem \ref{thm: flattop conti blow-up}.

\subsection{Preliminary results} Given sets $A,B \subset \mathbb{R}^n$, the symmetric difference is denoted by
\[
A\,\Delta \,B=(A\setminus B)\cup (B\setminus A).
\]

Recall the definition of the map $\Lambda$ in (\ref{eq: def of Lambda}). If $v_n \in \mathbb{B}^3$ is a sequence converging to $p \in \partial \Sigma$, then for all $n$ sufficiently large, there are unique $p_n \in \partial \Sigma$ and $s_n \in D^2_+(3 \epsilon)$ so that $v_n = \Lambda(p_n, s_n)$. Necessarily, $p_n$ tends to $p$ and $s_n$ tends to zero. By passing to a subsequence, we can also assume that
\[
 \lim_{n \to \infty} \frac{s_{n_2}}{s_{n_1}} = k \in [ - \infty, + \infty].
\]

Recall the definition of $\overline{Q}_{p, k}$ in (\ref{eq: Q^bar_p, k}) and $\overline{r}_k$ in (\ref{eq: r^bar_k}), we note that
\begin{equation} \label{eq: ball centered at Q^bar_p, k }
B_{\overline{R}_k}^4(\overline{Q}_{p, k}, 0) \cap \mathbb{S}^3 = B_{\overline{r}_k}(\overline{Q}_{p, k}, 0), \quad \text{ where} \quad \overline{R}_k = \sqrt{2\bigl(1 - \frac{k}{\sqrt{k^2 + 1}}\bigr)}
\end{equation}

\begin{prop}[Convergence of the free boundary canonical family]
\label{prop:convergence.sets.fb}
Consider a sequence
\[
        (v_n,t_n)\in \mathbb{B}^3 \times[- \pi,\pi]
\]
converging to \((v,t)\in \overline{\mathbb{B}^3}\times[- \pi, \pi]\). Then the following hold.

\begin{enumerate}[label=\normalfont(\roman*)]
\item If \(v\in \mathbb{B}^3\), then
\[
        \lim_{n\to\infty}
        \operatorname{vol}\left(
        A_{(v_n,t_n)}\Delta A_{(v,t)}
        \right)=0.
\]

\item If $v \in U$, then
\[
        \lim_{n\to\infty}
        \operatorname{vol}\left(
        A_{(v_n,t_n)}
        \Delta \,
        \pi^{-1}\left(B^+_{\pi+t}(v, 0)\right)
        \right)=0.
\]
and, given any \(\delta>0\),
\[
        \Sigma_{(v_n,t_n)}
        \subset
        \pi^{-1} \left(\overline{B^+}_{\pi + t + \delta}(v, 0) \setminus B^+_{\pi +t - \delta}(v, 0) \right)  \quad 
        \text{for all } \text{ sufficiently large } n.
\]
\item If $v \in U^*$, then
\[
\lim_{n \to \infty} \operatorname{vol}(A_{(v_n, t_n)} \Delta \pi^{-1}(B^+_t(-v, 0))) = 0
\]
and, given any \(\delta>0\),
\[
        \Sigma_{(v_n,t_n)}
        \subset
        \pi^{-1}\left(\overline{B^+}_{t + \delta}(-v, 0) \setminus B^+_{t - \delta}(-v, 0) \right)  \quad 
        \text{for all } \text{ sufficiently large }n.
\]
\item Suppose \(v=p\in \partial\Sigma\) and
\[
        v_n=\Lambda\bigl(p_n,(s_{n_1},s_{n_2})\bigr),
        \qquad
        \lim_{n\to\infty}\frac{s_{n_2}}{s_{n_1}}
        =k\in[-\infty,+\infty].
\]
Then
\[
        \lim_{n\to\infty}
        \operatorname{vol}\left(
        A_{(v_n,t_n)}
        \Delta
        \pi^{-1}(B^+_{\overline r_{k}+t}
        (\overline Q_{p,k}, 0))
        \right)=0.
\]
Moreover, given any \(\delta>0\),
\[
        \Sigma_{(v_n,t_n)}
        \subset
        \pi^{-1}\left( \overline{B^+}_{\overline r_{k}+t+\delta}
        (\overline Q_{p,k}, 0)
        \setminus
        B^+_{\overline r_{k}+t-\delta}
        (\overline Q_{p,k}, 0) \right)
        \quad
        \text{for all } \text{ sufficiently large }n.
\]
\end{enumerate}
\end{prop}
\begin{proof}
Proposition \ref{prop:convergence.sets.fb} (i) follows from the same arguments in \cite[Proposition 5.1 (i)]{marques2014} applied to $\pi(\Sigma_{v}), \pi(A_{v, t})$.

We now prove Proposition \ref{prop:convergence.sets.fb} (ii). Take a small ball $B^3_r(v) \cap \mathbb{B}^3 \subset A$. Observe that
\[
\begin{split}
    F_{v_n} (x) & = \frac{(1 - |v_n|^2)x - (1 - 2 \langle v_n, x \rangle + |x|^2) v_n}{1 - 2 \langle v_n, x \rangle + |x|^2 |v_n|^2} \\
    & = (1 - |v_n|^2)\frac{\frac{x}{|x|^2} - v_n}{|\frac{x}{|x|^2} - v_n|^2} - v_n, \quad \text{for }x \neq 0, \\
F_{v_n}(0) & = - v_n.
\end{split}
\]
Thus, $F_{v_n}$ will push every point in $\mathbb{B}^3$ that is outside a small neighborhood of $v \in \partial \mathbb{B}^3$ towards $-v$. As a result, given $\delta > 0$, there exists $n_0 \in \mathbb{N}$ such that for all $n \geq n_0$,
\begin{equation*}
    \mathbb{B}^3 \setminus B^3_{\tan{(\frac{\delta}{2})}} \left(- \sec{\left(\frac{\delta}{2}\right)} v\right) \subset F_{v_n}(B^3_r(v) \cap \mathbb{B}^3) \subset F_{v_n}(A) = A_{(v_n, 0)} \quad \text{and} \quad |t_n - t| \leq \frac{\delta}{2}.
\end{equation*}
where $B^3_{\tan{(\frac{\delta}{2})}} (- \sec{(\frac{\delta}{2})} v)$ intersects $\partial \mathbb{B}^3$ orthogonally and contains $-v$.
In particular, 
\[
\Sigma_{v_n} \subset B^3_{\tan{(\frac{\delta}{2}})} \left(- \sec{\left(\frac{\delta}{2}\right)} v\right) \cap \mathbb{B}^3
\]
for all $n \geq n_0$. Hence we have for all $n \geq n_0$,
\begin{equation} \label{eq: blow up Case 2 (i)}
     \mathbb{S}_+^3 \setminus B_{\frac{\delta}{2}}(-v, 0) = \pi\left(\mathbb{B}^3 \setminus B^3_{\tan{(\frac{\delta}{2})}} \left(- \sec{\left(\frac{\delta}{2}\right)} v\right)\right)\subset \pi \circ F_{v_n}(B^3_r(v)) \subset \pi (A_{(v_n, 0)}),
\end{equation}
and 
\begin{equation} \label{eq: blow up Case 2 (ii)}
    \pi(\Sigma_{v_n}) \subset B^+_{\frac{\delta}{2}}(-v, 0).
\end{equation}

If $t \geq 0$, then from (\ref{eq: blow up Case 2 (i)}) and the triangle inequality we have, for all $n \geq n_0$,
\[
\mathbb{S}_+^3 \setminus B_{\delta}(-v, 0) \subset \pi \left( A_{(v_n, -\frac{\delta}{2})} \right) \subset \pi \left(A_{(v_n, t_n)} \right).
\]
Hence, because $B_{\pi + t}(v, 0) = \mathbb{S}^3$, 
\begin{equation*}
   \pi \left(A_{(v_n, t_n)}\right) \Delta  B^+_{\pi + t}(v, 0) \subset B^+_{\delta}(-v, 0)
\end{equation*}
for all $n \geq n_0$. Hence we have
\begin{equation} \label{eq: blow up Case 2 (iii)}
A_{(v_n, t_n)} \Delta \pi^{-1}(B^+_{\pi + t}(v, 0)) \subset \pi^{-1}(B^+_{\delta}(-v, 0)) = B^3_{\tan{\delta}}(- \sec{\delta} \, v) \cap \mathbb{B}^3 
\end{equation}
for all $n \geq n_0$. Notice that if $t > 0$, then (\ref{eq: blow up Case 2 (i)}) implies that $A_{(v_n, t_n)} = \mathbb{B}^3$ and hence $\Sigma_{(v_n, t_n)} = \emptyset$ for any sufficiently large $n \in \mathbb{N}$.

If $t < 0$, choose $n_1 \geq n_0$ such that $t_n < 0$ for each $n \geq n_1$. We have
\[
\pi (A_{(v_n, t_n)}) \subset B^+_{\pi + t + \delta} (v, 0) \quad \text{for all }n \geq n_1
\]
because, picking $x \in \pi(A_{(v_n, t_n)})$ and $y \in \pi (\Sigma_{v_n})$ with $d_{v_n}(x) = -d(x, y)$, we obtain from (\ref{eq: blow up Case 2 (ii)}) and the triangle inequality
\[
d(x, (-v, 0)) \geq d(x, y) - d(y, (-v, 0)) = -d_{v_n}(x) - d(y, (-v, 0)) \geq -t_n - \frac{\delta}{2} \geq -t - \delta.
\]
Also
\[
B^+_{\pi + t - \delta}(v, 0) \subset \pi(A_{(v_n, t_n)}) \quad \text{for all } n \geq n_1
\]
because if $x \in B^+_{\pi + t - \delta}(v, 0)$, $x \notin B^+_{\delta - t}(-v, 0)$ and we obtain from (\ref{eq: blow up Case 2 (ii)}) 
\[
d(x, \pi(\Sigma_{v_n})) > d(x, \partial B^+_{\frac{\delta}{2}}(-v, 0)) > -t + \frac{\delta}{2} \geq - t_n.
\]
Hence, for all $n \geq n_1$, 
\begin{equation*}
    \left( \pi(A_{(v_n, t_n)}) \Delta B^+_{\pi + t} (v, 0) \right) \cup \pi(\Sigma_{(v_n, t_n)}) \subset \overline{B^+}_{\pi + t + \delta}(v, 0) \setminus B^+_{\pi + t - \delta}(v, 0).
\end{equation*}
Hence, taking the stereographic projection $\pi^{-1}$, we have
\begin{equation} \label{eq: blow up Case 2 (iv)}
\left( A_{(v_n, t_n)} \Delta \, \pi^{-1}(B^+_{\pi + t}(v, 0)) \bigr) \cup \Sigma_{(v_n, t_n)} \right) \subset \pi^{-1} \left(\overline{B^+}_{\pi + t + \delta}(v, 0) \setminus B^+_{\pi + t - \delta}(v, 0) \right)
\end{equation}
where
\[
\begin{split}
& \pi^{-1} \left(\overline{B^+}_{\pi + t + \delta}(v, 0) \setminus B^+_{\pi + t - \delta}(v, 0) \right) \\
& = \left( \overline{B^3}_{\tan{(\pi + t+ \delta)}} (\sec{(\pi + t + \delta)}v) \setminus B^3_{\tan{(\pi + t - \delta)}} (\sec{(\pi + t - \delta)}v) \right) \cap \mathbb{B}^3
\end{split}
\]
is the region between two orthogonal spherical caps whose centers are on the radial line starting from the origin and passes $v$. Thus, in any case, Proposition \ref{prop:convergence.sets.fb} (ii) follows from 
(\ref{eq: blow up Case 2 (iii)}) and (\ref{eq: blow up Case 2 (iv)}) since we can choose $\delta > 0$ arbitrarily small. Proposition \ref{prop:convergence.sets.fb} (iii) is proven exactly in the same way as Proposition \ref{prop:convergence.sets.fb} (ii). We now prove Proposition \ref{prop:convergence.sets.fb} (iv).

\begin{lem} \label{lem: A and A^*}
    There exists $r_0 > 0$ such that for every $p \in \partial \Sigma$, we have
    \[
    B^3_{\tan{r_0}}(\sec{r_0}(\cos{r_0} \, p - \sin{r_0} \, N(p))) \cap \mathbb{B}^3 \subset A
    \]
    and
    \[
    \overline{A} \subset \mathbb{B}^3 \setminus B^3_{\tan{r_0}}(\sec{r_0}(\cos{r_0} \, p + \sin{r_0} \,  N(p))).
    \]
\end{lem}

\begin{proof} 
It readily follows from the fact that $\Sigma$ is a smooth properly embedded surface in $\mathbb{B}^3$. 
\end{proof}

Write $v_n = \Lambda (p_n, (s_{n_1}, s_{n_2}))$, where $k_n = s_{n_2}/s_{n_1}$ tends to $k$ and $p_n$ tends to $p$. Set $B_q = \{x \in \mathbb{B}^3 :\langle x, N(q) \rangle < 0 \}$ be the half ball, and $D_q = \partial B_q \cap \mathbb{B}^3$ the equator disk boundary of $B_q$, for $q \in \partial \Sigma$. We can choose $r_0 > 0$ small so that 
\begin{equation} \label{eq: tangent half ball}
\begin{split}
& B^3_{\tan{r_0}}(\sec{r_0}(\cos{r_0} \, q - \sin{r_0} \,  N(q))) \cap \mathbb{B}^3 \subset B_{q}, \\
& B^3_{\tan{r_0}}(\sec{r_0}(\cos{r_0} \, q + \sin{r_0} \,  N(q))) \cap \mathbb{B}^3 \subset \mathbb{B}^3 \setminus \overline{B_{q}}.
\end{split}
\end{equation}
for every $q \in \partial \Sigma$. Here we use the fact that $\Sigma$ intersects $\partial \mathbb{B}^3$ orthogonally so that the tangent plane of $\Sigma$ at $q$, $T_{q} \Sigma: = \{ \langle x, N(q) \rangle = 0 \}$, is the same plane where $D_{q}$ lies. Making sure (\ref{eq: tangent half ball}) holds is the most important reason we have to work with properly embedded surfaces that intersect $\partial \mathbb{B}^3$ orthogonally.

Now it follows from Lemma \ref{lem: A and A^*} and (\ref{eq: tangent half ball}) that we can choose $r_0 >0 $ sufficiently small such that
 \begin{equation} \label{eq: blow up Case 3 (i)}
A \Delta B_{p_n} \subset \mathbb{B}^3 \setminus \left(B^3_{\tan{r_0}}(\sec{r_0}(\cos{r_0} \, p_n + \sin{r_0} \,  N(p_n))) \cup B^3_{\tan{r_0}}(\sec{r_0}(\cos{r_0} \, p_n - \sin{r_0} \,  N(p_n))\right).
\end{equation}
(\ref{eq: blow up Case 3 (i)}) is the crucial step to show Proposition \ref{prop:convergence.sets.fb} (iv). From Proposition \ref{prop:A.1} of Appendix \ref{appendix}, we obtain the existence of $C_0 > 0$ and $n_0 \in \mathbb{N}$ such that
\[
B^4_{\overline{R}_{k_n} - C_0 \sqrt{a_n}} (\overline{Q}_{p_n,k_n}, 0) \cap \mathbb{S}_+^3 \subset \pi \circ F_v(A) \subset B^4_{\overline{R}_{k_n} + C_0 \sqrt{a_n}} (\overline{Q}_{p_n,k_n}, 0) \cap \mathbb{S}_+^3,
\]
for all $n \geq n_0$, where $a_n = \sqrt{1 + k_n^2}s_{n1}$. Notice that $a_n \to 0$. 
Therefore, from (\ref{eq: ball centered at Q^bar_p, k }), we see that for each $\delta > 0$, there exists $n_1 \geq n_0$ such that for every $n \geq n_1$, we have
\begin{equation}
B^+_{\overline{r}_{k_n} - \frac{\delta}{2}}(\overline{Q}_{p_n,k_n}, 0) \subset \pi \circ F_v(A) \subset B^+_{\overline{r}_{k_n} + \frac{\delta}{2}}(\overline{Q}_{p_n,k_n}, 0)
\end{equation}
and
\begin{equation}
    \pi (\Sigma_{v_n}) \subset \left(\overline{B^+}_{\overline{r}_{k_n} + \frac{\delta}{2}}(\overline{Q}_{p_n,k_n}, 0) \setminus B^+_{\overline{r}_{k_n} - \frac{\delta}{2}}(\overline{Q}_{p_n,k_n}, 0) \right)
\end{equation}
Now we can verbatim follow the proof of \cite[Proposition 5.3 (iv)]{marques2014} to obtain that
\[
\left(\pi (A_{(v_n, t_n)}) \Delta B^+_{\overline{r}_k + t}(\overline{Q}_{p, k}, 0) \right)\cup \pi (\Sigma_{(v_n, t_n)}) \subset \overline{B^+}_{\overline{r}_k +t + \delta}(\overline{Q}_{p, k}, 0) \setminus B^+_{\overline{r}_k + t - \delta}(\overline{Q}_{p, k}, 0).
\]
Applying the stereographic projection $\pi^{-1}$, we obtain
\[
\left(A_{(v_n, t_n)} \Delta \, \pi^{-1}(B^+_{\overline{r}_k + t}(\overline{Q}_{p, k}, 0)) \right)  \cup \Sigma_{(v_n, t_n)} \subset \pi^{-1} \left( \overline{B^+}_{\overline{r}_k +t + \delta}(\overline{Q}_{p, k}, 0) \setminus B^+_{\overline{r}_k + t - \delta}(\overline{Q}_{p, k}, 0)\right)
\]
where 
\[
\begin{split}
& \pi^{-1} \left( \overline{B^+}_{\overline{r}_k +t + \delta}(\overline{Q}_{p, k}, 0) \setminus B^+_{\overline{r}_k + t - \delta}(\overline{Q}_{p, k}, 0)\right) \\
& = \mathbb{B}^3 \cap \left(\overline{B^3}_{\tan{(\overline{r}_k +t + \delta)}} (\sec{(\overline{r}_k +t + \delta)} \, \overline{Q}_{p, k}) \setminus B^3_{\tan{(\overline{r}_k +t - \delta)}} (\sec{(\overline{r}_k +t - \delta)} \, \overline{Q}_{p, k})\right)
\end{split}
\]
is the region between two orthogonal spherical caps whose centers are on the radial line from the origin that passes $\overline{Q}_{p, k}$. The result follows since $\delta > 0$ can be chosen arbitrarily small.
\end{proof}

\begin{proof}[Proof of Theorem \ref{thm: flattop conti blow-up}]
It is easy to see that the functions $\overline{r}$  defined on $\mathbb{S}^2 \cup \overline{\Omega}_{\epsilon}$, $\overline{Q}$ on $\mathbb{S}^2$, are both continuous. 
We consider the map
\[
\mathbf{U}: \overline{\mathbb{B}^3} \times [- \pi, \pi] \to \mathbf{I}_3(\mathbb{B}^3),
\]
\begin{equation} \label{def: U}
    \mathbf{U}(v, t) = \begin{cases}
       & [|A_{(T(v), t)}|] \quad \text{if }v \in \mathbb{B}^3 \setminus \overline{\Omega}_{\epsilon}, \\
      & [|\pi^{-1}\left(B^+_{\overline{r}(v) + t}(\overline{Q}(v),0)\right)|] \quad \text{if }v \in \mathbb{S}^2 \cup \overline{\Omega}_{\epsilon}.
    \end{cases}
\end{equation}
Note that $A_{(T(v), t)}$ has finite perimeter and so indeed $\mathbf{U}(v, t) \in \mathbf{I}_3(\mathbb{B}^3)$ for all $(v, t) \in \overline{\mathbb{B}^3} \times [- \pi, \pi]$. By Proposition \ref{prop:convergence.sets.fb}, we can apply the same arguments as \cite[Lemma 5.6]{marques2014} to see that $\mathbf{U}$ is continuous with respect to the mass topology of currents.
From the boundary rectifiability theorem (Theorem $30.3$ of \cite{Simon1983Lectures}) we know that $C(v, t) = \partial\mathbf{U}(v, t) \in \mathcal{Z}_2(\mathbb{B}^3, \partial \mathbb{B}^3; \mathbb{Z}_2)$, and Theorem \ref{thm: flattop conti blow-up} easily follows.

\begin{comment}
We are left to prove the final statement of Theorem \ref{thm: general surfaces blow up continuous in flat topology}. If $v \in \mathbb{S}^2 \cup \overline{\Omega}_{\epsilon}$, it is clear from (\ref{def: U}) that $\mathbf{U}(v, \pi) = [|\mathbb{B}^3|]$ and $\mathbf{U}(v, - \pi) = 0$, and that $C(v, \pm \pi) = 0$.

If $v \in \mathbb{B}^3 \setminus \overline{\Omega}_{\epsilon}$, set $\omega = T(v) \in \mathbb{B}^3$. Since $\pi(\Sigma_{\omega})$ is a smooth surface in $\mathbb{S}^3_+$, there can be no point $p \in \mathbb{S}^3$ with $d(p, \pi(\Sigma_{\omega})) \geq \pi$ (otherwise $\pi(\Sigma_{\omega}) \subset \{ -p \}$). Therefore $A_{(\omega, \pi)} = \mathbb{B}^3$ and $A_{(\omega, -\pi)} = \emptyset$. Since in this case $\mathbf{U}(v, t) = [|A_{(\omega, t)}|]$, we again have
$C(v, \pm \pi) = \partial \mathbf{U}(v, \pm \pi) = 0$.    
\end{comment}
\end{proof}

\part{Existence of three free boundary minimal annuli}
\section{Framework} \label{sec: Existence of 3 free-boundary minimal annuli}
In this section, we give the framework to prove Theorem \ref{thm:threeminimalannuli} which states that every compact $3$-manifold $M$ with nonnegative Ricci curvature and strictly convex boundary contains at least 3 embedded free boundary minimal annuli. We will adapt the scheme developed by Chu-Li \cite[Section 3]{chu2024existence} to our $6$-parameter family $\Psi$ (see Section \ref{sec: 6-parameter family}).

\subsection{Repetitive min-max in free boundary setting} \label{subsec: Repetitive min-max scheme} Without loss of generality, we assume that there exist finitely many free boundary annuli in $(M,g)$. In Section \ref{sec: 6-parameter family}, we obtain a $6$-parameter Simon-Smith family $\Psi$ satisfying the topological bound $(0, 1)$ in the unit ball $\mathbb{B}^3$. But this family can also be viewed as a Simon-Smith family in $\mathbb{B}^3$ with any metric. By Meeks-Simon-Yau \cite{meeks1982embedded} (also see \cite[Theorem 2.11]{FraserLi2014}), any compact $3$-manifold with strictly convex boundary is diffeomorphic to the unit ball $\mathbb{B}^3$. Thus, we can view $\Psi$ as a Simon-Smith family in $M$. Throughout the rest of the paper, We may mix the use of $M$ and $\mathbb{B}^3$.

We will use the following perturbation lemma. By compactness of properly embedded free boundary minimal surfaces of fixed topological type \cite[Theorem 1.2]{FraserLi2014}, let us take $L>0$ such that all free boundary minimal disks and annuli have area less than $L/2$ in $(M, g)$. 
\begin{prop}\label{prop:perturbMetric}
        Let $(M, g)$ and $L > 0$ be chosen above. Then, there exist a smooth metric $g'$ on $M$ such that $(M, g')$ has strictly convex boundary $\partial M$ and the following properties hold among all properly embedded $g'$-free boundary minimal surfaces with area less than $L$.
        \begin{enumerate}[label=\normalfont(\arabic*)]
            \item\label{item:sameNumber}  $(M, \partial M, g)$ and $(M, \partial M, g')$ have the same number of embedded minimal annuli.
            \item $(M, \partial M, g')$ has finitely many embedded minimal disks, each of which is non-degenerate.
            \item \label{item:area_disk_neq_annuli} For any embedded free boundary minimal disk $D$ in $(M, \partial M, g')$ and any positive integer $m$, the value $m \cdot \text{Area}_{g'}(D)$ can never equal the $g'$-area of an embedded free boundary minimal annulus in $(M, \partial M, g')$.
            \item \label{item:area_disk_linear_indep} For any pair of distinct embedded free boundary minimal disks $D_1, D_2$ in $(M, \partial M, g')$ and any positive integers $m_1, m_2$, we have
            \[
                  m_1 \cdot \text{Area}_{g'}(D_1) \neq m_2 \cdot  \text{Area}_{g'}(D_2)\,.
            \]
            \item Any properly embedded $g'$-free boundary minimal surface whose area is smaller than $L$ is unstable.
            \item The Frankel property holds for disks and annuli: For any pair of connected properly embedded $g'$-free boundary minimal surfaces $\Sigma_1,\Sigma_2$ satisfying the topological bound $(0,1)$, we have
            \[
            \Sigma_1\cap\Sigma_2\neq\emptyset.
            \]
            \item There is no closed embedded $g'$-minimal sphere in $M$.
        \end{enumerate}
    \end{prop}

We will prove Proposition \ref{prop:perturbMetric} in Section \ref{sec:perturbmetric} and give some idea of the proof here. We follow the same idea as \cite[Proposition 6.1]{chu2024existence} to perturb $g$ to satisfy conditions \normalfont{(1)}-\normalfont{(4)} above. It is not hard to see that conditions \normalfont{(5)}-\normalfont{(7)} are all open conditions. Thus, the perturbed metric $g'$ can inherit \normalfont{(5)}-\normalfont{(7)} from $g$ and the proof is complete.

Conditions \normalfont{(5)}-\normalfont{(7)} serve as an alternative of the nonnegative Ricci condition, as we cannot guarantee the perturbed metric $g'$ always has nonnegative Ricci curvature in $M$. As we shall see below, conditions \normalfont{(5)}-\normalfont{(7)} are enough for our use. Thus, in the rest of this section, we perturb the metric $(M, g)$ to satisfy conditions \normalfont{(1)}-\normalfont{(7)} of Proposition \ref{prop:perturbMetric} and still denote it by $(M, g)$.

\begin{lem}\label{lem:W_one_elmt}
Given any Simon-Smith family $\Phi: X \to \mathcal{S}^*(M)$ satisfying the topological bound $(0, 1)$ with $L = \mathbf{L}_{SS}(\Phi) > 0$, by Theorem \ref{thm:simon-smith min-max}, there exists a pulled-tight minimizing sequence $\{\Phi_i\}$ such that $\bC(\{\Phi_{i}\})\cap \mathcal{W}_{L, \leq 0,\le 1}$ contains a multiplicity one free boundary minimal surface $\Sigma$. Moreover, one of the following holds
\begin{enumerate}[label=\normalfont(\arabic*)]
            \item If $\Sigma$ is a free boundary minimal annulus, then every varifold in $ \bC(\{\Phi_{i}\})\cap \mathcal{W}_{L, \leq 0,\le 1}$ is associated with a multiplicity-one minimal annulus. 
            \item If $\Sigma$ is a free boundary minimal disk, then 
            \[\bC(\{\Phi_{i}\})\cap \mathcal{W}_{L, \leq 0,\le 1} = \{ | \Sigma| \}.
            \]
        \end{enumerate}
        In particular, $\bC(\{\Phi_{i}\})\cap \mathcal{W}_{L, \leq 0,\le 1}$ consists of finitely many varifolds.
        \end{lem}
    \begin{proof}
By Theorem \ref{thm:simon-smith min-max} and Proposition \ref{prop:perturbMetric} (5)(6), we can find a pulled-tight minimizing sequence $\{\Phi_i\}$ such that $\bC(\{\Phi_{i}\})\cap \mathcal{W}_{L, \leq 0,\le 1}$ contains a multiplicity one free boundary minimal surface $\Sigma$ that is either a disk or annulus. The dichotomy then follows from Proposition \ref{prop:perturbMetric} \ref{item:area_disk_neq_annuli} \ref{item:area_disk_linear_indep} and Theorem \ref{thm:special}.
    \end{proof}
By our assumption and Proposition \ref{prop:perturbMetric}, we take $d_{0}$ to be smaller than the shortest distance among the finitely many free boundary minimal disks and annuli in $(M, g)$ both in the varifold $\bF$-distance sense and in the relative cycle space distance.

\subsubsection{First step} \label{sec:boundary component cap} 

Applying Theorem \ref{thm:simon-smith min-max} with $r= d_{0}$ to $\Psi$, we obtain a pulled-tight minimizing sequence $\{ \Phi_{i}\}$ and a min-max limit varifold $V^{1} \in \mathcal{W}^{1}:=\bC(\{\Phi_{i}\}) \cap \mathcal{W}_{\bL(\Lambda(\Psi)), \le 0, \le 1}$. By Lemma \ref{lem:W_one_elmt}, we obtain the following dichotomy:
\begin{enumerate}[label= \textbf{Case} \normalfont\arabic* :]
    \item $\mathcal{W}^{1}$ consists of varifolds induced by a multiplicity-one minimal annulus;
    \item $\mathcal{W}^{1} = \{ |\Sigma| \}$ where $\Sigma$ is a multiplicity-one minimal disk.
\end{enumerate}

\textbf{Case $1$}. If Case $1$ occurs, by Theorem \ref{thm:simon-smith min-max} \normalfont{(3)}, there exists $\delta > 0$ and a pulled-tight minimizing sequence $\{ \Phi_i \} \subset \Lambda(\Phi)$ such that
\[
\mathbf{M}(\Phi_i(x)) \geq \mathbf{L}(\Lambda(\Phi)) - \delta \quad \Rightarrow \quad |\Phi_i(x)| \in \mathbf{B}_{d_0}^{\mathbf{F}}(\mathcal{W}^1).
\]
However, this is not sufficient to proceed with the topological argument in the cycle space $\mathcal{Z}_2(M,\partial M, \mathbb{Z}_{2})$, and we will apply the following theorem.
\begin{thm}\label{thm:currentsCloseInBoldF}
        Suppose $(M,g)$ satisfies conditions \normalfont{(1)}-\normalfont{(7)} of Proposition \ref{prop:perturbMetric}. Let $\Lambda$ be the homotopy class of a Simon-Smith family satisfying the topological bound $(\fg_{0}, \fb_{0})$ for some nonnegative integers $\fg_{0}, \fb_{0}$, with $L = \mathbf{L}(\Lambda) > 0$. Suppose that $\bC(\{\Phi_{i}\}) \cap \mathcal{W}_{\bL(\Lambda), \le \fg_{0}, \le \fg_{0}+\fb_{0}}$ consists of finitely many multiplicity one varifolds for some pulled-tight minimizing sequence $\{\Phi_i\}$ in $\Lambda$. Then for any $r > 0$, there exists $\eta>0$ and $\Phi \in \Lambda$ such that 
        \[
          \bM(\Phi(x))\geq L -\eta \implies [\Phi(x)] \in \bB^\bF_r([\bC(\{\Phi_{i}\}) \cap \mathcal{W}_{\bL(\Lambda), \le \fg_{0}, \le \fg_{0}+\fb_{0}}]) \,.
        \]
        \begin{proof} 
        It follows from the same proof as  \cite[Theorem 3.5]{chu2024existence}, as \cite[Corollary 7.2]{chu2024existence} also holds for relative cycles and the proof of \cite[Proposition 7.3]{chu2024existence} is a local argument.
        \end{proof}
    \end{thm}

We apply Theorem \ref{thm:currentsCloseInBoldF} to $(M, g)$, $\Lambda = \Lambda(\Psi)$, $r=d_0$, and $(\fg_{0}, \fb_{0}) = (0, 1)$. Then we obtain $\eta > 0$ and $\Phi \in \Lambda(\Psi)$ satisfying the properties of Theorem \ref{thm:currentsCloseInBoldF}. For late reference, we denote $\delta_1 := \eta$ and $\tilde\Psi^1 := \Phi$.

By refining the cubical complex structure of the parameter space $Y$, we obtain a $6$-dimensional subcomplex $C^1$ by refining the cubical complex such that 
    \begin{equation}\label{eq:Phi_iMassBound1}
        \bM(\tilde \Psi^1(x))\geq \bL(\Lambda(\Psi))-\delta_1/2 \implies x\in C^1
    \end{equation}
    and  
    \begin{equation}\label{eq:C1liesIn} 
        x\in C^1 \implies [\tilde \Psi^1(x)]\in\bB^\bF_{d_0}([\mathcal{W}^1]).
    \end{equation}
Roughly speaking, $C^1$ is the "cap" where $\tilde\Psi^1$ has large area and is close to some free boundary minimal annuli.

Since $(M, g)$ has only finitely many properly embedded free boundary minimal annuli, by the definition of $d_{0}$, and by further refining $Y$,  we can decompose $C^{1}$ into a disjoint union of $6$-dimensional subcomplexes $\{C^{1}_{j} \}_{1 \le j \le n_{i}}$ such that for each $j = 1, \cdots, n_1$, there is a distinct free boundary minimal annulus $A^1_j \in \mathcal{S}^*(M)$ with $|A^1_j| \in \mathcal{W}^1$, such that
\[
[\tilde \Psi^1](C^1_j) \subset \mathbf{B}_{d_0}^{\mathbf{F}}([A^1_j]).
\]
finally, we define \[
\Psi^1 := \tilde \Psi^1|_{\overline{Y\setminus C^1}}
\]
By (\ref{eq:Phi_iMassBound1}), $\Psi^1$ satisfies
\begin{equation}\label{eq:Psi1MassBound}
        \sup_ {x \in \text{dmn}(\Psi^{1})}\mathcal{H}^2(\Psi^1(x))<\bL(\Lambda(\Psi)).
\end{equation}
Later when we apply the min-max process for the second time, we will start with $\Psi_1$, and the area upper bound guarantees new min-max minimal surfaces will be detected. For notational convenience later, we denote
\[\tilde Y^1:= \text{dmn}(\tilde\Psi^1)=Y, \quad Y^1=\overline{Y\setminus C^1}=\text{dmn}(\Psi^1).
\]

\textbf{Case $2$}. In this case, the goal is once again to construct a map $\Psi_1$ that satisfies the area bound (\ref{eq:Psi1MassBound}). Additionally, we want $\Psi_1$, when restricted to the boundary of the removed cap, to form a Simon-Smith family of disks. To achieve this, we adapt the interpolation theorem in Chu-Li \cite{chu2024existence} to our setting and collect necessary notions first.

Let $W=M_f$ be the {\it mapping cylinder} of $f:X'\to X$, with $\partial_0W=X$ and $\partial_1W=X'$. We denote by $F_W:[0,1]\times W\to W$ its contraction onto $\partial_0W$, defined by $(t,s,x') \mapsto (ts,x')$, as a {\it cubical  map}, i.e., every cell of $X'$ is mapped to a cell of $X$. We regard \(W\) as a finite complex as in \cite{chu2024existence}. We will abuse notation by using the terms $\partial_0W$ and $X$, and $\partial_1W$ and $X'$ interchangeably. Note that $W$ can also be regarded as a cubical complex by ~\cite[\S 4]{buchstaber2002torus}. Let us state the free boundary analog of the interpolation theorem in \cite[Theorem 3.6]{chu2024existence}.

\begin{thm}\label{thm:mapping_cylinder}
         Suppose $(M, g)$ satisfies conditions \normalfont{(1)}-\normalfont{(7)} of Proposition \ref{prop:perturbMetric}. Let $\Phi:X \rightarrow \mathcal{S}^{*}(M)$ be a Simon-Smith family satisfying the topological bound $(0,1)$ with $L = \mathbf{L}(\Lambda(\Phi)) > 0$, where $X$ is a cubical subcomplex in $I^m$. Let $\{\Phi_i\}$ be a pulled-tight minimizing sequence in $\Lambda(\Phi)$, and suppose that the set $\bC(\{\Phi_{i}\}) \cap \mathcal{W}_{L, \le 0, \le 1}$ consists of exactly one varifold associated with a multiplicity-one nondegenerate free boundary minimal disk. Then there exists a cubical subcomplex $X'$ in $I^{m+1}$, and homotopy equivalence
        $f: X' \to X\,,$ which is a surjective cubical map after refining $X$, and a Simon-Smith family $\Phi': X' \to \mathcal{S}^*(M)$ satisfying topological bound $(0,1)$ with the following properties.
        \begin{enumerate}[label=\normalfont(\arabic*)]
            \item There exists an $\eta>0$ such that for $x \in X'$,  
                \[
                    \mathcal{H}^2(\Phi'(x))\geq L - \eta\quad\Longrightarrow\quad \fg(\Phi'(x))=0\,\text{and } \fb(\Phi'(x))=0.
                \]
            \item\label{item:mapping_cylinder} Let $W$ be the mapping cylinder $M_f = (([0, 1] \times X') \sqcup X) / \sim$ of $f$, where $(0, x) \sim f(x)$ for all $x \in X'$. There exists a Simon-Smith family $H: W \to \mathcal{S}^*(M)$ satisfying topological bound $(0,1)$ such that:
            \begin{enumerate}[label=\normalfont(\alph*)]
                \item $H|_{\partial_0 W}=\Phi$ and $H|_{\partial_1 W} = \Phi'$.
                \item\label{item_item:boundary-pinch-off} For all $x \in X'$, $t \mapsto H(t, x)$ is a {\em relative pinch-off process} (see Definition~\ref{defn:boundary_pinch_off} below). 
            \end{enumerate}
        \end{enumerate}
    \end{thm}

\begin{comment}
\begin{rem} \label{rem: relative pinch-off process}
We explain a heuristic of the relative pinch-off process we define below. If an annulus is sufficiently close to a disk, then there exists either an essential arc connecting the two boundary components of the annulus within in a small half-ball $B \cap \partial M$, or a homologically nontrivial simple closed curve  within a small interior ball $B$, respectively. The reader may refer to Theorem \ref{thm:relativesystolic} for the precise statement and proof. Let us denote this curve by $\gamma$. The relative pinch-off process is defined to be a free boundary Simon-Smith family that continuously pinches off the small (half-) cylinder containing $\gamma$. This process consists of a finite number of (half) neck-pinch surgeries and the surface becomes disjoint from this small (half-)ball. We shrink the components inside the ball into a point. 
\end{rem}    
\end{comment}

\begin{defn}\label{defn:boundary_pinch_off}
        In a compact Riemannian manifold with boundary $(M, g)$, given a $1$-parameter free boundary Simon-Smith family $\Phi: [a, b] \to \mathcal{S}^*(M)$, where we view the parameter space $[a, b]$ as a time interval, we call $\Phi$ by a {\em relative pinch-off process} if the following holds. 
        
        There exist finitely many spacetime points $(t_1,p_1),\cdots,(t_n,p_n)$, where $t_i\in[a,b]$, $p_i\in\Phi(t_i)$ such that: 
        \begin{enumerate}[label=\normalfont(\arabic*)]
            \item For each $(t,p)$, where $t\in[a,b]$ and $p\in\Phi(t)$, different from any $(t_i,p_i)$, there exists an open time interval $J\subset [a,b]$ around $t$, a one-parameter group of diffeomorphisms $\{\varphi_{t'}\}_{t' \in J} \subset \operatorname{Diff}(M, \partial M)$ and an (relative) open ball $U\subset M$ around $p$ such that for all $t' \in J$,
            \[
                \varphi_{t'}(\Phi(t)) \cap U = \Phi(t') \cap U\,.
            \]
            \item For each $i=1,\cdots,n$, there exists an open time interval $J_{i}\subset [a,b]$ around $t_i$ and a (relative) open ball $U\subset M$ around $p_i$ such that the family 
            \[
                \{\Phi(t')\cap U\}_{t'\in J_{i}}
            \] is of one of the following types:
            \begin{enumerate}[label=\normalfont(\alph*)]
            \item A {\it shrinking process}: For each $t'\in J_{i}$, $\Phi(t')$ does not intersect $\partial U \cap \operatorname{int}(M)$, and $\{ \Phi(t')\cap U\}_{t' \in J_{i}\cap[a,t_i)}$ is induced by a one-parameter group of diffeomorphisms in $U$, and $\Phi(t')\cap U = p_i$ for all $t' \in J_{i} \cap [t_i,b]$ (Here it is possible that $p_i \in \partial M$ and $U$ is diffeomorphic to a half geodesic ball $B_R(p_i) \cap M$).
            \item A {\it surgery process}: we have two types of surgeries as follows.
                \begin{enumerate}
                    \item Interior neck-pinch: for $p_i \in \operatorname{int}(M)$, pinching a (topological) cylinder at the point $p_i$ at time $t_i$, to obtain a (topological) double cone, and then either splitting it into two smooth disks or leaving it intact.
                    \item Boundary neck-pinch: for $p_i \in \partial M$ and $U$ is diffeomorphic to a half geodesic ball $B_R(p_i) \cap M$,  pinching a topological half-cylinder at $p_i$, at time $t_i$, to obtain a topological half-double cone, and then either splitting it into two smooth proper discs with boundaries on $\partial M$ or leaving it intact (see Figure \ref{fig:Boundary neck-pinch}).
                \end{enumerate}
           \end{enumerate}
        \end{enumerate}
    \end{defn}
\begin{comment}
\begin{rem}
 \textcolor{red}{can be deleted}   We apply the relative pinch-off process instead of the free boundary mean curvature flow due to the singularity. We also remark that even the singularity models for the free boundary mean curvature flow have been less explored than that for the mean curvature flow. We will see how this relative pinch-off process simplifies the topology of surfaces in Section \ref{sec:relativepinchoff}.
\end{rem}    
\end{comment}
\begin{figure} [H]
    \centering
\resizebox{0.5\linewidth}{!}{%
\begin{tikzpicture}[
    x={(0.9cm,0.25cm)},
    y={(-0.55cm,0.35cm)},
    z={(0cm,1cm)},
    line join=round,
    line cap=round,
    >=Latex
]

% Common parameters
\def\zmax{1.35}
\def\Nz{24}
\def\Nt{24}
\def\thetac{-31.4} % outer contour meridian for this projection

% Useful radius for the catenoid top/bottom
\pgfmathsetmacro{\rcat}{0.5*(exp(\zmax)+exp(-\zmax))}

% --------------------------------------------------
% Left figure: half catenoid
% --------------------------------------------------
\begin{scope}[shift={(-5.8cm,0cm)}]

% Plane x = 0
\fill[blue!20,opacity=0.45]
  (0,-2.4,-1.6) -- (0,2.4,-1.6) -- (0,2.4,1.6) -- (0,-2.4,1.6) -- cycle;

\draw[blue!70!black,thick]
  (0,-2.4,-1.6) -- (0,2.4,-1.6) -- (0,2.4,1.6) -- (0,-2.4,1.6) -- cycle;

% Half-catenoid surface: x >= 0
% (x,y,z) = (cosh(z) cos(theta), cosh(z) sin(theta), z), theta in [-90,90]

\foreach \i in {0,...,23} {
  \pgfmathsetmacro{\za}{-\zmax + 2*\zmax*\i/\Nz}
  \pgfmathsetmacro{\zb}{-\zmax + 2*\zmax*(\i+1)/\Nz}

  \pgfmathsetmacro{\ra}{0.5*(exp(\za)+exp(-\za))}
  \pgfmathsetmacro{\rb}{0.5*(exp(\zb)+exp(-\zb))}

  \foreach \j in {0,...,23} {
    \pgfmathsetmacro{\ta}{-90 + 180*\j/\Nt}
    \pgfmathsetmacro{\tb}{-90 + 180*(\j+1)/\Nt}

    \fill[gray!35,opacity=0.75]
      ({\ra*cos(\ta)},{\ra*sin(\ta)},\za)
      -- ({\ra*cos(\tb)},{\ra*sin(\tb)},\za)
      -- ({\rb*cos(\tb)},{\rb*sin(\tb)},\zb)
      -- ({\rb*cos(\ta)},{\rb*sin(\ta)},\zb)
      -- cycle;
  }
}

% Boundary / intersection with plane x = 0
\draw[black,thick]
  plot[domain=-\zmax:\zmax,samples=100]
  (0,{0.5*(exp(\x)+exp(-\x))},\x);

\draw[black,thick]
  plot[domain=-\zmax:\zmax,samples=100]
  (0,{-0.5*(exp(\x)+exp(-\x))},\x);

% Highest and lowest altitude lines
\draw[black,thick]
  plot[domain=-90:90,samples=100]
  ({\rcat*cos(\x)},{\rcat*sin(\x)},\zmax);

\draw[black,thick]
  plot[domain=-90:90,samples=100]
  ({\rcat*cos(\x)},{\rcat*sin(\x)},-\zmax);

% Outermost contour meridian
\draw[black,thick]
  plot[domain=-\zmax:\zmax,samples=100]
  ({0.5*(exp(\x)+exp(-\x))*cos(\thetac)},
   {0.5*(exp(\x)+exp(-\x))*sin(\thetac)},
   \x);

% Dashed lines in the plane x=0
\draw[dashed,thick] (0,-\rcat,\zmax) -- (0,\rcat,\zmax);
\draw[dashed,thick] (0,-\rcat,-\zmax) -- (0,\rcat,-\zmax);

\end{scope}

% --------------------------------------------------
% Right figure: half double cone
% --------------------------------------------------
\begin{scope}[shift={(5.8cm,0cm)}]

% Plane x = 0
\fill[blue!20,opacity=0.45]
  (0,-2.4,-1.6) -- (0,2.4,-1.6) -- (0,2.4,1.6) -- (0,-2.4,1.6) -- cycle;

\draw[blue!70!black,thick]
  (0,-2.4,-1.6) -- (0,2.4,-1.6) -- (0,2.4,1.6) -- (0,-2.4,1.6) -- cycle;

% Half double cone surface: x >= 0
% (x,y,z) = (|z| cos(theta), |z| sin(theta), z), theta in [-90,90]

\foreach \i in {0,...,23} {
  \pgfmathsetmacro{\za}{-\zmax + 2*\zmax*\i/\Nz}
  \pgfmathsetmacro{\zb}{-\zmax + 2*\zmax*(\i+1)/\Nz}

  \pgfmathsetmacro{\ra}{abs(\za)}
  \pgfmathsetmacro{\rb}{abs(\zb)}

  \foreach \j in {0,...,23} {
    \pgfmathsetmacro{\ta}{-90 + 180*\j/\Nt}
    \pgfmathsetmacro{\tb}{-90 + 180*(\j+1)/\Nt}

    \fill[gray!35,opacity=0.75]
      ({\ra*cos(\ta)},{\ra*sin(\ta)},\za)
      -- ({\ra*cos(\tb)},{\ra*sin(\tb)},\za)
      -- ({\rb*cos(\tb)},{\rb*sin(\tb)},\zb)
      -- ({\rb*cos(\ta)},{\rb*sin(\ta)},\zb)
      -- cycle;
  }
}

% Boundary / intersection with plane x = 0
\draw[black,thick]
  plot[domain=-\zmax:\zmax,samples=100]
  (0,{abs(\x)},\x);

\draw[black,thick]
  plot[domain=-\zmax:\zmax,samples=100]
  (0,{-abs(\x)},\x);

% Highest and lowest altitude lines
\draw[black,thick]
  plot[domain=-90:90,samples=100]
  ({\zmax*cos(\x)},{\zmax*sin(\x)},\zmax);

\draw[black,thick]
  plot[domain=-90:90,samples=100]
  ({\zmax*cos(\x)},{\zmax*sin(\x)},-\zmax);

% Outermost contour meridian
\draw[black,thick]
  plot[domain=-\zmax:\zmax,samples=100]
  ({abs(\x)*cos(\thetac)},
   {abs(\x)*sin(\thetac)},
   \x);

% Dashed lines in the plane x=0
\draw[dashed,thick] (0,-\zmax,\zmax) -- (0,\zmax,\zmax);
\draw[dashed,thick] (0,-\zmax,-\zmax) -- (0,\zmax,-\zmax);

% Mark the origin
\fill[black] (0,0,0) circle (0.9pt);
\node[black] at ($(0,0,0)+(0.52,-0.22,0)$) {$p_i$};

\end{scope}

% --------------------------------------------------
% Big right arrow in the middle
% --------------------------------------------------
\node at (0cm,0cm) {\Huge$\longrightarrow$};

\end{tikzpicture}%
}
 \caption{Boundary neck-pinch}
    \label{fig:Boundary neck-pinch}
\end{figure}
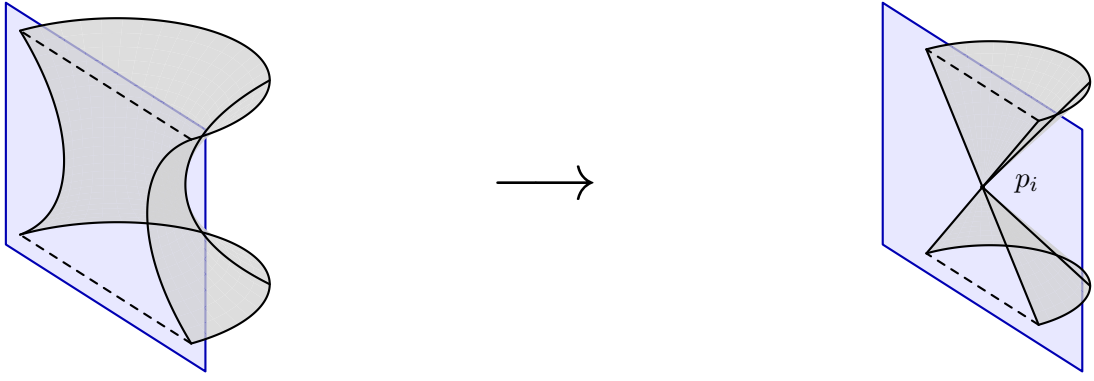
We will prove Theorem \ref{thm:mapping_cylinder} in Section \ref{boundaryinterpolation} later. The main idea is that if an annulus is sufficiently close to a disk in $\mathcal{Z}_2(M,\partial M;\bF;\mathbb{Z}_2)$, then the annulus must contain a nontrivial essential arc connecting its two boundary components  on a small half-ball, or a homologically nontrivial simple closed curve in a small interior ball. Then we apply the relative pinch-off process to coninuously pinch-off the small half-cylinder (cylinder) containing this essential arc (simple closed curve) to obtain a punctate surface which is a disjoint union of topological disks. Then we can extend Pitts' combinatorial arguments \cite[Theorem 4.10]{pitts2014existence} and can construct a desired deformation map $H$ as in \cite{chu2024existence}.

Note that $H$ is not a homotopy between two Simon-Smith families of surfaces $\Phi$ and $\Phi'$ since it involves with surgery and shrinking processes. However, in the cycle space, the induced maps $[\Phi]$ and $[\Phi']$  are in the same homology class in $\mathcal{Z}_2(M,\partial M;\bF;\mathbb{Z}_2)$ via the continuous map $[H]$. 

Applying Theorem \ref{thm:mapping_cylinder} to $\Psi$, we obtain
\begin{itemize}
    \item a cubical complex $\tilde Y^1$,
    \item a surjective, cubical, and homotopy equivalence $f^1:\tilde Y^1\to Y$
    \item a mapping cylinder $W^1:=M_{f^1}$ and the associated map $F^1:=F_{W^1}:[0,1]\times W^1\to W^1$ (see (\ref{eq: F_W})), 
    \item a Simon-Smith family \[
    \tilde\Psi^1:\tilde Y^1\to\mathcal{S}^*(M)
    \]
    satisfying the topological bound $(0,1)$, where $H^1|_{\partial_0W^1}=\Psi$,
    \item a constant $\delta_1 > 0$ (in place of $\eta$),
\end{itemize}
satisfying the properties listed in Theorem \ref{thm:mapping_cylinder}.

By refining $\tilde Y^1$, we obtain a cubical subcomplex $C^1$ of $\tilde Y^1$ the cubical complex such that 
\begin{equation}\label{eq:Phi_iMassBound2}
        \mathcal{H}^{2}(\tilde \Psi^1(x))\geq \bL(\Lambda(\Psi))-\delta_1/2 \implies x\in C^1,
\end{equation}
and $\tilde\Psi^1|_{C^1}$ is a Simon-Smith family of disks. Finally, we let \[\
\Psi^1:=\tilde\Psi^1|_{\overline{\tilde Y^1\setminus C^1}}\,
\]
As in Case $1$, a key property of $\Psi^1$ is that
\begin{equation}\label{eq:Psi1MassBound2}
        \sup_ {x \in \text{dmn}(\Psi^{1})}\mathcal{H}^2(\Psi^1(x))<\bL(\Lambda(\Psi)).
\end{equation}

In summary, in both cases, we obtain a Simon-Smith family $\Psi^1$ satisfying the topological bound $(0, 1)$ with parameter space $\overline{\tilde{Y}^{1}\setminus C^1}$ with
$\sup \mathcal{H}^2 \circ \Psi^1$ strictly less than the width $\mathbf{L}(\Lambda(\Psi))$. Thus, if we apply the Simon-Smith min-max theorem again to $\Psi^{1}$, all min-max minimal surfaces will have strictly smaller area than those obtained previously. We will continue repeating the min-max process.
\subsubsection{$k$-th stage} Inductively, we have a Simon-Smith family with the topological bound $(0,1)$ as
    \[
        \Psi^{k-1}:Y^{k-1}\to \mathcal{S}^*(M)\,,
    \] 
    where $Y^{k-1}$ is some cubical complex. If $k = 1$, we take $\Psi^0:=\Psi$ and $Y^0:=Y$. By applying Theorem \ref{thm:simon-smith min-max} and Lemma \ref{lem:W_one_elmt} to $\Psi^{k - 1}$ and choose $r = d_0$, then  there exists a min-max limit varifold $V^{k} \in \mathcal{W}^{k}:=\bC(\{\Phi_{i}^{k-1}\}) \cap \mathcal{W}_{\bL(\Lambda(\Psi^{k-1})), \le 0, \le 1}$. Again, we have the following dichotomy: 
    \begin{enumerate} [label= \textbf{Case} \normalfont\arabic* :]
    \item Every element in $\mathcal{W}^{k}$ is a varifold induced by a multiplicity-one minimal annuli;
    \item $\mathcal{W}^{k} = \{ |\Sigma^{k}| \}$ where $\Sigma^{k}$ is a multiplicity-one free boundary minimal disk.
\end{enumerate}

In the following, we adopt the convention that the superscript $^k$ labels the object from $k$-th step of the min-max process. 

\textbf{Case $1$.} As in the first step, we have:
\begin{itemize}
    \item A Simon-Smith family \(\tilde{\Phi}^k: Y^{k -1} \to \mathcal{S}^*(M)\) satisfying the topological bound $(0, 1)$ in $\Lambda(\Psi^{k - 1})$.
    \item a "cap" $C^k$, which is a subcomplex of $Y^{k - 1}$ and can be decomposed into a disjoint union of cubical subcomplexes $C_1^k, \cdots, C_{n_k}^k$, the restriction $\Psi^k:= \tilde{\Psi}^k|_{\overline{Y^{k - 1} \setminus C^k}}$, a new domain $Y^k:= \operatorname{dmn}(\Psi^k) = \overline{Y^{k - 1} \setminus C^k}$,
\end{itemize}
satisfying the following properties:
\begin{itemize}
    \item For each $j = 1, \cdots, n_k$, there is a distinct properly embedded free boundary minimal annulus $A_j^k \in \mathcal{S}^*(M)$, with $|A_j^k| \in \mathcal{W}^k$ such that
    \[
    [\tilde{\Psi}^k](C_j^{k}) \subset \mathbf{B}_{d_0}^{\mathbf{F}}([A_j^k]).
    \]
    \item $\sup_{Y^k} \mathcal{H}^2 \circ \Psi^k < \mathbf{L}(\Lambda(\Psi^{k - 1}))$.
\end{itemize}
Since $\tilde\Psi^k$ and $\Psi^{k-1}$ are homotopic in the sense of Simon-Smith family of surfaces, there exists a homotopy
\[
H^k:[0, 1] \times Y^{k - 1} \to \mathcal{S}^*(M),
\]
such that
\[
H^{k}(0, \cdot) = \Psi^{k - 1}, \quad H^k(1, \cdot) = \tilde \Psi^k,
\]
and for each $(t, x)$. 

For consistency with Case $2$,let $W_k=[0,1]\times Y_{k-1}$, identify its ends with $Y_{k-1}$ and $\widetilde Y_k$, and let $F_k$ denote the standard contraction onto $Y_{k-1}$.

\textbf{Case $2$}. As in the first step, we apply Theorem \ref{thm:mapping_cylinder} to $\Psi^{k-1}$ with $r= d_{0}$ and obtain the following.
\begin{itemize}
    \item a cubical complex $\tilde Y^k$, a surjective, cubical, and homotopy equivalence $f^k:\tilde Y^k\to Y^{k-1}$, a mapping cylinder $W^k:=M_{f^k}$ (viewed as a simplicial complex) with
    \[
    Y^{k - 1} = \partial_0 W^k, \quad \tilde{Y}^k = \partial_1 W^k
    \]
    and the associated map $F^k:=F_{W^k}:[0,1]\times W^k\to W^k$,
    \item a Simon-Smith family $\tilde\Psi^k:\tilde Y^k\to\mathcal{S}^*(M)$ satisfying the topological bound $(0, 1)$,
    \item a Simon-Smith family $H^k: W^k \to \mathcal{S}^*(M)$ satisfying the topological bound $(0, 1)$, with $H^k|_{\partial_0W^k}=\Psi^{k-1}$, $H^k|_{\partial_1W^k}=\tilde\Psi^k$,
    \item a "cap" $C^k$, which is a cubical subcomplex of $\tilde{Y}^k$, the restriction
    $\Psi^k:=\tilde\Psi^k|_{\overline{\tilde Y^k\setminus C^k}}$, a new domain $Y^k := \operatorname{dmn}(\Psi^k) = \overline{\tilde Y^k\setminus C^k}$,
\end{itemize}
satisfying the following:
\begin{itemize}
    \item The map $(t, w) \mapsto F^k(1 - t, w)$, where $(t, w) \in [0, 1] \times W^k$, is a strong deformation retract of $W^k$ onto $Y^{k - 1}$.
    \item For each $x \in \tilde{Y}^k$, the family $t \mapsto H^k(F^k(t, x))$, $t \in [0, 1]$, is a relative pinch-off process.
    \item $\tilde\Psi^k|_{C^k}$ is a Simon-Smith family of disks.
\end{itemize}

As before, the mass bound
\begin{equation}
 \sup_ {x \in \text{dmn}(\Psi^{k})}\mathcal{H}^2(\Psi^k(x))<\bL(\Lambda(\Psi^{k-1}))
\end{equation}
holds in both cases. Thus, the width $\bL(\Lambda(\Psi^{k}))$ is strictly decreasing by $k$. By combining this with our assumption that there are finitely many free boundary embedded disks and annuli in $(M,g)$ by Proposition \ref{prop:perturbMetric}, the repetitive min-max process must terminate in finitely many steps.
\subsubsection{Last stage.} Since $\bL(\Lambda(\Psi^{K-1}))=0$, there exists some Simon-Smith family $\Psi^K\in\Lambda(\Psi^{K-1})$ and a homotopy in the sense of Simon-Smith min-max,
    \[
        H^{K}:[0,1]\times Y^{K-1}\to \mathcal{S}^*(M),
    \]
    such that:
    \begin{itemize}
        \item $H^K(0,\cdot)=\Psi^{K-1}$, $H^K(1,\cdot)=\Psi^K$.
        \item For each $(t,x)$,  $H^K(t,x)$ is obtained from $\Psi^{K-1}(x)$ via some diffeomorphism of $M$ according to Definition~\ref{def:homotopyclasssimonsmith}.
        \item  
        $\sup_{Y^{K - 1}} \mathcal{H}^2_{g_{Eucl}}\circ \Psi^K<1$, where $g_{Eucl}$ denotes the Euclidean metric on $M$.
    \end{itemize}
Crucially, note that in the last bullet point, the Hausdorff measure is taken with respect to {\em the unit 3-ball $(\mathbb{B}^3,g_{Eucl})$} in (3).

We denote 
    \[
        W^K:=[0,1]\times Y^{K-1}\,,\quad \partial_0 W^K=\{0\}\times Y^{K-1}\,,\quad \partial_1 W^K=\{1\}\times Y^{K-1}\,,
    \] and define a map 
    \[
        F^K:[0,1]\times W^K\to W^K\,, \quad (s, (t, x)) \mapsto (st,x)\,.
    \]
    Note that, the map $(s,w)\mapsto F^K(1-s,w)$ for $(s,w)\in [0,1]\times W^K$, is a strong deformation retraction of $W^K$ onto $\{0\}\times Y^{K-1}$. As before, we identify $Y^{K-1}$ with the subset $\{0\}\times Y^{K-1}$ in $W^K$.
\subsubsection{New family $\Xi$} We have run the min-max processes $K$ times, generating multiple Simon-Smith families from the original family $\Psi$, each with a different parameter space. Now we proceed verbatim as \cite[\S 3.3]{chu2024existence} to glue these families together to construct a new Simon-Smith family $\Xi$ whose parameter space is homotopy equivalent to the original parameter space $Y$. Moreover, $\Xi$ satisfies the topological bound $(0,1)$ and the following proposition holds. 

\begin{prop}[{\cite[Proposition 3.11]{chu2024existence}}]\label{prop:XiPsiHomotopic}
        There exist:
        \begin{itemize}
            \item a simplicial complex $\tilde W$, containing $Y$ and $\text{dmn}(\Xi)$ as subcomplexes,
            \item a map $\tilde F: [0,1]\times  \tilde W\to \tilde W$ such that the map $(t,w)\mapsto\tilde F(1-t,w)$, with $(t,w)\in [0,1]\times\tilde W$, is a strong deformation retraction of $\tilde W$ onto $Y$, and $\tilde F(0,\cdot)|_{\text{dmn}(\Xi)}$ is a surjective homotopy equivalence from $\text{dmn}(\Xi)$ onto $Y$,
            \item a Simon-Smith family \(\tilde \Xi:\tilde W\to \mathcal{S}^*(M)\) with the topological bound $(0,1)$,      
        \end{itemize}
        such that:  
        \begin{enumerate}[label=\normalfont(\arabic*)]
            \item $\tilde\Xi|_{Y}=\Psi,\;\;\tilde\Xi|_{\text{dmn}(\Xi)}=\Xi.$
            \item For each $x\in \text{dmn}(\Xi)$, the family $t\mapsto \tilde\Xi(\tilde F(t,x))$, with $t\in [0,1]$, is a relative pinch-off process.
        \end{enumerate}
    \end{prop}
\begin{comment}
\begin{proof}
        The proof is identical to the proof of Proposition 3.11 in \cite{chu2024existence}. Here we note that we homotope $\Xi$ back to $\Psi$ with ``$\bF$-homotopies" introduced in \cite{chu2024existence}. The $\bF$-metric in $\mathcal{Z}_2(M;\mathbb{Z}_2)$ does not induce a metric on the space $\mathcal{S}^*(M)$. However, the $\bF$-metric would induce a pseudometric, and therefore a topology, on $\mathcal{S}^*(M)$.
    \end{proof}   
\end{comment}
\subsection{Trace of caps and the decomposition of $\Xi$} For each $k=1,..., K-1$, during the $k$-th min-max process, there is a dichotomy on the limit min-max free boundary minimal surfaces:
\begin{enumerate} [label= \textbf{Case} \normalfont\arabic* :]
    \item Some multiplicity one annuli $A^k_1, \cdots,A^k_{n_k}$ can be detected;
    \item A multiplicity one disk $D^{k}$ can be detected.
\end{enumerate}

If Case $1$ occurs. Then for each $k$, $C^k\subset\text{dmn}(\tilde\Psi^k)$ can be decomposed into a union of $C^k_1, \cdots,C^k_{n_k}$, and let us call each $C^k_j$ as an {\em annulus cap}. Note that $[\tilde \Psi^k|_{C^k_j}]$ is $\bF$-close to $[A^k_j]$.

If Case $2$ occurs, we call $C^k$ a {\em disk cap}.  $\tilde \Psi^k|_{C^k}$ is a Simon-Smith family of genus 0 with boundary complexity $0$.

Let us now define the trace of a cap $C$ obtained at the $k$-th step of the min-max process. We define the {\em trace of $C$}, $T(C) \subseteq \text{dmn}(\Xi)$ as follows. We first define the following sets which consist the trace:
    \begin{itemize}
        \item $B_k\;:=C$, a subcomplex of $\text{dmn}(\tilde\Psi^k)=\tilde Y^k$;
        \item $B_{l+1}:=(F^{l+1}(0,\cdot))^{-1}(B_l\cap Y^l)\subset W^{l+1}$ for $k \le l < K$.
    \end{itemize}
Then we define the trace of $C$ by the union of sets we defined above
    \[
        T(C):=B_k\cup B_{k+1}\cup\cdots\cup B_K\subset\text{dmn}(\Xi)\,.
    \]
Now let us denote by $N$ the number of all the annulus caps we obtained throughout the first $K-1$ steps from the min-max process.  We then decompose $\operatorname{dmn}(\Xi)$ as the union of the following $(N + 1)$ subcomplexes $D_0, \cdots, D_N$ defined by:
    \begin{itemize}
        \item  $D_0$ is the union of $\text{dmn}(\Psi^K)$ and $\bigcup_C T(C)$, where $C$ ranges over the disk caps. 
        \item for each annulus cap $C$, we consider its trace $T(C)$. There are in total $N$ such traces, denoted by $D_1, \cdots, D_N$.
    \end{itemize}
It follows directly from the definition of $\Xi$ that:
\[
\operatorname{dmn}(\Xi) = D_0 \cup D_1 \cup \cdots \cup D_N.
\]
\begin{prop}\label{prop:XiProperty}
        The decompositions above satisfy the following two properties:
        \begin{enumerate}[label=\normalfont(\arabic*)]
            \item\label{item:XiD0}   $\Xi|_{\bigcup_C T(C)}$,  where $C$ ranges over the disk caps,  is a Simon-Smith family of disks. 
            \item\label{item:traceDeformRetract} For each disk cap or annulus cap $C$, there exists a strong deformation retraction of the trace $T(C)$  back to $C$ in $\text{dmn}(\Xi)$. 
        \end{enumerate}    
    \end{prop}
    \begin{proof}
        \ref{item:XiD0} follows from the definition of the relative pinch-off process and Theorem \ref{thm:mapping_cylinder}. \ref{item:traceDeformRetract} follows from the definition of $T(C)$.
    \end{proof}
\subsection{Topological arguments} \label{subsec: Topological arguments} Now we show $n \geq 3$. Similar to \cite{chu2024existence}, we define the following quantities: 
\begin{itemize}
    \item $\lambda : = [\Psi]^*(\bar{\lambda})$ where $\bar{\lambda}$ is the generator of $H^*(\mathcal{Z}_2(\mathbb{B}^3, \partial \mathbb{B}^3; \mathbb{Z}_2) ; \mathbb{Z}_2)$. 
    \item $A \coloneqq \mathbb{RP}^4 \times \mathbb{RP}^1 \subset Y$. Define $\alpha$ to be the Poincar\'e dual $PD(A)$ of $A$.
\end{itemize}
We need the following topological fact.
\begin{thm} \label{thm: nontrivial topology}
In the cohomology ring $H^*(Y ; \mathbb{Z}_2)$,
\begin{equation*}
\lambda^4 \cup \alpha^2 \neq 0.
\end{equation*}
\end{thm}
We will prove Theorem \ref{thm: nontrivial topology} in Section \ref{section:topological arguments}. Based on Theorem \ref{thm: nontrivial topology}, we now prove $N \ge 3$ by contradiction. Assume $N \le 2$. By Proposition \ref{prop:XiPsiHomotopic}, $\text{dmn}(\Xi)$ and $Y$ are homotopy equivalent, and we abuse the notation, by viewing $\lambda,\alpha$ as elements of $H^1(\text{dmn}(\Xi);\mathbb{Z}_2)$ as well. For each $j=0,1,\cdots,N$, We define $i_j$ to be the inclusion map $D_j\hookrightarrow \text{dmn}(\Xi)$. Also, given a cohomology class $\gamma$ of $\text{dmn}(\Xi)$, we denote its pullback under  
    \[
        (i_j)^*: H^*(\text{dmn}(\Xi);\mathbb{Z}_2)\to H^*(D_j;\mathbb{Z}_2)
    \]
    by $\gamma|_{D_j}$.The Lyusternik–Schnirelmann vanishing lemma \cite[Lemma 5.3]{Haslhofer2019}\ then implies the following lemma.
    \begin{lem}\label{lem:LS}The following holds:
        \begin{itemize} 
            \item If $N=2$, then one of $\lambda^4|_{D_0}, \alpha|_{D_1}, \alpha|_{D_2}$ is non-zero.
            \item If $N=1$, then one of $\lambda^4|_{D_0}, \alpha^2|_{D_1}$ is non-zero.
            \item If $N=0$, i.e. $\text{dmn}(\Xi)=D_0$, then $\lambda^4|_{D_0}$ is non-zero.
        \end{itemize}
    \end{lem}
By Proposition \ref{prop:XiPsiHomotopic}, the following equality holds on the induced maps $[\Psi]$ and $[\Xi]$ into $\mathcal{Z}_2(M, \partial M;\mathbb{Z}_2)$:
    \[
        [\Xi]^*(\bar\lambda)=[\Psi]^*(\bar\lambda)=\lambda.
    \]
Hence, we have
\[
        \lambda^4|_{D_0}=i_0^*([\Xi]^*(\bar \lambda)^4)=([\Xi\circ i_0]^*(\bar \lambda))^4=([\Xi|_{D_0}]^*(\bar \lambda ))^4,
    \]
    which is non-zero if and only if $[\Xi|_{D_0}]$ is a $4$-sweepout. First, we prove the following.
    \begin{prop}\label{prop:not4sweepout}
        $\Xi|_{D_0}$ is not a $4$-sweepout.
    \end{prop}
    \begin{proof}
In this proof, we will treat $\Xi|_{D_0}$ as a free boundary Simon-Smith family of surfaces with the topological bound $(0, 1)$ in the Euclidean ball $(\mathbb{B}^3, g_{Eucl})$  instead of $(\mathbb{B}^3, g')$. For the rest of the proof, if there is no specification from the context, we will always work with $(\mathbb{B}^3, g_{Eucl})$. Recall that $D_0$ is the union of $E:= \bigcup_CT(C)$, where $C$ runs over the disk caps, and $\text{dmn}(\Psi^K)$.

By Proposition \ref{prop:XiProperty}, $\Xi|_{D_0}$ is a union of Simon-Smith family of disks and surfaces with small areas, as we have
         \[
            \sup  \mathcal{H}^2_{g_{Eucl}}\circ\Xi|_{\text{dmn}(\Psi^K)}= \sup \mathcal{H}^2_{g_{Eucl}}\circ\Psi^K<1<\pi.
        \]
Now suppose that $\Xi|_{D_0}$ is a $4$-sweepout. By definition of the $4$-width $\omega_4(\mathbb{B}^3)$ and Proposition \ref{prop:lowerfourthbound}, we have the Simon-Smith width
\begin{equation} \label{eq: relative min-max width}
\bL(\Lambda(\Xi|_{D_0})) \geq \omega_4(\mathbb{B}^3) > \pi > \sup  \mathcal{H}^2_{g_{Eucl}}\circ\Xi|_{\text{dmn}(\Psi^K)}.  
\end{equation}

Thus if we apply the relative Simon-Smith theorem (Theorem \ref{thm:relative simon-smith min-max}) to $\Xi|_{D_0}$, relative to $\text{dmn}(\Psi^K)\subset D_0$, we will detect the multiplicity one equator disk. Thus, $\bL(\Lambda(\Xi|_{D_0})) \le \pi$, contradiction to (\ref{eq: relative min-max width}). 
    \end{proof}
    Moreover, we have the following to complete the proof of $N \ge 3$. 
    \begin{thm}\label{thm:trivialInFirsthomo} Suppose $N=1,2$.
        Then for each $j=1,\cdots,N$, the map 
        \[
            (i_j)_*:H_1(D_j;\mathbb{Z}_2)\to H_1(\text{dmn}(\Xi);\mathbb{Z}_2)
        \]
        is the zero map. 
    \end{thm}
We will prove Theorem \ref{thm:trivialInFirsthomo} in Section \ref{sec: topology of boundary two cap}. By combining Proposition \ref{prop:not4sweepout} and Theorem \ref{thm:trivialInFirsthomo}, we can prove Theorem \ref{thm:threeminimalannuli} as follows.
\begin{proof}[Proof of Theorem \ref{thm:threeminimalannuli}]
    By Proposition \ref{prop:not4sweepout} and Theorem \ref{thm:trivialInFirsthomo}, we have that $\lambda^{4}|_{D_{0}}$ is zero and  $\alpha|_{D_1}$, and $\alpha|_{D_2}$ are zero in each case. This contradicts Lemma \ref{lem:LS} and we prove that $N \ge 3$. Also since $N$ annulus caps correspond to $N$ geometrically distinct annuli, we finish the proof.
\end{proof}
\section{Min-max construction and Multiplicity One theorem} \label{Sec:proofmin-max}
In this section, we discuss the Simon-Smith min-max setup and prove Theorem \ref{thm:simon-smith min-max} and Theorem \ref{thm:relative simon-smith min-max}. In particular, we apply the recent multiplicity one theorem in free boundary Simon-Smith min-max construction by Sarnataro-Stryker-Wang-Zhou \cite{sarnataro2026existence}.
\subsection{Existence and regularity for free boundary Simon-Smith min-max construction} We adopt Chu-Li \cite[Section 4]{chu2024existence} to our free boundary Simon-Smith min-max setup for the punctate surfaces, and we mostly explain necessary modifications. For a (relative) open subset $U \subset M$, we denote the set of isotopies $\{\varphi(t)\}_{t \in [0, 1]}$ in $U$ by $\mathfrak{Is}(U)$.
\begin{defn}[$(\varepsilon, \delta)$-deformation]\label{def:epsilon-delta-def}
        Given $\varepsilon, \delta > 0$, a (relative) open set $U \subset M$, and a punctate surface $\Sigma \in \mathcal{S}(M)$, we call an isotopy $\psi \in \mathfrak{Is}(U)$ {\it an $(\varepsilon, \delta)$-deformation of $\Sigma$ in $U$} provided that: 
        \begin{enumerate}[label=\normalfont(\arabic*)]
            \item $\mathcal{H}^2(\psi(t, \Sigma)) \leq \mathcal{H}^2(\Sigma) + \delta$ for all $t \in [0, 1]$.
            \item $\mathcal{H}^2(\psi(1, \Sigma)) \leq \mathcal{H}^2(\Sigma) - \varepsilon$.
        \end{enumerate}
        We define $\mathfrak{a}(U; \varepsilon, \delta)$ to be the set of all punctate surfaces that do not admit $(\varepsilon, \delta)$-deformations in $U$.
\end{defn}
 \begin{defn}[Admissible annuli]
        Given $K \in \mathbb{N}$ and $p \in M$, a collection of annuli centered in $p$
        \[
            A(p; s_1, r_1), \cdots, A(p; s_K, r_K),
        \]
        is called {\it $K$-admissible} if $2r_{i + 1} < s_i$ for all $i = 1, \cdots, K - 1$.

        For $R \in (0, \infty]$, if $\sup_i r_i \leq R$, we will say that these annuli are {\em of outer radius at most $R$}.
    \end{defn}
Also, we here note that standard pull-tight procedure works for the punctate surface setting.
    \begin{defn}[Almost-minimizing varifold in admissible annuli]\label{def:am_aa}
        Given $R \in (0, \infty]$, a sequence of punctate surfaces $\{\Sigma_j\} \subset \mathcal{S}^*(M)$ and a varifold $V \in \mathcal{V}_2(M)$, we say that a varifold $V$ is \textit{almost minimizing in every $K$-admissible collection of annuli of outer radius at most $R$ with respect to $\{\Sigma_j\}$}, if there exists $\varepsilon_j \to 0$, $\delta_j \to 0$ such that:
        \begin{enumerate}[label=\normalfont(\arabic*)]
            \item $\bF(|\Sigma_j|, V) < \varepsilon_j$.
            \item for any $K$-admissible collection of annuli $\{A_i\}^K_{i = 1}$ each of outer radius at most $R$, and any $\Sigma_j$, 
            \[
                \Sigma_j \in \bigcup^K_{i = 1}\mathfrak{a}(A_i; \varepsilon_j, \delta_j)\,.
            \]
        \end{enumerate}
    \end{defn}
 Now we have the following statement on the regularity of almost minimizing varifolds and the existence of almost minimizing varifolds. The proof directly follows from the proof of \cite[Proposition 4.8 and 4.9]{chu2024existence}.
    \begin{prop}[Existence and regularity of varifolds almost minimizing in admissible annuli]\label{prop:exist_am_aa}
        Suppose that $X$ is a finite cubical subcomplex of some cube $I(m, k)$, $\Phi: X \to \mathcal{S}^*(M)$ is a Simon-Smith family and $\{\Phi_i\}$ is a pulled-tight sequence in $\Lambda(\Phi)$. Then for $K = K(m) := 3^{m3^m}$ and any $R \in (0, \infty]$, there exists a varifold $W \in \bC(\{\Phi_i\}) \cap \mathcal{W}_L$, where $L = \|V\|(M)$, almost minimizing in every $K$-admissible collection of annuli of outer radius at most $R$ with respect to $\{\Sigma_j\}$, such that there exists a min-max subsequence $\Sigma_j := \Phi_{i_j}(x_j)$ such that $\Sigma_j$ converges to $W$ in the varifold sense. 
    \end{prop}
    Now we proceed to the existence and regularity part of the proof of Theorem \ref{thm:simon-smith min-max}.
    \begin{proof}[Proof of the existence and regularity part of Theorem \ref{thm:simon-smith min-max}] We can apply Proposition \ref{prop:exist_am_aa} and obtain
     \[
            W \in \bC(\{\Phi_i\}) \cap \mathcal{W}_L.
        \]
    \end{proof}
\subsection{Topological bounds} In the previous section, we proved the existence of a varifold $W \in \bC(\{\Phi_i\}) \cap \mathcal{W}_L$. We now prove the existence of a min-max varifold $W$ whose support satisfies the topological bound $(\fg_{0},\fb_{0})$. $W$ is almost minimizing in every $K$-admissible collection of annuli of outer radius at most $R$ with respect to $\{\Sigma_j\}$, and let us take $R = +\infty$ in this section.
\subsubsection{Genus bound} We prove the genus bound of min-max varifolds first.
\begin{prop}[Genus bound of varifolds almost minimizing in admissible annuli]\label{prop:am_genus_bound}
        Given a Compact Riemannian $3$-manifold $(M, g)$ with strictly mean convex boundary, $K \in \mathbb{N}^+$, $\mathfrak{g}_0 \in \mathbb{N}$, $R \in (0, \infty]$, a sequence $\{\Sigma_j\} \subset \mathcal{S}(M)$ and a stationary varifold $V \in \mathcal{V}_2(M)$, suppose that:
        \begin{enumerate}[label=\normalfont(\roman*)]
            \item We can choose for each $j \in \mathbb{N}^+$ a finite set $P_j$ such that $\Sigma_j \setminus P_j$ is a smooth surface and 
            \[
                P = \lim_{j \to \infty} P_j
            \]
            in the Hausdorff sense where $P$ is a finite set in $M$;
            \item $\sup_j \mathfrak{g}(\Sigma_j) \leq \mathfrak{g}_0$;
            \item $V$ is almost minimizing in every $K$-admissible collection of annuli of outer radius at most $R$ with respect to $\{\Sigma_j\}$.
        \end{enumerate}
        Then $V = \sum^l_{i = 1} m_i |\Gamma_i| \in \mathcal{W}_{L}$ and the genus $\sum^l_{i = 1} m_i \mathfrak{g}(\Gamma_i)$
        of the support of $V$ is bounded by $\mathfrak{g}_{0}$ for $L = \|V\|(M)$.
    \end{prop}
    Our genus bound above is an adaptation of the effective genus bound for free boundary minimal surfaces \cite[Theorem 3.2]{ketover2016free} to our setting. 
     
    \begin{proof}[Proof of Proposition \ref{prop:am_genus_bound}]
    By applying the improved lifting lemma with multiplicity in \cite{ketover2019genus}, which applies to our sequence of punctate surfaces after applying away from the finite limiting singular set, and replacing neck-pinching surgeries by half neck-pinching surgeries and neck-pinching surgeries in the proof of \cite[Proposition 4.10]{chu2024existence}, we can follow the proof of \cite[Proposition 4.10]{chu2024existence} and obtain
    \begin{equation}
            \sum^l_{i = 1} m_i \mathfrak{g}(\Gamma_i) \leq \mathfrak{g}_0\,.
        \end{equation}
    \end{proof}
\subsubsection{Bound of the sum of the genus and boundary complexities} In this section, we prove that the sum of the genus and boundary complexity of min-max surfaces is bounded above by those of the surfaces in the min-max sequence. This is a general topological bound of the multiplicity one version. We here adapt the arguments by Franz-Schulz \cite{franz2023topological} to our setting.

The proof of the topological bound hinges on the tree lifting lemma which applies to a loop-free network of curves, a free boundary version of Simon's lifting lemma in De Lellis-Pellandini \cite{de2018min}, proven by Franz-Schulz in \cite[Lemma 4.8]{franz2023topological}. We take the definition of a tree and the tree lifting lemma from Definition 4.7 in \cite{franz2023topological}, which holds for punctate surfaces by applying theorems away from singular sets $P_{i}$ on each $\Sigma_{i}$.
\begin{prop}[The upper bound of the sum of the genus and boundary complexities]\label{prop:am_boundary_bound}
        Given a Riemannian manifold $(M, g)$ with a strictly mean convex boundary, $K \in \mathbb{N}^+$, $\mathfrak{g}_0, \mathfrak{b}_{0} \in \mathbb{N}$, $R \in (0, \infty]$, a sequence $\{\Sigma_j\} \subset \mathcal{S}(M)$ and a stationary varifold $V \in \mathcal{V}_2(M)$, suppose that:
        \begin{enumerate}[label=\normalfont(\roman*)]
            \item We can choose for each $j \in \mathbb{N}^+$ a finite set $P_j$ such that $\Sigma_j \setminus P_j$ is a smooth surface and 
            \[
                P = \lim_{j \to \infty} P_j
            \]
            in the Hausdorff sense where $P$ is a finite set in $M$;
            \item $\sup_j \mathfrak{g}(\Sigma_j) \leq \mathfrak{g}_0$;
            \item $\sup_j \mathfrak{g}(\Sigma_j) + \mathfrak{b}(\Sigma_{j}) \leq \mathfrak{g}_0+\mathfrak{b}_0$;
            \item $V$ is almost minimizing in every $K$-admissible collection of annuli of outer radius at most $R$ with respect to $\{\Sigma_j\}$.
        \end{enumerate}
        Then $V \in \mathcal{W}_{L, \leq \mathfrak{g}_0, \leq \mathfrak{g}_0+\mathfrak{b}_{0}}$ for $L = \|V\|(M)$.
    \end{prop}

  \begin{proof}
     In Proposition \ref{prop:am_genus_bound}, we proved that $V \in \mathcal{W}_L$ and the genus of the support of $V$ is bounded by $\mathfrak{g}_{0}$ with $V = m_1|\Gamma_1| + \cdots + m_l |\Gamma_l|$ for a disjoint set of smooth, connected, embedded free boundary minimal surfaces $\{\Gamma_i\}^l_{i = 1}$ and a set of positive integers $\{m_i\}^l_{i = 1}$ whose genus is bounded by $\mathfrak{g}_{0}$.

     For each $\Gamma_{i}$, choose the finite collection of simple closed curves and properly embedded rooted trees used there. After a small perturbation, these curves and trees may be chosen away from the finite limiting puncture set $P$ and from the finite set where graphical convergence may fail. Hence, the improved lifting argument for the closed curves and the tree-lifting argument of \cite[Lemma 4.8]{franz2023topological} apply exactly as in the smooth setting, since the sequence is smooth in a neighborhood of the chosen curves and trees.
        
 In this procedure, we obtain sequences of punctate surfaces $\widetilde{\Sigma}_{j} \rightarrow V$ as varifold, which is obtained by neck-pinching and half neck-pinching processes from $\Sigma_{j}$ as in \cite[Proposition 4.10]{chu2024existence}. Since the surfaces $\widetilde\Sigma_j$ are not necessarily smooth, we replace the small neighborhood of punctate set by finitely many disks and half-disks through surgeries and the removal of connected components and still have smooth surfaces $\widetilde\Sigma_j' \rightarrow V$. Then we obtain a sequence of properly embedded smooth surfaces $\{\tilde\Sigma'_j\}$ such that, for all sufficiently large $j$, $|\widetilde\Sigma'_j|\to V\text{ as varifolds}$. $\fg + \fb$ does not increase over these procedures by \cite[Lemma 3.6]{franz2023topological}.

With this sequence of smooth surfaces $\{\tilde\Sigma'_j\}$, we apply the proof of~\cite[Theorem~4.11]{franz2023topological} verbatim to conclude that
\[
\mathfrak g(\Gamma)+\mathfrak b(\Gamma)\le \mathfrak g_0+\mathfrak b_0.
\]
Together with Proposition~\ref{prop:am_genus_bound}, this implies
\[
V\in \mathcal W_{L,\le \mathfrak g_0,\le \mathfrak g_0+\mathfrak b_0}.
\qedhere
\]
  \end{proof}
\subsection{Multiplicity one theorem of unstable min-max free boundary minimal surfaces in Simon-Smith setting} We apply the Multiplicity One Theorem by Sarnataro-Stryker-Wang-Zhou \cite[Theorem 1.3]{sarnataro2026existence} to our free boundary setting. The proof follows from the free boundary PMC min-max theorem therein (See Wang-Zhou \cite{wang2023existence} for the closed Simon-Smith setting, and Sun-Wang-Zhou \cite{sun2024multiplicity}). We now prove Theorem \ref{thm:simon-smith min-max} (2) based on \cite[Theorem 1.3]{sarnataro2026existence}.

\begin{rem}
    Sarnataro-Stryker-Wang-Zhou \cite{sarnataro2026existence} proved the multiplicity one theorem where Simon-Smith families consist of properly embedded separating surfaces. In the closed setting, Chu-Li \cite[Section 4.4]{chu2024existence} adapted the Multiplicity One Theorem for min-max surfaces by Wang-Zhou \cite[Theorem 7.3]{wang2023existence} to the punctate surfaces. The same adaptation works by reproving the $C^{1,1}$-regularity of free boundary PMC surfaces for punctate surfaces and deriving the multiplicity one theorem for min-max free boundary surfaces in Simon-Smith setting. We therefore use this extension without further proof.
\end{rem}

\begin{proof}[Proof of Theorem \ref{thm:simon-smith min-max} \normalfont{(2)}]
By the multiplicity-one theorem for the free boundary Simon--Smith min--max theory \cite[Theorem 3.5]{sarnataro2026existence}, the min--max varifold can be written as
\[
    V=\sum_j m_j|\Gamma_j|,
\]
where the $\Gamma_j$'s are pairwise disjoint connected smooth properly embedded minimal surfaces with possibly empty free boundary. Moreover, if $\Gamma_j$ is two-sided and unstable, then $m_j=1$, while if $\Gamma_j$ is one-sided, then its connected double cover is stable. The topological bounds follow from Proposition \ref{prop:am_boundary_bound}.
\end{proof}
\subsection{Surfaces near the critical set} We proved $\bC(\{\Phi_{i}\})\cap \mathcal{W}_{L, \leq \mathfrak{g}_0,\le \fg_{0}+\fb_{0}} \neq \emptyset$ in the previous sections. However, even together with the pulled-tight sequence, this does not give the following conclusion in Theorem \ref{thm:simon-smith min-max} (3) directly: There exists $\eta > 0$ such that for all sufficiently large $i$,
		\[
			\mathcal{H}^2(\Phi_{i}(x)) \geq L - \eta \implies |\Phi_i(x)| \in \bB^{\bF}_{r}(\mathcal{W}_{L, \leq \mathfrak{g}_0,\le \fg_{0} + \fb_{0}}\cap \bC(\{\Phi_{i}\}))\,.
		\]
We show this by applying arguments in Section 4.5 in \cite{chu2024existence}, which is a refinement from Marques-Neves \cite[Theorem 4.7]{marques2021morse}. Note that these arguments follow from Pitts' combinatorial argument \cite[Theorem 4.10]{pitts2014existence}. Combining this with our proof of Theorem \ref{thm:simon-smith min-max} (1), (2) in the previous sections, we complete the proof of Theorem \ref{thm:simon-smith min-max}.
\begin{proof}[Proof of Theorem \ref{thm:simon-smith min-max}] Theorem \ref{thm:simon-smith min-max} (3) follows from the argument of \cite[Theorem~2.16]{chu2024existence} verbatim once we replace notions in the closed setting by those in the free boundary setting. Combining this with the proof of the existence and regularity part of Theorem \ref{thm:simon-smith min-max}, we finish the proof of Theorem \ref{thm:simon-smith min-max}.
\end{proof}
\subsection{Proof of Theorem \ref{thm:relative simon-smith min-max}} We complete the proof of Theorem \ref{thm:relative simon-smith min-max} as in \cite[Theorem 2.21]{chu2024existence}.
\begin{proof}[Proof of Theorem \ref{thm:relative simon-smith min-max}]
With the same argument as \cite[Section 4.6]{chu2024existence}, for each $\Phi' \in \Lambda_Z(\Phi)$, one can restrict all the deformations in the previous subsections to occur away from the compact set
    \[
        X' := \left\{x \in X : \mathcal{H}^2(\Phi'(x)) \leq \frac{\bL(\Lambda_Z(\Phi)) + \sup_{z \in Z} \mathcal{H}^2(\Phi'(z))}{2}\right\}\,,
        \]
since $\bL(\Lambda_Z(\Phi)) > \sup_{z \in Z} \mathcal{H}^2(\Phi(z))$. By smoothing deformations, we obtain the desired conclusion.
\end{proof}

\section{Perturbing the metric} \label{sec:perturbmetric}
In this section, we show Proposition \ref{prop:perturbMetric}. Let $(M, g_0)$ be a compact $3$-manifold with nonnegative Ricci and strictly convex boundary $\partial M$ and $L > 0$ be as chosen in \S \ref{subsec: Repetitive min-max scheme}.  We have the following lemma.
\begin{lem} \label{lem:eo}
Suppose $(M, g_0)$ contains only finitely many free boundary minimal annuli, given by $\mathcal{T}:= \{ A_1, \cdots, A_l \}$. Then for any $L > 0$, there exists a (relatively) open subset $U$ containing $\bigcup_{1 \leq j \leq l} A_j$ in $M$ and an open neighborhood $\mathcal{O}$ of $g_0$ in $C^{\infty}$ such that for every $g' \in \mathcal{O}$,
\begin{enumerate}
    \item Every properly embedded free boundary disk with $g'$-area less than $L$ in $M$ is not entirely contained in $U$.
    \item If we further assume $g'|_U = g_0|_U$, then every properly embedded free boundary minimal annulus $A$ with $g'$-area less than $L$ in $(M, g')$ belongs to $\mathcal{T}$.
\end{enumerate}
\end{lem}
\begin{proof}
    The proof is the same as \cite[Lemma 6.4]{chu2024existence}.
\end{proof}
Now we follow Ambrozio-Carlotto-Sharp \cite{ambrozio2018compactness} to discuss the bumpy metric in our free boundary setting. Let $3\leq q\leq \infty$ be an integer or infinity, and fix $L>\sup_j \operatorname{area}_{g_0}(\Sigma_j)$. We take $U$ and $\mathcal{O}$ from Lemma \ref{lem:eo}. We set
\begin{align*}
        \Gamma^q & := \{C^q\text{-metric }g'\in \mathcal{O}\text{ conformal to }g: g'|_U = g_0|_U\} \\ 
        \mathcal{M}^q & := \{(g, D): g\in \Gamma^q\,;\; D \text{ is a free boundary minimal disk in }(M^3, g), \text{ area}_g(D)<L\}\,;
    \end{align*}
Let us use $C^q$-topology on Riemannian metrics and $C^{2,\alpha}$-topology on embedding of submanifolds where $\alpha\in (0,1)$.
\begin{lem}\label{lem:fredholmfb}
$\mathcal M^q$ is a separable $C^{q-2}$ Banach manifold, and the projection
\[
    \Pi:\mathcal M^q\to \Gamma^q,\qquad (g,\Sigma)\mapsto g,
\]
is a $C^{q-2}$ Fredholm map of index $0$.

Moreover, $g\in\Gamma^q$ is a regular value of $\Pi$ if and only if every $(g,\Sigma)\in \mathcal M^q$ is non-degenerate.
\end{lem}
\begin{proof}
    For every $(\hat{g}, \hat{\Sigma})\in \mathcal{M}^q$ with unit normal field $\hat{\nu}$, we apply~\cite[Proposition~41, 45, 46]{ambrozio2018compactness} with:
    \begin{align*}
            & \Gamma := \Gamma^q\,, \;\;\;\;\;  X := C^{2,\alpha}(\hat{D})\,,  \;\;\;\;\;  Y:= C^{0,\alpha}(\hat{D}) \times C^{1,\alpha}(\partial \hat{D})\,, \;\;\;\;\;  \mathcal{H}:= L^2(\hat{D}) \times L^{2}(\partial \hat{D})\,, \\
            & F: \Gamma\times X \to \mathbb{R} \,, \;\; (g, w) \mapsto \text{area}_g(\text{graph}_{\hat{D}, \hat{g}}(w))\,, \\
            & (H,\Theta): \Gamma\times X\to Y\,, \; \text{ given by } 
\left\langle (H,\Theta)(g,w),(v,v|_{\partial\hat D})\right\rangle_{\mathcal{H}} = \left.\frac{d}{dt}\right|_{t=0}F(g,w+tv) \,,
        \end{align*}
        where $(H, \Theta)$ is characterized by the first variation formula and $H$ is the mean curvature and $\Theta$ is the intersection angle function in \cite{ambrozio2018compactness}.

 The space $\Gamma^q$ can be identified, after possibly shrinking the neighborhood $\mathcal O$ of $g_0$, with an open subset of the Banach space
\[
    C^q_U(M) := \{ f\in C^q(M): f|_U=0\}
\]
by the map $f\mapsto e^{2f}g_0$ where we endow $C^q_U(M)$ with the $C^q$-topology. Near $(\hat g,\hat D)$, we have
\begin{equation} \label{eq: level set of M_q}
\mathcal M^q = \{ (g, w) \in \Gamma \times X : (H, \Theta)(g, w) = 0 \}.    
\end{equation}
By ~\cite[Proposition~45]{ambrozio2018compactness}, the operator $\mathcal{F}_{\hat D}(\hat g,0): X \to Y$ given by
\[
    \mathcal{F}_{\hat D} (\hat g,0)(w) := \bigl(-L_{\hat D,\hat g}w,\, B_{\hat D,\hat g}w|_{\partial \hat D}\bigr),
\]
where
\[
    L_{\hat D,\hat g} = \Delta_{\hat D,\hat g} + |A_{\hat D,\hat g}|^2+ \operatorname{Ric}_{\hat g}(\hat\nu,\hat\nu),
\]
and
\[
    B_{\hat D,\hat g}w = \partial_{\hat v}w + h^{\partial M}_{\hat g}(\hat\nu,\hat\nu)w \quad\text{on }\partial\hat D,
\]
is a Fredholm map of index $0$.
Now by \cite[Proposition 46]{ambrozio2018compactness}, we proceed the same as \cite[Lemma 6.5]{chu2024existence} to finish the proof.
\begin{comment}
The Fredholm statement follows from \cite{ambrozio2018compactness} by applying the arguments to $\Gamma_{q}$. Let $0\neq \kappa \in  \ker \mathcal{F}_{\hat D}(\hat g,0)$. By Lemma \ref{lem:eo}, $\hat D \setminus\overline U\neq\emptyset$. The unique continuation argument implies that $\kappa$ is not identically zero on $\hat D\setminus U$. For a conformal variation \(g_s=(1+sf)g\), with \(f\in C_c^q(M\setminus U)\), the variation formula in \cite{ambrozio2018compactness} gives
\[
\left.\frac d{ds}\right|_{s=0} \langle(H,\Theta)(g_s,0),(\kappa,\kappa|{\partial D})\rangle =\int_D\nu(f)\kappa d\mu.
\]
By choosing $f$ so that $\int_D\nu(f)\kappa d\mu$ is nonzero, this gives that we could verify the hypothesis needed to apply \cite[Proposition 46]{ambrozio2018compactness}. Hence, we have that $\mathcal M^q$ is a separable $C^{q-2}$ Banach manifold, $\Pi:\mathcal M^q\to\Gamma^q$ is Fredholm of index zero, and its regular values are precisely the metrics for which all such free boundary minimal disks are nondegenerate.   
\end{comment}
\end{proof}
\begin{proof}[Proof of Proposition \ref{prop:perturbMetric}]
\begin{comment}
By the free-boundary compactness theorem and the standard Jacobi-field argument, any sequence of connected free-boundary minimal surfaces with uniformly bounded area and topology converges, after passing to a subsequence, smoothly and with multiplicity one, along the varying metric. Also note that any properly embedded surface in $(M,g)$ is orientable.

 We take $U\subset M$ and the neighborhood $\mathcal{O} \ni g_0$ from Lemma~\ref{lem:eo}. By Lemma \ref{lem:fredholmfb} and Sard–Smale theorem, regular values are dense in $\Gamma^q$. By compactness and nondegeneracy, the set of such smooth metrics is open and dense in $\Gamma^\infty$.
 
 Given $\epsilon > 0$, we can take $g_{0}' \in \Gamma^q_{bumpy} \cap \Gamma_{\infty}$ that is $\frac{\epsilon}{2}$-close to $g_0$. By the same arguments as the proof of Proposition 6.1 in \cite{chu2024existence} and the fact that $\Gamma^q_{bumpy} \cap \Gamma_{\infty}$ is open and dense in $\Gamma_{\infty}$, we can find $g'$ that is $\frac{\epsilon}{2}$-close to $g_0'$ that satisfies \normalfont{(3)} and \normalfont{(4)}. As a result, 
\end{comment}
By Lemma \ref{lem:fredholmfb} and the compactness theorem \cite[Theorem 2]{ambrozio2018compactness}, we can proceed verbatim as the proof of  \cite[Proposition 6.1]{chu2024existence} to show that metrics $g'$ satisfying assertions \normalfont{(1)}-\normalfont{(4)} of Proposition \ref{prop:perturbMetric} can be taken arbitrarily close to $g_0$. 

Now it suffices to show that there exists a neighborhood $\mathcal{V}$ of $g_0$ such that assertions \normalfont{(5)}-\normalfont{(7)} hold for any metric $g \in \mathcal{V}$, i.e., assertions \normalfont{(5)}-\normalfont{(7)} are open conditions.

By the curvature estimates for stable minimal surfaces \cite[Theorem 1.3]{DeLellis2013} and the compactness theorem \cite[Theorem 3]{White1987}, \normalfont{(5)} and \normalfont{(7)} are open conditions.

It remains to verify that assertion \normalfont{(6)} is also open. Suppose not, then there exists a sequence of metrics $g_i\to g_0$ and two disjoint connected properly embedded $g_i$-free boundary minimal surfaces $\Sigma_i^1$ and $\Sigma_i^2$ satisfying the topological bound $(0, 1)$. By the compactness theorem of free boundary minimal surfaces, after passing to a subsequence, $\{\Sigma_i^j \}_i$ converges smoothly and with multiplicity one to a connected properly embedded $g_0$-free boundary minimal surface $\Sigma_\infty^j$, for $j=1,2$.

If $\Sigma_\infty^1\neq \Sigma_\infty^2$, then by the Frankel property for $(M,g_0)$ the two limits intersect. By the maximum principle, the two distinct embedded free boundary minimal surfaces must intersect transversely, contradicting the smooth convergence of $\Sigma_i^1 \cup \Sigma_i^2$ to $\Sigma_{\infty}^1 \cup \Sigma_{\infty}^2$.

Now assume $\Sigma_\infty^1=\Sigma_\infty^2$, by the graphical convergence, the normalized separation between two disjoint sheets gives a positive Jacobi field on the limit surface, hence the limit is (degenerate) stable, contradicting the nonexistence of stable free boundary minimal surfaces in the metric $g_0$. Hence \normalfont{(6)} is an open condition. This gives that there exists $g'$ arbitrarily close to $g_0$ which satisfies conditions (1)-(7).
\end{proof}
\section{Boundary version of the interpolation} \label{boundaryinterpolation}
We are going to prove the following boundary version of the interpolation result \cite[Proposition 8.1]{chu2024existence}.

\begin{prop}[Boundary interpolation near a disk]\label{prop:boundary_interpolation}
Let $X$ be a cubical subcomplex of $I(m,k)$, $\mathcal D$ be a compact set of properly embedded disks in $M$, $q\in\mathbb N$, and $\epsilon \in (0,1)$. Then there exists $\delta=\delta(\epsilon,q,m,\mathcal D)>0$ with the following property.

Suppose that $\Phi:X\to \mathcal S^*(M)$ is a Simon-Smith family satisfying the topological bound $(0,1)$ and $N_P(\Phi)<q$, and suppose that there is a map $D:X\to \mathcal D$ such that
\[
    \mathbf F([\Phi(x)],[D(x)])<\delta
\]
for every $x\in X$. Then there exists a Simon-Smith family $H:[0,1]\times X\to \mathcal S^*(M)$ satisfying the same topological bound $(0,1)$ such that
\begin{enumerate}[label=\normalfont(\arabic*)]
    \item $H(0,x)=\Phi(x)$ and $\fb(H(1,x))=0$ for every $x\in X$;
    \item for every $x\in X$ and $0\leq t\leq t'\leq 1$,
    \[
        \mathcal H^2(H(t,x))\leq \mathcal H^2(\Phi(x))+\epsilon,
        \qquad
        \fb(H(t',x))\leq \fb(H(t,x));
    \]
    \item for each $x\in X$, the path $t\mapsto H(t,x)$ is a relative pinch-off process.
\end{enumerate}
\end{prop}
\subsection{Nontrivial essential arc or short closed loop in a small ball} In this section, we prove that if an annulus (possibly with singularities) is sufficiently close to a disk, then we can find an essential arc or a short simple closed loop contained within a small ball by adapting Section 8.1 in \cite{chu2024existence}.
\begin{thm}[The Loewner inequality for annuli]\label{thm:relativesystolic}
Let \(\Sigma \subset M\) be a smooth properly embedded annulus. We define
\[
\begin{split}
\operatorname{sys}(\Sigma) & : = \inf \{ \mathcal{H}^1(\gamma) : [\gamma] \neq 0 \in H_1(\Sigma ; \mathbb{Z}) \}, \\
\operatorname{sys}_{\operatorname{rel}}(\Sigma) & : = \inf \{ \mathcal{H}^1(\gamma) : [\gamma] \neq 0 \in H_1(\Sigma, \partial \Sigma ; \mathbb{Z}) \}.
\end{split}
\]

Then
\begin{equation}
    \min \{ \operatorname{sys}(\Sigma), \operatorname{sys}_{\operatorname{rel}}(\Sigma) \} \leq \sqrt{\mathcal{H}^2(\Sigma)}.
\end{equation}

In particular, one of the following holds:
\begin{enumerate}[label=\normalfont(\arabic*)]
    \item There exists a simple closed curve
    $\gamma\subset \Sigma$ with
    \[
        [\gamma]\neq 0\in H_1(\Sigma;\mathbb Z_2), \qquad \mathcal H^1(\gamma)\le \sqrt{\mathcal H^2(\Sigma)}.
    \]
    \item There exists a properly embedded simple arc  \(\gamma\subset \Sigma\) $(\text{an essential arc})$ with
    \[
        [\gamma]\neq 0\in H_1(\Sigma,\partial\Sigma;\mathbb Z_2), \qquad \mathcal H^1(\gamma)\le \sqrt{\mathcal H^2(\Sigma)}.
    \]
\end{enumerate}
\end{thm}

\begin{proof}
Let $\Gamma_0, \Gamma_1$ be the two boundary components of $\Sigma$. Define
\( d \coloneqq \operatorname{d}_{\Sigma}(\Gamma_0, \Gamma_1)\)
to be the intrinsic distance between $\Gamma_0, \Gamma_1$. Then we have $d \geq \operatorname{sys}_{\operatorname{rel}}(\Sigma)$.

On the other hand, consider the intrinsic distance function \(h(x) := \operatorname{d}_\Sigma(x, \Gamma_0) \quad \text{for }x \in \Sigma.\) By \textit{Sard's} Theorem, for almost every $t \in (0, d)$, the level set $h^{-1} \{ t\}$ is a finite union of simple closed curves that separates $\Gamma_0$ from $\Gamma_1$, and hence at least one connected component is nontrivial in $H_1(\Sigma ;\mathbb{Z})$. Thus, by definition, we have
\begin{equation} \label{eq: horizonal}
    \mathcal{H}^1(h^{-1} \{ t\}) \geq \operatorname{sys}(\Sigma)
\end{equation}
for almost every $t \in (0, d)$. By  \eqref{eq: horizonal}, and the coarea formula, we have
\begin{equation} \label{eq: systolic inequlaity}
\begin{split}
\mathcal{H}^2(\Sigma) & = \int_{0}^d \int_{h^{-1} \{ t\}} \frac{1}{|\nabla_{\Sigma} h(x)|}  \, d\mathcal{H}^1(x) \, dt 
 \geq \int_0^d \mathcal{H}^1(h^{-1} \{ t\}) \, dt 
 \geq d \cdot \operatorname{sys}(\Sigma)
 \geq \operatorname{sys}_{\operatorname{rel}}(\Sigma) \cdot \operatorname{sys}(\Sigma).
\end{split} 
\end{equation}
Therefore, the claim follows from (\ref{eq: systolic inequlaity}).
\end{proof}

Now we show the existence of an essential arc or a simple closed curve when an annulus (possibly with singularities) is sufficiently close to a disk.
\begin{lem}[Existence of a small essential arc or a simple closed curve on smooth annuli] \label{lem:boundarysmallcurve}
Let $D\subset (M,\partial M)$ be a properly embedded disk. For any $\epsilon>0$, there exists $\delta>0$ such that if $\Sigma$ is a properly embedded annulus with
\[
    \mathbf F([\Sigma],[D])<\delta,
\]
then there exists either
\begin{enumerate}[label=\normalfont(\arabic*)]
    \item $p \in \operatorname{Int}(M)$ and a simple closed curve $\gamma\subset \Sigma\cap B_{\epsilon}(p)$ with $[\gamma]\neq0\in H_1(\Sigma;\mathbb Z_2)$, or
    \item $p \in \partial M$ and a properly embedded simple arc $\gamma\subset \Sigma\cap B_{\epsilon}(p)$ with
    $[\gamma]\neq 0\in H_1(\Sigma,\partial\Sigma;\mathbb Z_2)$.
\end{enumerate}
In either case, a union of disk components is obtained by the corresponding interior or boundary neck-pinch.
\end{lem}
\begin{proof}
We can adapt the proof of \cite[Lemma 8.3]{chu2024existence}. Triangulate the disk $D$ using prisms in the interior and half-prisms near $\partial M$. After cutting $\Sigma$ along the resulting short closed curves and proper arcs, Theorem \ref{thm:relativesystolic} implies that any component which is not a disk contains a short essential closed curve or a short essential proper arc. If every component is a disk, then one of the cutting curves or arcs is essential in $\Sigma$. We follow the remainder of the argument in \cite[Lemma 8.3]{chu2024existence}, after replacing disks by half-disks near boundary.

Finally, an essential proper arc in an annulus joins its two boundary components, so the corresponding boundary neck-pinch process produces only disk components. The same holds for an interior neck-pinch along an essential simple closed curve. This proves the claim.
\end{proof}

We obtain a small essential arc or a simple closed curve on punctate surface as well by an approximation argument, where the proof remains the same as in Lemma 8.4 and 8.5 in \cite{chu2024existence}.
\begin{lem}[Existence of a small essential arc or a simple closed curve on punctate annuli] \label{lem:small_essential_piece_punctate}
Let $\mathcal D$ be a compact subset of smooth properly embedded disks in $M$, endowed with the smooth topology. Let $q\in\mathbb N$ and $\epsilon\in(0,1)$. Then there exists $\delta=\delta(\mathcal D,q,\epsilon)>0$ with the following property.

Assume that $\Sigma\in\mathcal S^*(M)$ is a punctate surface with punctate set $P$, $\#P\le q$, and
\[
    \fg(\Sigma)=0,\qquad \fb(\Sigma)=1,
\]
and there exists some $D\in\mathcal D$ such that
\[
\mathbf F([\Sigma],[D])\le \delta
\]
Then there exists an admissible ball or half-ball $B_\epsilon(p)$ and an embedded essential curve $\gamma\subset \Sigma\cap \mathcal B_\epsilon(p)$ of one of the following two types:
\begin{enumerate}[label=\normalfont(\arabic*)]
    \item $\gamma$ is a simple closed curve with
    \[
        [\gamma]\neq 0\in H_1(\Sigma;\mathbb Z_2);
    \]

    \item $\gamma$ is a properly embedded simple arc with
    \[
        [\gamma]\neq 0\in H_1(\Sigma,\partial\Sigma;\mathbb Z_2).
    \]
\end{enumerate}
In either case, a union of disk components is obtained by the corresponding interior or boundary neck-pinch.
\end{lem}
\begin{proof}
    The proof is identical to \cite[Lemmata 8.4 and 8.5]{chu2024existence}, replacing \cite[Lemma 8.3]{chu2024existence} by Lemma \ref{lem:boundarysmallcurve} and the arguments on balls by balls and half-balls.
\end{proof}
\subsection{Local deformations in the boundary setting} In the previous subsection, we detected a small essential arc or a simple closed curve for singular annuli within small balls, when they are sufficiently close to a properly embedded disk. We now construct pinching and shrinking deformations in a continuous manner along a Simon-Smith family as \cite[Section 8.2]{chu2024existence} with necessary modifications.
\begin{lem}[Relative local deformation A: pinching]\label{lem:relative-local-pinching}
There exists a constant $C=C(M,g)>0$ with the following property. 

Let $\Lambda>0$, $0<r<\operatorname{injrad}(M,g)/4$, and $p\in M$. Suppose that $\Phi:X\to \mathcal S^*(M)$ is a Simon--Smith family and $x_0\in X$ is such that $\Phi(x_0)\in \mathcal S^{*}(M)$ and
\[
    \mathcal H^2\big(\Phi(x_0)\cap A_M(p;r,2r)\big)
    \leq \Lambda r^2 .
\]
Then there exist $s\in(r,2r)$, $\zeta\in (0,1/5\min\{\Lambda r,s-r,2r-s\})$, an integer $K\geq 0$, a neighborhood $O_{x_0}\subset X$ of $x_0$, and a Simon--Smith family
\[
    H:[0,1]\times O_{x_0}\to \mathcal S^{*}(M)
\]
with the following properties.

\begin{enumerate}[label=\normalfont(\arabic*)]
    \item For every $y\in O_{x_0}$, the surface
    \[
        \Sigma_y:=\Phi(y)\cap A_M(p;s-5\zeta,s+5\zeta)
    \]
    is smooth and varies smoothly in $y$.
    
    \item  For every $s'\in[s-4\zeta,s+4\zeta]$, $\Phi(y)$ intersects $\partial B(p,s')\cap M$ transversely, and
    \[
        \Phi(y)\cap \big(\partial B(p,s')\cap M\big) = \bigsqcup_{i=1}^K \gamma_i(s',y),
    \]
    where each $\gamma_i(s',y)$ is either a simple closed curve or a properly embedded arc with endpoints on $\partial M$. Furthermore,
    \[
        \sum_{i=1}^K \mathcal H^1(\gamma_i(s',y)) \leq 4\Lambda r .
    \]
    \item[If ]$K=0$, then $H(t,y)=\Phi(y)$ for all $(t,y)\in[0,1]\times O_{x_0}$. If $K\geq 1$, then
    \item For every $y \in O_{x_0}$, $H(0, y) = \Phi(y)$.
    \item there exists $0=t_0<t_1<\cdots<t_K<t_{K+1}=1$ such that for every $0\leq i\leq K$, every $t\in(t_i,t_{i+1})$, and every $y\in O_{x_0}$,
    \[
        \Sigma_{t,y}:=H(t,y)\cap A_M(p;s-5\zeta,s+5\zeta)
    \]
    is obtained from $\Sigma_y$ by performing $i$ times of neck-pinches along $\gamma_1(s,y),\ldots,\gamma_i(s,y)$. More precisely,for each neck-pinch, either a closed curve component is replaced by two disks or a properly embedded arc component is replaced by two half-disks with boundary on $\partial M$. Moreover, we have
    \[
        H(1,y)\cap A_M(p;s-\zeta,s+\zeta)=\emptyset .
    \]
    \item For each $i=1,\ldots,K$, near $t_i$, the family $H(t,y)$ is a surgery process via pinching the cylinder or half-cylinder over $\bigcup_{s'\in(s-2\zeta,s+2\zeta)}\gamma_i(s',y)$.

    \item The deformation is supported in the slightly smaller annulus and has an area control:
    \begin{align*}
        H(t,y)\setminus A_M(p;s-3\zeta,s+3\zeta) &= \Phi(y)\setminus A_M(p;s-3\zeta,s+3\zeta) \\
        \mathcal H^2(H(t,y)) &\leq \mathcal H^2(\Phi(y))+C\Lambda^2 r^2
    \end{align*}
    for every \((t,y)\in[0,1]\times O_{x_0}\).
\end{enumerate}
\end{lem}
\begin{proof}
The proof is the same as \cite[Lemma 8.6]{chu2024existence}, with minor modifications near $\partial M$. The level spheres are replaced by half-spheres, and an isoperimetric inequality is replaced by a relative isoperimetric inequality near boundary; see, for instance, Ritorè \cite{ritore2023isoperimetric}. Closed curves are filled by disks as in the closed case, while properly embedded arcs are filled by half-disks in the hemispheres. The area estimates are unchanged up to a constant depending only on $(M,g)$.
\end{proof}

\subsection{Proof of Proposition \ref{prop:boundary_interpolation} and Theorem \ref{thm:mapping_cylinder}}
\begin{proof}[Proof of Proposition \ref{prop:boundary_interpolation}]
It follows the proof of \cite[Proposition 8.1]{chu2024existence}. We here explain the modifications in the free boundary setting. Throughout the proof, balls and annuli are replaced by relative balls and annuli. Note that the combinatorial arguments \cite[Lemmata 8.9--8.11]{chu2024existence} can be applied for these admissible balls, since the arguments only rely on the combinatorial selection arguments.

Choose constants $0<\epsilon_2\ll \epsilon_1\ll \epsilon$ as in \cite[Proposition 8.1]{chu2024existence}, with \(\epsilon_1\) also sufficiently small within the neighborhood where Fermi coordinate is well-defined. Also, the annular area estimate holds for sufficiently small $\delta>0$ for all admissible annuli:
\[
    \sup_{p\in M}\mathcal H^2\bigl(\Phi(x)\cap \mathcal A(p;r,2r)\bigr)\le 100r^2,
    \qquad r\in(\epsilon_2,\epsilon_1).
\]
Moreover, by Lemma~\ref{lem:small_essential_piece_punctate}, whenever $\fb(\Phi(x))=1$, there exists a point $p_x\in M$ such that $\Phi(x)\cap\mathcal B_{\epsilon_2}(p_x)$ contains either a simple closed curve nontrivial in $H_1(\Phi(x);\mathbb Z_2)$, or a properly embedded simple arc nontrivial in $H_1(\Phi(x),\partial\Phi(x);\mathbb Z_2)$. We add this point to the punctate set.

We then apply \cite[Lemmata 8.10 and 8.11]{chu2024existence} to choose,
after subdividing \(X\), pairwise disjoint or properly nested admissible
annuli around the points in \(P(x)\). On each such annulus, we apply
Lemma~\ref{lem:relative-local-pinching}. Along the relative neck-pinch, the local area increase is bounded by $C r^2$, and the constants are chosen so that the total increase over all cells is at most \(\epsilon\). We apply a relative version of \cite[Lemma 8.8]{chu2024existence} for the shrinking process.

By adapting the proof of \cite[Proposition 8.1]{chu2024existence} and constructing the homotopy $H$ by combining the local deformations as in Step 3 of \cite[Proposition 8.1]{chu2024existence}, the area is controlled along the homotopy i.e. $\mathcal H^2(H(t,x))\le \mathcal H^2(\Phi(x))+\epsilon$ for all $t\in[0,1]$ and we have $\fb(H(1,x))=0$, and the boundary complexity is nonincreasing along the homotopy. Therefore \(H\) satisfies all the claimed properties.
\end{proof}
\begin{proof}[Proof of Theorem~\ref{thm:mapping_cylinder}]
We explain the modification from \cite[Theorem 3.6]{chu2024existence}, replacing the interpolation theorem near a sphere by Proposition~\ref{prop:boundary_interpolation}.

Let $\Gamma$ be the unique free boundary minimal disk such that $|\Gamma|\in\mathcal W_{L,\leq 0,\leq 1}$. By Theorem~\ref{thm:currentsCloseInBoldF}, after replacing $\Phi$ by a family in the same homotopy class and refining $X$, we may assume that there exist $r>0$ and $\eta_0>0$ such that
\begin{equation} \label{eqn:nearmaximal}
    \mathcal H^2(\Phi(x))\ge L-\eta_0 \quad\Longrightarrow\quad \mathbf F([\Phi(x)],[\Gamma])<r .
\end{equation}
We choose $r>0$ sufficiently small so that Proposition~\ref{prop:boundary_interpolation} applies to the compact family $\mathcal D=\{\Gamma\}$. We can construct surjective cubical homotopy equivalence $f:X'\to X$ and a family $\Phi':X'\to\mathcal S^*(M)$, where $\Phi'(x)$ satisfies the topological bound $(0, 1)$ for any $x$ satisfying (\ref{eqn:nearmaximal}). By taking $W$ as the mapping cylinder of $f$, we obtain $W$ satisfying desired properties with the topological bound $(0,1)$.
\end{proof}
\section{Relative pinch-off process} \label{sec:relativepinchoff}
In this section, we adapt Section $9$ of \cite{chu2024existence} to list certain topological properties of the relative pinch-off process that will be used in Section \ref{sec: topology of boundary two cap}. Consider a relative pinch-off process $\Phi: [0, T] \to \mathcal{S}(M)$, similar to the level-set flow in mean curvature flow, we define:
\[
W(t) := \{t \} \times (M \setminus \Phi(t)) \subset [0, T] \times M,
\]
and 
\[
W[t_1, t_2] := \bigcup_{t \in [t_1, t_2]} W(t) \subset [0, T] \times M.
\]
We have the following proposition.
\begin{prop} \label{prop: mcf}
If $\Phi : [0, T] \to \mathcal{S}(M)$ is a relative pinch-off process, we have:
\begin{enumerate}[label=\normalfont(\arabic*)]
    \item Any loop in $W(T)$ is homotopic in $W[0, T]$ to a loop in $W(0)$.
    \item If a loop is homologically non-trivial in $W(T)$, then it is homotopic in $W[0, T]$ to a homologically non-trivial loop in $W(0)$.
    \item By Definition \ref{def: Punctate surfaces}, for each $t \in [0, T]$, $W(t)$ can be written as the disjoint union of two open sets $\operatorname{in}(\Phi(t))$ and $\operatorname{out}(\Phi(t))$, such that both of their reduced boundaries are $\Phi(t) \setminus \Phi(t)_{iso}$, and they both vary continuously in $t$ (as Caccioppoli sets). Then
    \[
    \operatorname{rank}\left(H_1(\operatorname{in}(\Phi(t)); \mathbb{Z})\right) \quad \text{and} \quad \operatorname{rank}\left(H_1(\operatorname{out}(\Phi(t)); \mathbb{Z}) \right)
    \]
    are both non-increasing in $t$.
    \end{enumerate}
\end{prop}
\begin{proof}
The proof is the same as \cite[Proposition~9.1]{chu2024existence} with trivial modifications.
\begin{comment} 
The proof of $(1)$ is the same as that of
\cite[Proposition~9.1 (1)]{chu2024existence}.

By the proof of \cite[Proposition 9.1 (2)]{chu2024existence}, to show $(2)$, it suffices to consider how to extend the $2$-chain $\{C_t\}$ bounded by the loop $\{\gamma_t\}$ in \cite[Proposition 9.1 (2)]{chu2024existence} over a singular point $(t_0, p_0)$ where $p_0 \in \partial M$. 

If $(t_0, p_0)$ is a boundary neck-pinch point, the surgery region $U$ is a half-ball. For $t< t_0$ close to $t_0$, the solid cylinder and the disk $\Phi(t) \cap U$ in the interior neck-pinch case are replaced by a solid
half-cylinder and a half-disk $\Phi(t) \cap U$, respectively. By slightly perturbing away from $\partial M$, we can assume the 
bounding $2$-chain $C_t$ is disjoint from both $\partial M$ and $\Phi(t)$.
Thus, its intersection with the  half-disk $\Phi(t) \cap U$ consists
only of simple closed curves. These curves can be removed by the same
surgery arguments as in the interior neck-pinch case \cite[Proposition 9.1 (2) Case (a)]{chu2024existence}. The argument for $(t_0, p_0)$ being a shrinking process in a half-ball is identical to that in an interior ball, with $\partial U$ replaced by the relative boundary $\partial U \cap \operatorname{int}(M)$. Thus we can follow the proof of  \cite[Proposition 9.1 (2)]{chu2024existence} to show \normalfont{$(2)$}.

Finally, \textnormal{$(3)$} follows from
\textnormal{$(2)$} by the same linear-independence argument as in
\cite[Proposition~9.1 (3)]{chu2024existence}.
\end{comment}

\end{proof}

\section{Nonvanishing cohomology of $6$-parameter family} \label{section:topological arguments}
In this section, we proceed as \cite[Section 5.5]{chu2024existence} to show Theorem \ref{thm: nontrivial topology}. By Remark \ref{rem: Continuous orientation}, since our parameter space $Y := \mathbb{RP}^4 \times \mathbb{RP}^2$ is  the trivial $\mathbb{RP}^4$ bundle of $\mathbb{RP}^2$, our proof is simpler than \cite[Section 5.5]{chu2024existence}. We will use the $6$-parameter family (see Theorem \ref{thm:genuszerotwoboundary}).
\begin{equation} \label{eq: 6-parameter}
\Psi([(v, t)], [a]) = \Phi_6([(v, t)], [a])= \Phi_4^{[a]}[(v,t)],
\qquad ([(v,t)],[a])\in Y.
\end{equation}
and the topological quantities $\lambda, A$ (see \S \ref{subsec: Topological arguments}) below.
\begin{comment}
the following definitions:
\begin{itemize}
    \item $\lambda : = [\Phi_6]^*(\bar{\lambda})$ where $\bar{\lambda}$ is the generator of $H^*(\mathcal{Z}_2(\mathbb{B}^3, \partial \mathbb{B}^{3};\mathbb{Z}_2) ; \mathbb{Z}_2)$. 
    \item $A \coloneqq \mathbb{RP}^4 \times \mathbb{RP}^1 \subset Y$. Define $\alpha$ to be the Poincar\'e dual $PD(A)$ of $A$. 
\end{itemize}   
\end{comment}
\begin{thm}
In the cohomology ring $H^*(Y ; \mathbb{Z}_2)$,
\begin{equation*}
\lambda^4 \cup \alpha^2 \neq 0.
\end{equation*}
\end{thm}
\begin{proof}
Define the submanifold
\begin{equation*}
    E :=\{([(v, t)],[a]) \in \mathbb{RP}^4 \times \mathbb{RP}^2: t = 0 \}.
\end{equation*}                                        
Then $E \cong \mathbb{RP}^3 \times \mathbb{RP}^2$, and $E$ can be viewed as a subset of $Y$. Geometrically, $\Phi_6 |_E$ consists of images of Clifford annuli under conformal maps, and topological disks. Let $\gamma := PD(E) \in H^1(Y ; \mathbb{Z}_2)$ be the Poincar\'e dual of $E$.

Then $\gamma$ and $\alpha$ are the pullbacks of the standard generators of $H^1(\mathbb{RP}^4;\mathbb Z_2)$ and $H^1(\mathbb{RP}^2;\mathbb Z_2)$, respectively, and hence generate $H^*(Y;\mathbb Z_2)$. If $k_1,k_2\subset\mathbb{RP}^2$ are transverse representatives of the dual of $\alpha$ meeting at $p$, then
\[
\alpha^2=PD(\mathbb{RP}^4\times\{p\}).
\]
\begin{clm} 
    $\lambda = \gamma$ in $H^1(Y ; \mathbb{Z}_2).$
\end{clm}
\begin{proof}
Since $H^{*}(Y;\mathbb{Z}_{2})$ is generated by $\gamma$ and $\alpha$, we can assume
\[
\lambda = c_1 \gamma + c_2 \alpha \in H^1(Y ; \mathbb{Z}_2). 
\]
By Remark \ref{rem: Continuous orientation}, for any cycle $\sigma$ detected by $\alpha$, $\Phi_6 \circ \sigma$ is not a $1$-sweepout as it corresponds to a rotation of some Clifford annulus that preserves the orientation. Thus, $c_2 = 0$. Since $\Phi_6 |_{Y_p}$ is a $4$-sweepout, $l_p^*(\lambda^4) \neq 0 \in H^4(Y_p ; \mathbb{Z}_2)$. Thus, $\lambda $ is nonzero and hence $\lambda = \gamma$.
\end{proof}

Now we can finally give the proof of Theorem \ref{thm: nontrivial topology}. Since $H^{*}(Y;\mathbb{Z}_{2}) \cong \mathbb{Z}_{2}[\gamma,\alpha]/(\gamma^{5}, \alpha^{3})$, we have
\[
\lambda^4 \cup \alpha^2 = \gamma^{4} \cup \alpha^2 \neq 0.
\]
\end{proof}
\section{Topology of annulus cap} \label{sec: topology of boundary two cap}

We show Theorem \ref{thm:trivialInFirsthomo} in this section. As explained in Section \ref{sec: Existence of 3 free-boundary minimal annuli}, it suffices to show that for each annulus cap $C \subset \text{dmn}(\Xi)$, the map
\[
i_* : H_1(C; \mathbb{Z}_2) \to H_1(\text{dmn}(\Xi); \mathbb{Z}_2)
\]
induced by the inclusion $i : C \xhookrightarrow{} \text{dmn}(\Xi)$ is trivial. Recall Proposition \ref{prop:XiProperty}, $T(C)$ strongly deformation retracts onto $C$, and hence, it suffices to prove the claim for loops contained in the original annulus cap $C$. In particular, for any fixed loop $c \subset C$, we aim to show that
\begin{equation} \label{eq: i(c) is trivial in first homology}
[i(c)] = 0 \in H_1(\text{dmn}(\Xi); \mathbb{Z}_2).    
\end{equation}
Note that, by the definition of an annulus cap, the image
\begin{equation}
    [\Xi](c) \subset \mathbf{B}_{d_0}^{\mathbf{F}}([A])
\end{equation}
for some properly embedded free boundary minimal annulus $A \subset M$.

The proof of (\ref{eq: i(c) is trivial in first homology}) will consist of three steps. First, we need to understand the topology of the elements of $\Xi|_c$. In the second step, using the family $\tilde{\Xi}$ obtained in Proposition \ref{prop:XiPsiHomotopic}, we, in a certain sense, homotope the family $\Xi|_c$ back to some subfamily of the space of unoriented Clifford annuli. In the third step, we prove that $c_0$ is homologically trivial by studying the family $\Phi$. 

The first two steps follow in \cite[Section 10]{chu2024existence} with easy modifications. We have to make some modifications to the third step in \cite[Section 10]{chu2024existence}. Now we give the details below.

\subsection{The family $\Xi|_c$} \begin{comment}
In Proposition \ref{prop:XiPsiHomotopic}, we obtained a mapping cone $\tilde{W}$ containing $Y$ and dmn$(\Xi)$, a map $\tilde{F}: [0, 1] \times \tilde{W} \to \tilde{W}$, and a Simon-Smith family $\tilde{\Xi}: \tilde{W} \to \mathcal{S}^*(M)$ that each surface is of genus $0$ with at most $2$ boundary components. Then we can define a map
\[
G_1 : [0, 1] \times \mathbb{S}^1 \to\tilde{W}
\]
such that $G_1 (1, \cdot )$ parametrizes the loop $c \subset \tilde{W}$ given above. We then define, for every $t \in [0, 1]$ and $\theta \in \mathbb{S}^1$,
\[
G_1(t, \theta) := \tilde{F}(t, G_1(1, \theta)).
\]
By Proposition \ref{prop:XiPsiHomotopic}, for each $\theta \in \mathbb{S}^1$, the family $\{\tilde{\Xi}\circ G_1(t, \theta) \}_{t \in [0, 1]}$ is a relative pinch-off process.  
\end{comment}
To prove Theorem \ref{thm:trivialInFirsthomo}, we need more information about $\Xi(x)$ for $x \in C$ beyond the fact that $C$ is an annulus cap. More precisely, we need that the inside region enclosed by the surface resemble a solid ball, and the outside region resemble a solid torus.

\begin{prop} \label{prop: rigidity of surfaces close to an annulus}
In a Riemannian $3$-ball $(M, g)$, let $\Sigma$ be a smooth properly embedded unknotted annulus. There exists a constant $\tau (\Sigma) > 0$ with the following property.

Let $S \in \mathcal{S}(M)$. By Definition \ref{def: Punctate surfaces}, let $\Omega, \Omega' \subset M$ be the two open subsets of $M \setminus S$, whose reduced boundaries are both $S \setminus S_{iso}$. Assume that:
\begin{enumerate}[label=\normalfont(\roman*)]
    \item $H_1(\Omega; \mathbb{Z})$ and $H_1(\Omega'; \mathbb{Z})$ are both either $0$ or $\mathbb{Z}$, and cannot be both $\mathbb{Z}$ at the same time.
    \item $\mathbf{F}([S], [\Sigma]) < \tau(\Sigma)$.
\end{enumerate}
Then,
\begin{enumerate} [label=\normalfont(\arabic*)]
    \item Fix an inside direction for $\Sigma$, such that  the exterior region bounded by $\Sigma$, $\operatorname{out}(\Sigma)$ is a solid torus, and the interior region $\operatorname{in}(\Sigma)$ is a solid ball. Then there is a unique choice of an inside direction for $S$ such that
    \[
    \operatorname{vol}(\operatorname{in}(S) \Delta \operatorname{in}(\Sigma)), \operatorname{vol}(\operatorname{out}(S) \Delta \operatorname{out}(\Sigma)) < \tau(\Sigma).
    \]
    where $\{ \operatorname{in}(S), \operatorname{out}(S) \} = \{ \Omega, \Omega' \}$. After this choice, we relabel the two components so that $\Omega=out(S)$, $\Omega'=in(S)$.
    \item $H_1(\Omega, \mathbb{Z}) \cong \mathbb{Z}$ and $H_1(\Omega', \mathbb{Z}) \cong 0$.
    \item Fixing a generator $a_0$ for $H_1(\operatorname{out}(\Sigma); \mathbb{Z})$, there is a unique generator $a_1$ for $H_1(\operatorname{out}(S); \mathbb{Z})$ such that there exists a loop $\gamma \subset \operatorname{out}(S) \cap \operatorname{out}(\Sigma)$ that induces both $a_0$ and $a_1$.
\end{enumerate}
\end{prop}

\begin{proof}
It follows by applying the same arguments in \cite[Proposition 10.1]{chu2024existence} to $H_1(\operatorname{out}(\Sigma), \mathbb{Z})$ and the relative homology group $H_1(\overline{\operatorname{in}(\Sigma)},  \partial \operatorname{in}(\Sigma) \cap \partial M; \mathbb{Z})$.
\end{proof}

Since $(M, g)$ has only finitely many properly embedded free boundary minimal annuli, we can assume that the constant $d_0 > 0$ in \S \ref{sec:boundary component cap} is smaller than $\tau(\Sigma)$ given in Proposition \ref{prop: rigidity of surfaces close to an annulus} for every free boundary minimal annulus $\Sigma$ in $M$. Thus, $[\Xi(x)]$ satisfies all the  properties in Proposition \ref{prop: rigidity of surfaces close to an annulus}.

\subsection{Homotope $\Xi|_c$ to $\Psi_{c_0}$} We first show that all smooth properly embedded annuli in $\Psi$ can be deformation retracted to Clifford annuli.

\begin{lem} \label{lem: Z_0 deformation retract to Z_1}
Consider $\Psi$ as a Simon-Smith family in the unit ball $M$. Let $Z_0, Z_1 \subset Y$ denote the sets of parameters corresponding to Clifford annuli, and properly embedded annuli. Then $Z_1$ can be deformation retracted to $Z_0$.
\end{lem}

\begin{proof}
Since our construction of $\Psi$ (see Theorem \ref{thm:genuszerotwoboundary}) is similar to the $9$-parameter family in \cite[Section 5]{chu2024existence}, it follow from the same proof as \cite[Lemma 10.3]{chu2024existence}.
\end{proof}

Now we apply the same arguments as \cite[\S 10.2]{chu2024existence} to obtain a homotopy

\[
G : [0, 1] \times \mathbb{S}^1 \to \tilde{W}
\]
such that:
\begin{enumerate}[label=\normalfont(\roman*)]
    \item $G(0, \cdot)$ parametrizes $c_0 \subset Z_0 \subset \tilde{W}$.
    \item $G(1, \cdot)$ parametrizes $c \subset \text{dmn}(\Xi) \subset \tilde{W}$.
    \item For each $\theta \in \mathbb{S}^1$, the family $\tilde{\Xi} \circ G_3(\cdot, \theta)$ is a relative pinch-off process.
    \item For each $(t, \theta) \in [0, 1] \times \mathbb{S}^1$,
    \[
    H_1(\operatorname{in}(\tilde{\Xi} \circ G(t, \theta); \mathbb{Z}) \cong 0, \quad H_1(\operatorname{out}(\tilde{\Xi} \circ G(t, \theta); \mathbb{Z}) \cong \mathbb{Z}.
    \]
\end{enumerate}
In particular, by \normalfont{(i)}, \normalfont{(ii)} above and the fact that $\tilde{F}$ induces a strong deformation retraction of $\tilde{W}$ onto $Y$ (see Proposition \ref{prop:XiPsiHomotopic}), to prove that $[c] = 0$ in $H_1(\operatorname{dmn}(\Xi); \mathbb{Z}_2)$, it suffices to show that $[c_0] = 0$ in $H_1(Z_0, \mathbb{Z}_2)$.

Recall that $Z_0 \cong \mathbb{RP}^2$. Thus, $H_1(Z_0; \mathbb{Z}_2) = \mathbb{Z}_2$ has $1$ generator $\lambda \in \mathbb{Z}_2$. Thus, it suffices to rule out the possibility that $[c_0] = \lambda$. By \normalfont{(iv)} above, we can continuously choose an inside direction on the family $\tilde{\Xi} \circ G$. As a result, we can naturally identify the groups
\[
H_1(\operatorname{out}(\tilde{\Xi} \circ G(t, \theta)); \mathbb{Z}), \quad (t, \theta) \in [0, 1] \times \mathbb{S}^1
\]
in the following sense. Let $k \in \mathbb{N}$, for each $i, j = 1, \cdots, k$, define the closed rectangle
\[
R_{i, j} = \left[\frac{i - 1}{k}, \frac{i}{k}\right] \times \left[2 \pi \frac{j - 1}{k}, 2 \pi \frac{j}{k}\right] \subset [0, 1] \times \mathbb{S}^1.
\]
We have the following proposition from \cite[Proposition 10.4]{chu2024existence}.

\begin{prop}[\ChuProp] 
\label{prop: Rectanles}
For sufficiently large $k \in \mathbb{N}$, we have that for each $(i, j)$, there exists a loop $\gamma_{i, j} $ such that the following hold: If $R_{m, n}$ is one of the rectangles that intersect $R_{i, j}$, then for every $p \in R_{m, n}$, we have:
\begin{enumerate}[label=\normalfont(\roman*)]
    \item $\gamma_{i, j} \subset \operatorname{out}(\tilde{\Xi} \circ G(p))$ and generates $H_1(\operatorname{out}(\tilde{\Xi}\circ G(p)); \mathbb{Z})$
    \item  $\gamma_{i, j} $ and $\gamma_{m, n} $ are homologous in $\operatorname{out}(\tilde{\Xi} \circ G(p))$
\end{enumerate}
\end{prop}

\subsection{The family $\Phi|_{c_0}$.} We now focus on the loop $\{0 \} \times \mathbb{S}^1 \subset [0, 1] \times \mathbb{S}^1$. By Proposition \ref{prop: Rectanles} above, we can piecewise define a map $L_0 : \mathbb{S}^1 \times \mathbb{S}^1 \to M$ by
\[
L_0(\theta, \cdot) = \gamma_{1, j} \subset \operatorname{out}(\tilde{\Xi} \circ G(0, \theta)) \quad \text{for }\theta \in \left[2 \pi \frac{j - 1}{k}, 2 \pi \frac{j}{k}\right].
\]
for $j = 1, \cdots, k$. $L_0$ may not be continuous under this definition, but Proposition \ref{prop: Rectanles} \normalfont{(ii)} gives 
\[
[\gamma_{1, j} ] = [\gamma_{1, j + 1}] \in H_1\left(\operatorname{out}\left(\tilde{\Xi} \circ G\left(0, 2\pi\frac{j}{k}\right)\right); \mathbb{Z}\right).
\]
Thus, we can make $L_0$ continuous by using a homotopy in $\operatorname{out}(\tilde{\Xi} \circ G_3(0, 2\pi\frac{j}{k})$  to continuously deform $\gamma_{1, j}$ to $\gamma_{1, j+1}$ in $(2\pi j/k - \epsilon, 2\pi j/k + \epsilon) \times \mathbb{S}^1$ for each $j$. Hence, we may assume $L_0$ is continuous. Since $\tilde{\Xi} \circ G(0, 0) = \tilde{\Xi} \circ G(0, 2 \pi)$, 
\begin{equation} \label{eq: Orientation is preserved}
    [L_0(0, \cdot)] = [L_0(2 \pi, \cdot)] \in H_1(\operatorname{out}(\tilde{\Xi} \circ G(0, \theta)); \mathbb{Z})
\end{equation}
However, by the topological lemma below, (\ref{eq: Orientation is preserved}) cannot hold if $[\tilde{\Xi} \circ G(0, \cdot)] = [c_0] \neq0$ in $H_1(Z_0 ; \mathbb{Z}_{2})$.

\begin{lem} \label{lem: topological lemma}
Let $d: [0, 2 \pi] \to Z_0$ with $d(0) = d(2 \pi)$ parametrize a loop representing $\lambda \in H_1(Z_0 ; \mathbb{Z}_2)$. There is a unique continuous choice of inside direction for $\Psi \circ d(\theta)$ such that $\operatorname{out}(\Psi \circ d(\theta))$ is a solid torus, for each $\theta \in [0, 2 \pi]$. Let
\[
L: [0, 2 \pi] \times \mathbb{S}^1 \to M
\]
be a continuous map such that each loop $L(\theta, \cdot)$ is in $\operatorname{out}(\Phi \circ d(\theta))$ and generates
\[
H_1(\operatorname{out}(\Psi \circ d (\theta)); \mathbb{Z}) \cong \mathbb{Z}.
\]
Then $[L(0, \cdot)]$ and $[L(2 \pi, \cdot)]$, which both are elements in $H_1(\operatorname{out}(\Psi \circ d (0)); \mathbb{Z})$ differ exactly by a sign.
\end{lem}

\begin{proof}
    Let us view $\Psi$ as a Simon-Smith family in the unit ball $M$, so that each $\Psi \circ d(\theta)$ is a Clifford annulus in $\mathcal{A}$. Since $Z_0 \cong \mathbb{RP}^2$ and $\pi_1(Z_0) = \mathbb{Z}_2$, we also have $[d] = \lambda \in \pi_1 (Z_0)$.

    Consider the set
    \begin{equation} \label{eq: the space of solid annulus with orientations}
        \begin{split}
        \bigl\{ (\Omega, a_0): \, & \Omega \subset M \text{ is a solid torus bounded by some annulus in } \mathcal{A},\\
        & a_0 \text{ is one of the two generators of }H_1(\Omega; \mathbb{Z})\bigr\}
    \end{split}
    \end{equation}

    Evidently, there is a natural topology on this set such that it is a double cover of the space
    \begin{equation} \label{eq: the space of soild torus without orientations}
    \bigl\{ \Omega: \,  \Omega \subset M \text{ is a solid torus bounded by some annulus in } \mathcal{A}\bigr\} \cong Z_0
    \end{equation}
    Note that $L$ induces a map $L': \theta \mapsto [L(\theta, \cdot)]$ from $[0, 2 \pi]$ into the space (\ref{eq: the space of solid annulus with orientations}). To prove the lemma, it suffices to show that, while the solid tori corresponding to $L'(0)$ and $L'(2 \pi)$ are the same, the first homology classes of the solid torus given by $L'(0)$ and $L'(2 \pi)$ differ by a sign.

    Note that $d$ can be viewed as a path $d' : \theta \mapsto \operatorname{out}(\Psi \circ d(\theta))$ from $[0, 2 \pi]$ into the space (\ref{eq: the space of soild torus without orientations}) and $L'$ is a lift of $d'$ into the space (\ref{eq: the space of solid annulus with orientations}). Now we identify the space (\ref{eq: the space of soild torus without orientations}) as $\mathbb{RP}^2$, and its double cover (\ref{eq: the space of solid annulus with orientations}) as $\mathbb{S}^2$.
    Fix $x \in \mathbb{S}^2$ such that $[x]$ in $\mathbb{S}^2 / \sim \,\cong \mathbb{RP}^2$ with $x \sim -x$ corresponds to $d'(0)$. Since $\mathbb{S}^2$ is simply connected and $[d'] = [d] \neq 0 \in \pi_1 (\mathbb{RP}^2)$, $L'$ cannot form a loop. Hence, $L'(0)$ and $L'(2 \pi)$ must differ by a sign.
\end{proof}

Plugging in $L_0$ and $c_0$ into Lemma \ref{lem: topological lemma},  we have (\ref{eq: Orientation is preserved}) cannot hold if $[\tilde{\Xi} \circ G(0, \cdot)] = [c_0] \neq0$ in $H_1(Z_0 ; \mathbb{Z}_{2})$. Thus, $[c]$ is trivial in $H_1(\operatorname{dmn}(\Xi); \mathbb{Z}_{2})$ and we are done with the proof of Theorem \ref{thm:trivialInFirsthomo}.

\appendix

\section{Conformal images} \label{appendix}
In this appendix we collect some facts about conformal transformations of $\mathbb{R}^3$. Recall that $F_v : \mathbb{B}^3 \to \mathbb{B}^3$ is defined by
\[
F_v(x) = \frac{(1 - |v|^2)x - (1 - 2 \langle v, x \rangle + |x|^2) v}{1 - 2 \langle v, x \rangle + |x|^2 |v|^2}.
\]
Given $p, N \in \mathbb{S}^2$ with $\langle p, N \rangle = 0$ and $0 < r < \frac{\pi}{4}$, we define
\[
\begin{split}
\Delta(p, N, r) & = \mathbb{B}^3 \setminus (B^3_{\tan{r}}(\sec{r} \, (\cos{r} \, p + \sin{r} \, N)) \cup B^3_{\tan{r}}(\sec{r} \, (\cos{r} \, p - \sin{r} \, N)) \\
& = \mathbb{B}^3 \setminus \left( \pi^{-1}\left(B^+_r(\cos{r} \, p + \sin{r} \, N, 0) \right)  \cup \pi^{-1}\left(B^+_r(\cos{r} \, p - \sin{r} \, N, 0) \right)\right).
\end{split}
\]
 This appendix follows \cite[Appendix B]{marques2014}, since our conformal deformations are similar to those in \cite{marques2014}. In our setting, the blow-up limits are orthogonal spherical caps instead of geodesic spheres. So we have to make some necessary modifications  to \cite[Appendix B]{marques2014}.
\begin{prop} \label{prop:A.1}

There is $C_0 > 0$ and, for each $r \in (0, \frac{\pi}{4})$, $C_1 = C_1(r) > 0$ and $\epsilon_0 = \epsilon_0(r) > 0$ such that the following holds: for every
\[
v = (1 - s)(\cos{t}\, p + \sin{t} \, N),
\]
with
\[
\begin{split}
& p, N \in \mathbb{S}^2, \quad \langle p, N \rangle = 0,\quad 0 < s \leq \frac{1}{2}, \quad \text{and} \quad |t| \leq \epsilon_{0},  \\
& H_N : = \{ \langle x, N \rangle < 0 : x \in \mathbb{R}^3 \} \cap \mathbb{B}^3
\end{split}
\]
we have
\[
B^4_{\overline{R} - C_0 \sqrt{| (s, t)|}} (\overline{Q}, 0) \cap \mathbb{S}_+^3 \subset \pi \circ F_v(H_N) \subset B^4_{\overline{R} + C_0 \sqrt{| (s, t)|}} (\overline{Q}, 0) \cap \mathbb{S}_+^3,
\]
and
\[
\pi \circ F_v(\Delta(p, N, r)) \subset \big(B^4_{\overline{R} +  C_1 \sqrt{|(s, t)|}}(\overline{Q}, 0) \setminus B^4_{\overline{R} - C_1 \sqrt{|(s, t)|}} (\overline{Q}, 0) \big) \cap \mathbb{S}^3_+.
\]
where
\[
\begin{split}
\overline{Q} & = - \frac{\frac{t}{s}}{\sqrt{1 + (\frac{t}{s})^2}}p - \frac{1}{\sqrt{1 + (\frac{t}{s})^2}} N, \\
\overline{R} & = \sqrt{2\left(1 - \frac{\frac{t}{s}}{\sqrt{1 + (\frac{t}{s})^2}}\right)}.
\end{split}
\]
\end{prop}

\begin{proof}
We decompose the proof of Proposition \ref{prop:A.1} into the following two lemmas.

\begin{lem} \label{lem: A.2}
There is $C_0 > 0$ so that for every
\[
v = (1 - s)(\cos{t} \,  p + \sin{t} \,  N),
\]
with
\[
p, N \in \mathbb{S}^2, \quad \langle p, N \rangle = 0, \quad 0 < s \leq \frac{1}{2}, \quad \text{and } |t| \leq \frac{1}{2}, 
\]
we have
\[
B^4_{\overline{R} - C_0 \sqrt{| (s, t)|}} (\overline{Q}, 0) \cap \mathbb{S}_+^3 \subset \pi \circ F_v(H_N) \subset B^4_{\overline{R} + C_0 \sqrt{| (s, t)|}} (\overline{Q}, 0) \cap \mathbb{S}_+^3.
\]

\end{lem}
\begin{proof}
Observe that $(F_v)^{-1} = F_{-v}$. For any $y \in F_v(H_N)$, we have
\begin{equation} \label{eq:H_v}
\langle F_{-v}(y), N \rangle < 0.
\end{equation}
By direct computations, (\ref{eq:H_v}) can be expanded as
\[
\left\langle \frac{2y}{|y|^2 + 1}, -2 \langle v, N \rangle v - (1 - |v|^2)N \right\rangle > 2 \langle v, N \rangle,
\]
which implies
\[
\left\langle \pi(y), \left(-\frac{Q}{|Q|}, 0\right) \right\rangle > \frac{2 \langle v, N \rangle}{|Q|},
\]
where 
\[
Q = 2 \langle v, N \rangle v + (1 - |v|^2)N = (2s - s^2 +2(1 - s)^2 \sin^2{t}) N + 2(1 - s)^2\sin{t} \cos{t} \, p
\]
and 
\[
|Q|^2 = (2s - s^2)^2 + 4(1 - s)^2 \sin^2{t}.
\]
Thus, we have
\begin{equation}
    \pi(F_v(H_N)) = B^4_{R} \left(   - \frac{Q}{|Q|}, 0\right) \cap \mathbb{S}_{+}^3,
\end{equation}
where $R = \sqrt{2(1 - \frac{2 \langle v, N \rangle}{|Q|})}$. Now we can directly follow \cite[Lemma B.6]{marques2014} to conclude the proof (Replace $\overline{Q}$ in their proof by $(\overline{Q}, 0)$ given above and all the computations remain the same).
\end{proof}

\begin{lem} \label{lem: A.3}
For every $r \in (0, \pi/4)$ and $0<s \le 1/2$, there is $C_1 = C_1(r) > 0 $ and $\epsilon_0 = \epsilon_0(r) > 0 $ so that for every
\[
v = (1 - s) (\cos{t} \, p + \sin{t} \, N),
\]
with 
\[
p, N \in \mathbb{S}^2, \langle p, N \rangle = 0, \quad \text{and} \quad |t| \leq \epsilon_0,
\]
we have 
\[
\pi \circ F_v(\Delta(p, N, r)) \subset \big(B^4_{\overline{R} +  C_1 \sqrt{|(s, t)|}}(\overline{Q}, 0) \setminus B^4_{\overline{R} - C_1 \sqrt{|(s, t)|}} (\overline{Q}, 0) \big) \cap \mathbb{S}^3_+.
\]
\end{lem}

\begin{proof}
let $\sigma^i = (-1)^{i + 1}$, $i = 1, 2$. Define
\[
B_i = B^3_{\tan{r}}\left(\sec{r}(\cos{r} \, p + \sigma^i \sin{r} \, N)\right) \cap \mathbb{B}^3
\]
and $h_i = \cos{r} (\cos{r}p + \sigma^i \sin{r} N)$. Then, for any $y \in F_v(B_i)$, $F_{-v}(y) \in B_i$. Thus, we have
\[
|F_{-v}(y) - \sec{r}(\cos{r} \, p + \sigma^i \sin{r} \, N)|^2 < \tan^2{r}.
\]
We then expand it to obtain
\begin{equation} \label{eq:B_i}
1 + |F_{-v}(y)|^2 - 2 \left\langle\frac{h_i}{\cos^2{r}}, F_{-v}(y) \right\rangle < 0.
\end{equation}
Using the fact that
\[
|F_{-v}(y)|^2 = \frac{|y + v|^2}{1 + 2 \langle v, y \rangle + |v|^2|y|^2}
\]
and $|h_i| = \cos{r}$, by direct computations, (\ref{eq:B_i}) can be simplified to
\[
\left\langle \frac{2y}{|y|^2 + 1}, Q_i \right\rangle > |h_i|^2 (1 + |v|^2) - 2 \langle h_i, v \rangle
\]
where
\[
Q_i = (1 - |v|^2) h_i - 2(|h_i|^2 - \langle h_i, v \rangle)v.
\]
Thus, we have
\[
\pi(F_v(B_i)) = B^4_{R_i}\left(\frac{(Q_i, 0)}{|(Q_i, 0)|}\right) \cap \mathbb{S}^3_+,
\]
where
\[
R_i = \sqrt{2\left(1 - \frac{|h_i|^2 (1 + |v|^2) - 2 \langle h_i, v \rangle}{|Q_i|}\right)}.
\]
Now we can just follow \cite[Lemma B.7]{marques2014} to conclude the proof (Replace $Q_i$ in their proof by $(Q_i, 0)$ given above and all the computations are the same). 
\end{proof}
Proposition \ref{prop:A.1} then follows from Lemma \ref{lem: A.2} and \ref{lem: A.3}.

\end{proof}

\section{Lifting lemma for free boundary minimal annuli with multiplicity}
In this section, we show that for any free boundary Simon-Smith family with the topological bound $(0, 1)$, the min-max limit cannot be annulus with high multiplicity by proving a topological bound with multiplicity with an improved lifting lemma. As punctate surfaces can be modified to be smooth properly embedded in finitely many arbitrarily small balls, we only need to consider smooth properly embedded surfaces. By our proof of Theorem \ref{thm:simon-smith min-max} (3) (also see \cite[\S 4.5]{chu2024existence}), it suffices to show the following theorem.

\begin{thm}\label{thm:special}
Given a Riemannian manifold $(M^{3},g)$ with a strictly mean convex boundary, and $A\subset M$ be a two-sided, properly embedded,
free-boundary minimal annulus. Given $K \in \mathbb{N}^+, R \in [0, +\infty)$. Suppose $\{S_j\}$ are connected, smooth,
properly embedded disks or annuli such that
\[
                         |S_j|\to m|A|
\]
in the varifold sense for some $m\in\mathbb N^+$, and $mA$ is \textit{almost minimizing in every $K$-admissible collection of annuli of outer radius at most $R$} with respect to $\{S_j\}$. Then $m = 1$.
\end{thm}
We denote the mod $2$ intersection number by $\#_{2}$.
\begin{lem}[Lifting lemma with multiplicity for free boundary minimal annuli]
\label{lem:main}
Under the same assumptions as Theorem \ref{thm:special}, after passing to a subsequence, there are surfaces $S'_j$ obtained
from $S_j$ by finitely many neck and half-neck surgeries, shrinkings, and ambient isotopies, together with curves
\[
 c_j^1,\ldots,c_j^m\subset S'_j,
 \qquad a_j^1,\ldots,a_j^m\subset S'_j,
\]
such that
\begin{enumerate}[label=\normalfont(\roman*)]
 \item $|S'_j|\to m|A|$ and $\beta_{1}(S'_j)\leq\beta_{1}(S_j)$;
 \item $S_j'$ is an annulus for all sufficiently large $j$;
 \item the $c_j^r$ are pairwise disjoint degree-one lifts of one homologically nontrivial loop in $A$;
 \item every $a_j^s$ is a proper embedded arc joining the two ends of $S_j'$ and  its  numbers of intersection $(\text{mod }2)$ with $c_j^r$ are
 \begin{equation}\label{eq:duality}
             \#_2(a_j^s\cap c_j^r)=\delta_{sr}
             \qquad(1\leq r,s\leq m).
 \end{equation}
\end{enumerate}
The arcs $a_j^s$ need not be mutually disjoint.
\end{lem}

\begin{proof}
Choose a homologically nontrivial loop $c\Subset\operatorname{Int}A$ and a proper arc
$a\subset A$ meeting $c$ transversely once at some point $q$. Then (i) follows from \cite[Proposition 3.7]{franz2023topological}, and (ii) from \cite[Theorem 4.11]{franz2023topological}.

(iii) follows from the Improved lifting lemma with multiplicity \cite[Proposition 2.2]{ketover2019genus} as $c$ is in the interior of A. Moreover, by \cite[\S 4.2 - 4.8]{ketover2019genus}, the same almost minimizing property holds for $mA$ with respect to $\{S_j'\}$, and near $q$, $S_j'$ consists of $m$ ordered graphs over $A$. 

Now we show (iv). Take $q_-, q_+ \in a$ on opposite sides of $q$. Let $a[q_-, q_+] \subset a$ be the arc between these two points. By taking $q_-, q_+$ sufficiently close to $q$, there is a unique arc $b_{j,i}$ on the $i$-th graph such that the nearest point projection
\[
\pi_A: b_{j,i} \to a[q_-, q_+]
\]
is a diffeomorphism for each $1 \leq i \leq m$. Let $\partial_{\pm}A$ be the two boundary components of $A$ and $A^\pm$ be the two components of $A\setminus c$, with $\partial_{\pm}A \subset \overline{A^{\pm}}$ respectively. For each $\sigma \in \{ +, - \}$, we take sufficiently small balls $R^{\sigma}$ centered at $q_{\sigma}$ such that $S_j' \cap R^{\sigma}$ consists of $m$ graphical disks $\{D_{j, s}^{\sigma}\}_{s = 1}^m$. Choose an embedded arc
$\tau^\sigma\subset A^\sigma$ from $q^\sigma$ to
$\partial_\sigma A$, and  disjoint relative open
neighborhoods $\mathcal O^\pm\subset M$ such that
\[
 R^\sigma\cup\tau^\sigma\subset\mathcal O^\sigma,
 \qquad
 \overline{\mathcal O^\sigma}\cap\pi_A^{-1}(c)=\varnothing.
\]
Let $p_{j,s}^\sigma \in \pi_A^{-1}(q_\sigma)$ be the endpoint of $b_{j,s}$ over $q_\sigma$. The prescribed disk $D_{j,s}^\sigma$ is then unchanged and has limiting multiplicity one.  Thus we may choose
$D_{j,s}^\sigma$ as the initial component in Claim~0 of \cite[Lemma~4.8]{franz2023topological} and apply Claims~1--2 there to obtain an embedded arc \(\gamma_{j,s}^\sigma\subset S'_j\cap\mathcal O^\sigma\)
from $p_{j,s}^\sigma$ to the free boundary end 
$\partial_\sigma A$. Hence,
\(\gamma_{j,s}^\sigma\cap\pi_A^{-1}(c)=\varnothing.\)
After a small perturbation inside $\mathcal O^\sigma$, its interior
also misses $b_{j,s}$.  Therefore
\[
 a_j^s=\gamma_{j,s}^-\cup b_{j,s}\cup\gamma_{j,s}^+
\]
is a proper embedded arc, and, since every $c_j^r$ is contained in
$\pi_A^{-1}(c)$,
\[
 \#_2(a_j^s\cap c_j^r)
 =\#_2(b_{j,s}\cap c_j^r)
 =\delta_{sr}.
\]
\end{proof}

\begin{proof}[Proof of Theorem \ref{thm:special}]
Consider
\( [c_j^r]\in H_1(S'_j;\mathbb{Z}_2),
 [a_j^s]\in H_1(S'_j,\partial S'_j;\mathbb{Z}_2)\)
obtained in Lemma \ref{lem:main}.
The mod-$2$ Poincare--Lefschetz pairing and \eqref{eq:duality} make the
$m$ loop classes linearly independent: intersecting
$\sum_r\lambda_r[c_j^r]=0$ with $[a_j^s]$ gives $\lambda_s=0$. Thus, we have
\(m\leq\beta_1(S'_j)\leq\beta_1(S_j)\leq1.\) So $m=1$.
\end{proof}

\bibliographystyle{amsplain}
\bibliography{citation}

@article {marques2014,
    AUTHOR = {Marques, Fernando C. and Neves, Andr\'e},
     TITLE = {Min-max theory and the {W}illmore conjecture},
   JOURNAL = {Ann. of Math. (2)},
  FJOURNAL = {Annals of Mathematics. Second Series},
    VOLUME = {179},
      YEAR = {2014},
    NUMBER = {2},
     PAGES = {683--782},
      ISSN = {0003-486X,1939-8980},
   MRCLASS = {53C42 (49Q20)},
  MRNUMBER = {3152944},
MRREVIEWER = {Andrea\ Mondino},
       DOI = {10.4007/annals.2014.179.2.6},
       URL = {https://doi-org.proxy.libraries.rutgers.edu/10.4007/annals.2014.179.2.6},         
}

@article{chu2024existence,
  title={Existence of 5 minimal tori in 3-spheres of positive Ricci curvature},
  author={Chu, Adrian Chun-Pong and Li, Yangyang},
  journal={Cambridge Journal of Mathematics, To appear},
  year={2026}
}

@book{pitts2014existence,
  title={Existence and regularity of minimal surfaces on Riemannian manifolds},
  author={Pitts, Jon T},
  volume={27},
  year={2014},
  publisher={Princeton University Press}
}

@incollection{marques2020applications,
  title={Applications of min--max methods to geometry},
  author={Marques, Fernando C and Neves, Andr{\'e}},
  booktitle={Geometric Analysis: Cetraro, Italy 2018},
  pages={41--77},
  year={2020},
  publisher={Springer}
}

@article{franz2023topological,
  title={Topological control for min-max free boundary minimal surfaces},
  author={Franz, Giada and Schulz, Mario B},
  journal={Journal of European Mathematical Society, To appear},
  year={2026}
}

@article{wang2023existence,
  title={Existence of four minimal spheres in $\mathbb{S}^{3}$ with a bumpy metric},
  author={Wang, Zhichao and Zhou, Xin},
  journal={arXiv preprint arXiv:2305.08755},
  year={2023}
}

@article{ketover2019genus,
  title={Genus bounds for min-max minimal surfaces},
  author={Ketover, Daniel},
  journal={Journal of Differential Geometry},
  volume={112},
  number={3},
  pages={555--590},
  year={2019},
  publisher={Lehigh University}
}

@article{zhou2020multiplicity,
  title={On the multiplicity one conjecture in min-max theory},
  author={Zhou, Xin},
  journal={Annals of Mathematics},
  volume={192},
  number={3},
  pages={767--820},
  year={2020},
  publisher={JSTOR}
}

@article{guang2021min,
  title={Min--max theory for free boundary minimal hypersurfaces II: general Morse index bounds and applications},
  author={Guang, Qiang and Li, Martin Man-chun and Wang, Zhichao and Zhou, Xin},
  journal={Mathematische Annalen},
  volume={379},
  number={3},
  pages={1395--1424},
  year={2021},
  publisher={Springer}
}

@book{buchstaber2002torus,
  title={Torus actions and their applications in topology and combinatorics},
  author={Buchstaber, Victor M and Panov, Taras E},
  number={24},
  year={2002},
  publisher={American Mathematical Soc.}
}

@article{ketover2016free,
  title={Free boundary minimal surfaces of unbounded genus},
  author={Ketover, Daniel},
  journal={arXiv preprint arXiv:1612.08691},
  year={2016}
}

@phdthesis{nurser2016low,
  title={Low min-max widths of the round three-sphere},
  author={Nurser, Charles Arthur George},
  year={2016},
  school={Imperial College London}
}

@article{ros1999willmore,
  title={The Willmore conjecture in the real projective space},
  author={Ros, Antonio},
  journal={Mathematical Research Letters},
  volume={6},
  number={5},
  pages={487--493},
  year={1999},
  publisher={International Press of Boston, Inc. Somerville, MA 02143, USA}
}

@article{franz2023equivariant,
  title={Equivariant index bound for min--max free boundary minimal surfaces},
  author={Franz, Giada},
  journal={Calculus of Variations and Partial Differential Equations},
  volume={62},
  number={7},
  pages={201},
  year={2023},
  publisher={Springer}
}

@article{marques2016morse,
  title={Morse index and multiplicity of min-max minimal hypersurfaces},
  author={Marques, Fernando C and Neves, Andr{\'e}},
  journal={Cambridge journal of mathematics},
  volume={4},
  number={4},
  pages={463--511},
  year={2016},
  publisher={International Press of Boston, Inc. Somerville, MA 02143, USA}
}

@article{tran2016index,
  title={Index characterization for free boundary minimal surfaces},
  author={Tran, Hung},
  journal={Comm. Anal. Geom. 28(1): 189-222},
  year={2020}
}

@article{devyver2019index,
  title={Index of the critical catenoid},
  author={Devyver, Baptiste},
  journal={Geometriae Dedicata},
  volume={199},
  number={1},
  pages={355--371},
  year={2019},
  publisher={Springer}
}

@article{nitsche1985stationary,
  title={Stationary partitioning of convex bodies},
  author={Nitsche, Johannes CC},
  journal={Archive for rational mechanics and analysis},
  volume={89},
  number={1},
  pages={1--19},
  year={1985},
  publisher={Springer}
}

@article{wang2024existence,
  title={Existence of infinitely many free boundary minimal hypersurfaces},
  author={Wang, Zhichao},
  journal={Journal of Differential Geometry},
  volume={126},
  number={1},
  pages={363--399},
  year={2024},
  publisher={Lehigh University}
}

@article{colding2003min,
  title={The min--max construction of minimal surfaces},
  author={Colding, Tobias H and De Lellis, Camillo},
  journal={Surveys in differential geometry},
  volume={VIII},
  year={2003}
}

@inproceedings{de2018min,
  title={Min-max theory for minimal hypersurfaces with boundary},
  author={De Lellis, Camillo and Ramic, Jusuf},
  booktitle={Annales de l'Institut Fourier},
  volume={68},
  number={5},
  pages={1909--1986},
  year={2018}
}

@article{sun2024multiplicity,
  title={Multiplicity one for min--max theory in compact manifolds with boundary and its applications},
  author={Sun, Ao and Wang, Zhichao and Zhou, Xin},
  year={2024},
  publisher={Springer Berlin Heidelberg}
}

@article{marques2021morse,
  title={Morse index of multiplicity one min-max minimal hypersurfaces},
  author={Marques, Fernando C and Neves, Andr{\'e}},
  journal={Advances in Mathematics},
  volume={378},
  pages={107527},
  year={2021},
  publisher={Elsevier}
}

@book {Simon1983Lectures,
    AUTHOR = {Simon, Leon},
     TITLE = {Lectures on geometric measure theory},
    SERIES = {Proceedings of the Centre for Mathematical Analysis,
              Australian National University},
    VOLUME = {3},
 PUBLISHER = {Australian National University, Centre for Mathematical
              Analysis, Canberra},
      YEAR = {1983},
     PAGES = {vii+272},
      ISBN = {0-86784-429-9},
   MRCLASS = {49-01 (28A75 49F20)},
  MRNUMBER = {756417},
MRREVIEWER = {J.\ S.\ Joel},
}

@article{sarnataro2026existence,
  title={Existence of free boundary minimal disks in convex regions},
  author={Sarnataro, Lorenzo and Stryker, Douglas and Wang, Zhichao and Zhou, Xin},
  journal={arXiv preprint arXiv:2606.02396},
  year={2026}
}

@article {FraserLi2014,
    AUTHOR = {Fraser, Ailana and Li, Martin Man-chun},
     TITLE = {Compactness of the space of embedded minimal surfaces with
              free boundary in three-manifolds with nonnegative {R}icci
              curvature and convex boundary},
   JOURNAL = {J. Differential Geom.},
  FJOURNAL = {Journal of Differential Geometry},
    VOLUME = {96},
      YEAR = {2014},
    NUMBER = {2},
     PAGES = {183--200},
      ISSN = {0022-040X,1945-743X},
   MRCLASS = {53C42},
  MRNUMBER = {3178438},
MRREVIEWER = {Claudio\ Gorodski},
       URL = {http://projecteuclid.org/euclid.jdg/1393424916},
}

@article {White1987,
    AUTHOR = {White, B.},
     TITLE = {Curvature estimates and compactness theorems in
              {$3$}-manifolds for surfaces that are stationary for
              parametric elliptic functionals},
   JOURNAL = {Invent. Math.},
  FJOURNAL = {Inventiones Mathematicae},
    VOLUME = {88},
      YEAR = {1987},
    NUMBER = {2},
     PAGES = {243--256},
      ISSN = {0020-9910,1432-1297},
   MRCLASS = {58E12 (49F10 53C20 58D10)},
  MRNUMBER = {880951},
MRREVIEWER = {Helmut\ Kaul},
       DOI = {10.1007/BF01388908},
       URL = {https://doi.org/10.1007/BF01388908},
}

@article{ambrozio2018compactness,
  title={Compactness analysis for free boundary minimal hypersurfaces},
  author={Ambrozio, Lucas and Carlotto, Alessandro and Sharp, Ben},
  journal={Calculus of Variations and Partial Differential Equations},
  volume={57},
  number={1},
  pages={22},
  year={2018},
  publisher={Springer}
}

@article{white1991space,
  title={The space of minimal submanifolds for varying Riemannian metrics},
  author={White, Brian},
  journal={Indiana University Mathematics Journal},
  pages={161--200},
  year={1991},
  publisher={JSTOR}
}

@book {John2019,
    AUTHOR = {Ratcliffe, John G.},
     TITLE = {Foundations of hyperbolic manifolds},
    SERIES = {Graduate Texts in Mathematics},
    VOLUME = {149},
   EDITION = {Third},
 PUBLISHER = {Springer, Cham},
      YEAR = {[2019] \copyright 2019},
     PAGES = {xii+800},
      ISBN = {978-3-030-31597-9; 978-3-030-31596-2},
   MRCLASS = {57M50 (20H10 30F40 57K32)},
  MRNUMBER = {4221225},
       DOI = {10.1007/978-3-030-31597-9},
       URL = {https://doi-org.proxy.libraries.rutgers.edu/10.1007/978-3-030-31597-9},
}

@article{almgren1962homotopy,
  title={The homotopy groups of the integral cycle groups},
  author={Almgren, F},
  journal={Topology},
  volume={1},
  issue={4},
  pages={257-299},
  year={1962}
}

@article {Haslhofer2019,
    AUTHOR = {Haslhofer, Robert and Ketover, Daniel},
     TITLE = {Minimal 2-spheres in 3-spheres},
   JOURNAL = {Duke Math. J.},
  FJOURNAL = {Duke Mathematical Journal},
    VOLUME = {168},
      YEAR = {2019},
    NUMBER = {10},
     PAGES = {1929--1975},
      ISSN = {0012-7094,1547-7398},
   MRCLASS = {49Q05 (49J35 53E10 58E12)},
  MRNUMBER = {3983295},
MRREVIEWER = {Hung\ Thanh\ Tran},
       DOI = {10.1215/00127094-2019-0009},
       URL = {https://doi.org/10.1215/00127094-2019-0009},
}

@article{chodosh2017minimal,
  title={Minimal hypersurfaces with bounded index},
  author={Chodosh, Otis and Ketover, Daniel and Maximo, Davi},
  journal={Inventiones mathematicae},
  volume={209},
  number={3},
  pages={617--664},
  year={2017},
  publisher={Springer}
}

@article{meeks1982embedded,
  title={Embedded minimal surfaces, exotic spheres, and manifolds with positive Ricci curvature},
  author={Meeks, William and Simon, Leon and Yau, Shing-Tung},
  journal={Annals of Mathematics},
  volume={116},
  number={3},
  pages={621--659},
  year={1982},
  publisher={JSTOR}
}

@article {DeLellis2013,
    AUTHOR = {De Lellis, Camillo and Tasnady, Dominik},
     TITLE = {The existence of embedded minimal hypersurfaces},
   JOURNAL = {J. Differential Geom.},
  FJOURNAL = {Journal of Differential Geometry},
    VOLUME = {95},
      YEAR = {2013},
    NUMBER = {3},
     PAGES = {355--388},
      ISSN = {0022-040X,1945-743X},
   MRCLASS = {53C42 (53A10)},
  MRNUMBER = {3128988},
MRREVIEWER = {Fei-Tsen\ Liang},
       URL = {http://projecteuclid.org.proxy.libraries.rutgers.edu/euclid.jdg/1381931732},
}

@article{poincare1905lignes,
  title={Sur les lignes géodésiques des surfaces convexes},
  author={Poincaré, Henri},
  journal={Transactions of the American Mathematical Society},
  pages={237--274},
  year={1905},
  publisher={JSTOR}
}

@article{birkhoff1917dynamical,
  title={Dynamical systems with two degrees of freedom},
  author={Birkhoff, George D},
  journal={Proceedings of the National Academy of Sciences},
  volume={3},
  number={4},
  pages={314--316},
  year={1917}
}

@article{lyusternik1947topological,
  title={Topological methods in variational problems and their application to the differential geometry of surfaces},
  author={Lyusternik, Lazar Aronovich and Shnirel'man, Lev Genrikhovich},
  journal={Uspekhi Matematicheskikh Nauk},
  volume={2},
  number={1},
  pages={166--217},
  year={1947},
  publisher={Russian Academy of Sciences, Steklov Mathematical Institute of Russian~…}
}

@article{grayson1989shortening,
  title={Shortening embedded curves},
  author={Grayson, Matthew A},
  journal={Annals of Mathematics},
  volume={129},
  number={1},
  pages={71--111},
  year={1989},
  publisher={JSTOR}
}

@article{yau1982problem,
  title={Problem section in seminar on differential geometry, edited by st Yau},
  author={Yau, ST},
  journal={Annals of Math. Studies, Princeton, NJ},
  year={1982}
}

@article{smith1983existence,
  title={On the existence of embedded minimal 2-spheres in the 3-sphere, endowed with an arbitrary metric},
  author={Smith, Francis R},
  journal={Bulletin of the Australian Mathematical Society},
  volume={28},
  number={1},
  pages={159--160},
  year={1983},
  publisher={Cambridge University Press}
}

@article{almgren1966some,
  title={Some interior regularity theorems for minimal surfaces and an extension of Bernstein's theorem},
  author={Almgren, Frederick J},
  journal={Annals of Mathematics},
  volume={84},
  number={2},
  pages={277--292},
  year={1966},
  publisher={JSTOR}
}

@article{lawson1970unknottedness,
  title={The unknottedness of minimal embeddings},
  author={Lawson Jr, H Blaine},
  journal={Inventiones mathematicae},
  volume={11},
  number={3},
  pages={183--187},
  year={1970},
  publisher={Springer-Verlag Berlin/Heidelberg}
}

@article{willmore1965note,
  title={Note on embedded surfaces},
  author={Willmore, Thomas J},
  journal={An. Sti. Univ.“Al. I. Cuza” Iasi Sect. I a Mat.(NS) B},
  volume={11},
  number={493-496},
  pages={20},
  year={1965}
}

@article{smith2019morse,
  title={The Morse index of the critical catenoid},
  author={Smith, Graham and Zhou, Detang},
  journal={Geometriae Dedicata},
  volume={201},
  number={1},
  pages={13--19},
  year={2019},
  publisher={Springer}
}

@article{ketover2022flipping,
  title={Flipping Heegaard splittings and minimal surfaces},
  author={Ketover, Daniel},
  journal={arXiv preprint arXiv:2211.03745},
  year={2022}
}

@article{karpukhin2024embedded,
  title={Embedded minimal surfaces in $\mathbb{S}^{3}$ and $\mathbb{B}^{3}$ via equivariant eigenvalue optimization},
  author={Karpukhin, Mikhail and Kusner, Robert and McGrath, Peter and Stern, Daniel},
  journal={arXiv preprint arXiv:2402.13121},
  year={2024}
}

@article{kapouleas2020index,
  title={The index and nullity of the Lawson surfaces $\xi_{g,1}$},
  author={Kapouleas, Nikolaos and Wiygul, David},
  journal={Cambridge Journal of Mathematics},
  volume={8},
  number={2},
  pages={363--405},
  year={2020},
  publisher={International Press of Boston, Inc. Somerville, MA 02143, USA}
}

@article{marques2017existence,
  title={Existence of infinitely many minimal hypersurfaces in positive Ricci curvature},
  author={Marques, Fernando C and Neves, Andr{\'e}},
  journal={Inventiones mathematicae},
  volume={209},
  number={2},
  pages={577--616},
  year={2017},
  publisher={Springer}
}

@article{song2023existence,
  title={Existence of infinitely many minimal hypersurfaces in closed manifolds},
  author={Song, Antoine},
  journal={Annals of Mathematics},
  volume={197},
  number={3},
  pages={859--895},
  year={2023},
  publisher={Department of Mathematiccs, Princeton University Princeton, New Jersey, USA}
}

@article{li2015general,
  title={A general existence theorem for embedded minimal surfaces with free boundary},
  author={Li, Martin Man-chun},
  journal={Communications on Pure and Applied Mathematics},
  volume={68},
  number={2},
  pages={286--331},
  year={2015},
  publisher={Wiley Online Library}
}

@article{li2021min,
  title={Min-max theory for free boundary minimal hypersurfaces, I: Regularity theory},
  author={Li, Martin Man-Chun and Zhou, Xin},
  journal={Journal of Differential Geometry},
  volume={118},
  number={3},
  pages={487--553},
  year={2021},
  publisher={Lehigh University}
}

@inproceedings{gromov2006dimension,
  title={Dimension, non-linear spectra and width},
  author={Gromov, Misha},
  booktitle={Geometric Aspects of Functional Analysis: Israel Seminar (GAFA) 1986--87},
  pages={132--184},
  year={2006},
  organization={Springer}
}

@techreport{gromov2002isoperimetry,
  title={Isoperimetry of waists and concentration of maps},
  author={Gromov, Mikhail},
  year={2002},
  institution={SIS-2003-246}
}

@article{gromov2007singularities,
  title={Singularities, expanders and topology of maps Part I: Homology versus Volume in the Volume in the Spaces of Cycles},
  author={Gromov, Mikhail},
  journal={Geometric and Functional Analysis},
  year={2007}
}

@article{guth2009minimax,
  title={Minimax problems related to cup powers and Steenrod squares},
  author={Guth, Larry},
  journal={Geometric And Functional Analysis},
  volume={18},
  number={6},
  pages={1917--1987},
  year={2009},
  publisher={Springer}
}

@article{chodosh2023p,
  title={The p-widths of a surface},
  author={Chodosh, Otis and Mantoulidis, Christos},
  journal={Publications math{\'e}matiques de l'IH{\'E}S},
  volume={137},
  number={1},
  pages={245--342},
  year={2023},
  publisher={Springer}
}

@article{chu2023strong,
  title={A strong multiplicity one theorem in min-max theory},
  author={Chu, Adrian Chun-Pong and Li, Yangyang},
  journal={Advances in Mathematics 482 : 110623},
  year={2025}
}

@article{chu2023free,
  title={A free boundary minimal surface via a 6-sweepout},
  author={Chu, Adrian Chun-Pong},
  journal={The Journal of Geometric Analysis},
  volume={33},
  number={7},
  pages={230},
  year={2023},
  publisher={Springer}
}

@article{donato2020first,
  title={The first $ p $-widths of the unit disk},
  author={Donato, Sidney},
  journal={Journal of Geometric Analysis, 32(6), 177},
  year={2022}
}

@article{donato2024first,
  title={The first width of non-negatively curved surfaces with convex boundary},
  author={Donato, Sidney and Montezuma, Rafael},
  journal={The Journal of Geometric Analysis},
  volume={34},
  number={2},
  pages={60},
  year={2024},
  publisher={Springer}
}

@article{marx2026p,
  title={The p-widths of the Hemisphere},
  author={Marx-Kuo, Jared},
  journal={arXiv preprint arXiv:2603.17734},
  year={2026}
}

@article{wang2026existence,
  title={Existence of two embedded minimal spheres in $\mathbb{S}^{3}$ with an arbitrary metric},
  author={Wang, Zhichao and Zhou, Xin},
  journal={arXiv preprint arXiv:2607.08631},
  year={2026}
}

@article{sarnataro2023optimal,
  title={Optimal regularity for minimizers of the prescribed mean curvature functional over isotopies},
  author={Sarnataro, Lorenzo and Stryker, Douglas},
  journal={Cambridge Journal of Mathematics},
  volume={13},
  number={3},
  pages={609-706},
  year={2026}
}

@article{white1989every,
  title={Every three-sphere of positive Ricci curvature contains a minimal embedded torus},
  author={White, Brian},
  journal={Bull. Amer. Math. Soc.},
  volume={21},
  number={1},
  year={1989}
}

@article{chu2025min,
  title={Min-max theory and minimal surfaces with prescribed genus},
  author={Chu, Adrian Chun-Pong and Li, Yangyang and Wang, Zhihan},
  journal={arXiv preprint arXiv:2507.23239},
  year={2025}
}

@article{chu2026enumerative,
  title={An enumerative min-max theorem for minimal surfaces},
  author={Chu, Adrian Chun-Pong and Li, Yangyang and Wang, Zhihan},
  journal={arXiv preprint arXiv:2601.01736},
  year={2026}
}

@article{chu2025minimal,
  title={Minimal surfaces with arbitrary genus in 3-spheres of positive Ricci curvature},
  author={Chu, Adrian Chun-Pong},
  journal={arXiv preprint arXiv:2508.06019},
  year={2025}
}

@article{struwe1984free,
  title={On a free boundary problem for minimal surfaces},
  author={Struwe, Michael},
  journal={Inventiones mathematicae},
  volume={75},
  number={3},
  pages={547--560},
  year={1984},
  publisher={Springer}
}

@inproceedings{gruter1986embedded,
  title={On embedded minimal disks in convex bodies},
  author={Gr{\"u}ter, Michael and Jost, J{\"u}rgen},
  booktitle={Annales de l'Institut Henri Poincar{\'e} C, Analyse non lin{\'e}aire},
  volume={3},
  number={5},
  pages={345--390},
  year={1986},
  organization={Elsevier}
}

@article{jost1986existence,
  title={Existence results for embedded minimal surfaces of controlled topological type, II},
  author={Jost, J{\"u}rgen},
  journal={Annali della Scuola Normale Superiore di Pisa-Classe di Scienze},
  volume={13},
  number={3},
  pages={401--426},
  year={1986}
}

@article{haslhofer2025free,
  title={Free boundary minimal disks in convex balls},
  author={Haslhofer, Robert and Ketover, Daniel},
  journal={Journal f{\"u}r die reine und angewandte Mathematik (Crelles Journal)},
  volume={2025},
  number={828},
  pages={307--326},
  year={2025},
  publisher={De Gruyter}
}

@article{maximo2017free,
  title={Free boundary minimal annuli in convex three-manifolds},
  author={Maximo, Davi and Nunes, Ivaldo and Smith, Graham},
  journal={Journal of Differential Geometry},
  volume={106},
  number={1},
  pages={139--186},
  year={2017},
  publisher={Lehigh University}
}

@article{ache2016metrics,
  title={Metrics with nonnegative Ricci curvature on convex three-manifolds},
  author={Ach{\'e}, Antonio and Maximo, Davi and Wu, Haotian},
  journal={Geometry \& Topology},
  volume={20},
  number={5},
  pages={2905--2922},
  year={2016},
  publisher={Mathematical Sciences Publishers}
}

@article{zhou2016free,
  title={On the free boundary min--max geodesics},
  author={Zhou, Xin},
  journal={International Mathematics Research Notices},
  volume={2016},
  number={5},
  pages={1447--1466},
  year={2016},
  publisher={Oxford University Press}
}

@article{ko2026existence,
  title={Existence and Morse Index of two free boundary embedded geodesics on Riemannian 2-disks with convex boundary},
  author={Ko, Dongyeong},
  journal={American Journal of Mathematics, To appear},
  year={2026}
}

@article{ko2024morse,
  title={Morse index bound of simple closed geodesics on 2-spheres and strong Morse inequalities},
  author={Ko, Dongyeong},
  journal={Journal f{\"u}r die reine und angewandte Mathematik (Crelles Journal)},
  volume={2024},
  number={817},
  pages={33--66},
  year={2024},
  publisher={De Gruyter}
}

@article{li2024existence,
  title={Existence of embedded minimal tori in three-spheres with positive Ricci curvature},
  author={Li, Xingzhe and Wang, Zhichao},
  journal={arXiv preprint arXiv:2409.10391},
  year={2024}
}

@article{li2025minimal,
  title={Minimal surfaces with low genus in lens spaces},
  author={Li, Xingzhe and Wang, Tongrui and Yao, Xuan},
  journal={Journal f{\"u}r die reine und angewandte Mathematik (Crelles Journal)},
  volume={2025},
  number={828},
  pages={175--218},
  year={2025},
  publisher={De Gruyter}
}

@article{ko2025regularity,
  title={Regularity of Cohomogeneity two equivariant isotopy minimization problems and minimal hypersurfaces with large first Betti number on spheres},
  author={Ko, Dongyeong},
  journal={arXiv preprint arXiv:2512.12322},
  year={2025}
}

@article{wang2026embedded,
  title={Embedded minimal ${S}^1$-bundles in $\mathbb{S}^{4}$},
  author={Wang, Tongrui},
  journal={arXiv preprint arXiv:2606.31091},
  year={2026}
}

@article {hatcher1983proof,
    AUTHOR = {Hatcher, Allen E.},
     TITLE = {A proof of the {S}male conjecture, {${\rm Diff}(S\sp{3})\simeq
              {\rm O}(4)$}},
   JOURNAL = {Ann. of Math. (2)},
  FJOURNAL = {Annals of Mathematics. Second Series},
    VOLUME = {117},
      YEAR = {1983},
    NUMBER = {3},
     PAGES = {553--607},
      ISSN = {0003-486X,1939-8980},
   MRCLASS = {57M99 (57S05)},
  MRNUMBER = {701256},
MRREVIEWER = {R.\ C.\ Kirby},
       DOI = {10.2307/2007035},
       URL = {https://doi.org/10.2307/2007035},
}

@article{urbano1990minimal,
  title={Minimal surfaces with low index in the three-dimensional sphere},
  author={Urbano, Francisco},
  journal={Proceedings of the American Mathematical Society},
  pages={989--992},
  year={1990},
  publisher={JSTOR}
}

@article {Brendle2012,
    AUTHOR = {Brendle, Simon},
     TITLE = {Embedded minimal tori in {$S^3$} and the {L}awson conjecture},
   JOURNAL = {Acta Math.},
  FJOURNAL = {Acta Mathematica},
    VOLUME = {211},
      YEAR = {2013},
    NUMBER = {2},
     PAGES = {177--190},
      ISSN = {0001-5962,1871-2509},
   MRCLASS = {53A10 (49Q05 53C42)},
  MRNUMBER = {3143888},
MRREVIEWER = {Jo\~ao\ Lucas Marques Barbosa},
       DOI = {10.1007/s11511-013-0101-2},
       URL = {https://doi-org.proxy.libraries.rutgers.edu/10.1007/s11511-013-0101-2},
}

@inproceedings {LiMartin2020,
    AUTHOR = {Li, Martin Man-chun},
     TITLE = {Free boundary minimal surfaces in the unit ball: recent
              advances and open questions},
 BOOKTITLE = {Proceedings of the {I}nternational {C}onsortium of {C}hinese
              {M}athematicians 2017},
     PAGES = {401--435},
 PUBLISHER = {Int. Press, Boston, MA},
      YEAR = {[2020] \copyright 2020},
      ISBN = {978-1-57146-392-0},
   MRCLASS = {53A10},
  MRNUMBER = {4251121},
}

@book{ritore2023isoperimetric,
  title={Isoperimetric inequalities in Riemannian manifolds},
  author={Ritorè, Manuel},
  year={2023},
  publisher={Springer}
}

@article{chu2025unknottedness,
  title={Unknottedness of free boundary minimal surfaces and self-shrinkers},
  author={Chu, Sabine and Franz, Giada},
  journal={Calculus of Variations and Partial Differential Equations},
  volume={64},
  number={8},
  pages={237},
  year={2025},
  publisher={Springer}
}

@article{chodosh2020minimal,
  title={Minimal surfaces and the Allen--Cahn equation on 3-manifolds: index, multiplicity, and curvature estimates},
  author={Chodosh, Otis and Mantoulidis, Christos},
  journal={Annals of Mathematics},
  volume={191},
  number={1},
  pages={213--328},
  year={2020},
  publisher={JSTOR}
}

@article{chu2025minimals,
  title={Minimal surface doublings and electrostatics for Schrödinger operators},
  author={Chu, Adrian Chun-Pong and Stern, Daniel},
  journal={arXiv preprint arXiv:2509.18630},
  year={2025}
}

@article {Medvedev2023,
    AUTHOR = {Medvedev, Vladimir},
     TITLE = {On the index of the critical {M}\"obius band in {$\Bbb{B}^4$}},
   JOURNAL = {J. Geom. Anal.},
  FJOURNAL = {Journal of Geometric Analysis},
    VOLUME = {33},
      YEAR = {2023},
    NUMBER = {3},
     PAGES = {Paper No. 93, 29},
      ISSN = {1050-6926,1559-002X},
   MRCLASS = {53A10 (49Q05 58E12)},
  MRNUMBER = {4531070},
MRREVIEWER = {Shujie\ Zhai},
       DOI = {10.1007/s12220-022-01069-w},
       URL = {https://doi-org.proxy.libraries.rutgers.edu/10.1007/s12220-022-01069-w},
}

@article{schulz2024equivariant,
  title={Equivariant free boundary minimal discs and annuli in ellipsoids},
  author={Schulz, Mario B},
  journal={arXiv preprint arXiv:2406.13465},
  year={2024}
}
\end{document}